\documentclass[a4paper,11pt,oneside]{article}
\usepackage[margin=3cm]{geometry}
\usepackage{aicescover}

\usepackage{hyperref}

\usepackage{latexsym,amsmath,amsthm,amssymb,calc}
\usepackage[utf8]{inputenc}
\usepackage{dsfont}
\usepackage{color}
\usepackage{graphicx}
\usepackage{mdwlist}
\usepackage{enumerate}
\usepackage{setspace}
\usepackage{multicol}
\usepackage{multirow}
\usepackage{cite}

\usepackage[section]{algorithm}
\usepackage{algorithmicx}
\usepackage{algpseudocode}
\algrenewcommand\algorithmicrequire{\makebox[32pt][l]{\textrm{input}}}
\algrenewcommand\algorithmicensure{\makebox[32pt][l]{\textrm{output}}}
\algrenewcommand\algorithmicfunction{\textrm{function}}
\algrenewcommand\algorithmicwhile{\textrm{while}}
\algrenewcommand\algorithmicdo{}
\algrenewcommand\algorithmicend{\textrm{end}}
\algrenewcommand\algorithmicforall{\textrm{for all}}
\algrenewcommand\algorithmicfor{\textrm{for}}
\algrenewcommand\algorithmicrepeat{\textrm{repeat}}
\algrenewcommand\algorithmicuntil{\textrm{until}}

\DeclareFontFamily{U}{mathx}{\hyphenchar\font45}
\DeclareFontShape{U}{mathx}{m}{n}{
      <5> <6> <7> <8> <9> <10>
      <10.95> <12> <14.4> <17.28> <20.74> <24.88>
      mathx10
      }{}
\DeclareSymbolFont{mathx}{U}{mathx}{m}{n}
\DeclareMathSymbol{\bigtimes}{1}{mathx}{"91}

\theoremstyle{plain}
\newtheorem{theorem}{Theorem}
\newtheorem{lemma}{Lemma}
\newtheorem{proposition}{Proposition}
\newtheorem{corollary}{Corollary}
\newtheorem{remark}{Remark}
\newtheorem{assumptions}{Assumptions}
\theoremstyle{definition}
\newtheorem{definition}{Definition}
\newtheorem{example}{Example}

\newcommand{\N}{\mathds{N}}
\newcommand{\R}{\mathds{R}}

\DeclareMathOperator{\linspan}{span}

\DeclareMathOperator{\supp}{supp}
\DeclareMathOperator{\dist}{dist}

\DeclareMathOperator*{\argmin}{arg\,min}

\DeclareMathOperator{\range}{range}
\DeclareMathOperator{\ops}{ops}

\newcommand{\id}{{\rm I}}

\newcommand{\dif}{\,{\rm d}}

\newcommand{\spl}[1]{{\rm \ell}_{#1}}

\newcommand{\spH}[1]{{\rm H}^{#1}}

\newcommand{\spL}[1]{{\rm L}_{#1}}

\newcommand{\zinterval}[2]{\{{#1},\ldots,{#2}\}}

\newcommand{\Ocal}{{\mathcal{O}}}
\newcommand{\Tcal}{{\mathcal{T}}}
\newcommand{\Hcal}{{\mathcal{H}}}
\newcommand{\Rcal}{{\mathcal{R}}}

\newcommand{\Ical}{{\mathcal{I}}}

\newcommand{\Acal}{{\mathcal{A}}}

\newcommand{\Fcal}{{\mathcal{F}}}

\newcommand{\Rank}{\Rcal}

\newcommand{\err}{{\lambda}}
\newcommand{\cerr}{{\mu}}

\DeclareMathOperator{\rank}{rank}

\DeclareMathOperator{\apply}{\textsc{apply}}
\DeclareMathOperator{\coarsen}{\textsc{coarsen}}
\DeclareMathOperator{\rhs}{\textsc{rhs}}
\DeclareMathOperator{\recompress}{\textsc{recompress}}
\DeclareMathOperator{\solve}{\textsc{solve}}

\providecommand{\abs}[1]{\lvert#1\rvert}
\providecommand{\bigabs}[1]{\bigl\lvert#1\bigr\rvert}

\providecommand{\norm}[1]{\lVert#1\rVert}
\providecommand{\bignorm}[1]{\bigl\lVert#1\bigr\rVert}
\providecommand{\Bignorm}[1]{\Bigl\lVert#1\Bigr\rVert}

\newcommand{\UU}{\mathbb{U}}
\newcommand{\VV}{\mathbb{V}}
\def\bu{{\bf u}}
\def\bU{{\bf U}}
\def\bV{{\bf V}}

\newcommand{\bv}{{\bf v}}
\newcommand{\bw}{{\bf w}}
\newcommand{\ba}{{\bf a}}
\def\e2{\spl{2}(\nabla^d)}
\def\cA{{\cal A}}
\def\ga{\gamma}
\def\garatio{{\rho_\ga}}

\newcommand{\As}{{\Acal^s}}

\newcommand{\AH}[1]{{\Acal_\Hcal({#1})}}
\newcommand{\AF}[1]{{\Acal_\Fcal({#1})}}

\newcommand{\kk}[1]{{\mathsf{#1}}}
\newcommand{\rr}[1]{{\mathsf{#1}}}

\newcommand{\KK}[1]{{\mathsf{K}_m}}

\newcommand{\Proj}{\operatorname{P}}
\newcommand{\rbest}{\operatorname{\bar{r}}}
\newcommand{\Cbest}{{\operatorname{\bar{C}}}}
\newcommand{\Psvd}[2]{\operatorname{P}_{\UU({#1}),{#2}}}
\newcommand{\hatPsvd}[1]{\operatorname{\hat P}_{#1}}
\newcommand{\Pbest}[2]{\operatorname{P}_{\bar\UU({#1},{#2})}}
\newcommand{\rsvd}{\operatorname{r}}
\newcommand{\Cctr}{\operatorname{C}}
\newcommand{\hatCctr}[1]{\operatorname{\hat C}_{#1}}
\newcommand{\Restr}[1]{\operatorname{R}_{#1}} 

\newcommand{\constsvd}{\kappa_{\rm P}}
\newcommand{\constcrs}{\kappa_{\rm C}}

\newcommand\eref[1]{(\ref{#1})}
\newcommand{\beqn}{\begin{equation}}
\newcommand{\eeqn}{\end{equation}}

\newcommand{\bA}{\mathbf{A}}
\newcommand{\bbf}{\mathbf{f}}
\newcommand{\cL}{{\cal L}}
\newcommand{\cD}{{\cal D}}

\newcommand{\cN}{{\cal N}}
\newcommand{\bB}{\mathbf{B}}

\newcommand{\hdimtree}[1]{\mathcal{D}_{#1}}
\newcommand{\hroot}[1]{{0_{#1}}}
\newcommand{\leaf}[1]{\cL(\hdimtree{#1})}
\newcommand{\nonleaf}[1]{\cN(\hdimtree{#1})}
\newcommand{\tucker}[1]{\mathcal{T}({#1})}
\newcommand{\htucker}[1]{\mathcal{H}({#1})}

\newcommand{\leftchild}{{{\rm c}_1}}
\newcommand{\rightchild}{{{\rm c}_2}}
\newcommand{\child}[1]{{{\rm c}_{#1}}}
\newcommand{\hsum}[1]{\mathrm{\Sigma}_{#1}}

\newcommand{\dd}{{\rm rank}}

\title{Adaptive Near-Optimal Rank Tensor Approximation for High-Dimensional Operator Equations}

\makeatletter
\let\@fnsymbol\@arabic
\makeatother

\author{Markus Bachmayr\thanks{IGPM, RWTH Aachen,
Germany, email: bachmayr@igpm.rwth-aachen.de} 
~and~ Wolfgang Dahmen\thanks{IGPM and AICES, RWTH Aachen, Germany, email: dahmen@igpm.rwth-aachen.de}}

\date{May 1, 2013; revised October 18, 2013}

\begin{document}

\aicescoverpage

\maketitle

\begin{abstract}
We consider a framework for the construction of iterative schemes for
operator equations that combine low-rank approximation in tensor formats and 
adaptive approximation in a basis. Under fairly general
assumptions, we obtain a rigorous convergence analysis, where all
parameters required for the execution of the methods depend only on
the underlying infinite-dimensional problem, but not on a concrete
discretization.  Under certain assumptions on the rates for the involved
low-rank approximations and basis expansions, we can also give bounds
on the computational complexity of the iteration as a function of the
prescribed target error.
Our theoretical findings are illustrated and supported by computational experiments.
These demonstrate that problems in very high dimensions can be treated
with controlled solution accuracy.


\end{abstract}

\section{Introduction}\label{sect:intro}

\subsection{Motivation}\label{ssect:motiv}
Any attempt to recover or approximate a function of a large number of variables with the aid of classical
low-dimensional techniques is inevitably  
impeded by the {\em curse of dimensionality}. This means that, when only
assuming classical smoothness (e.g. in terms of Sobolev or Besov regularity) of order $s>0$, the necessary computational
work needed  to realize a desired target accuracy $\varepsilon$ in $d$
dimensions scales like $\varepsilon^{-d/s}$, i.e., one faces  an exponential increase in the spatial dimension $d$.
This can be ameliorated by {\em dimension-dependent} smoothness measures. In many
high-dimensional problems of interest, the approximand has bounded high-order
mixed derivatives, which under suitable assumptions can be used to construct
sparse grid-type approximations where the computational work scales like $C_d
\varepsilon^{-1/s}$. Under such regularity assumptions, one can thus obtain a
convergence rate independent of $d$. In general, however, the constant $C_d$
will still grow exponentially in $d$.  This has been shown to hold even under
extremely restrictive smoothness assumptions in \cite{Novak:09}, and has been
observed numerically in a relatively simple but realistic example in
\cite{Dijkema:09}.

Hence, in contrast to the low-dimensional regime, regularity is no longer a sufficient structural property that ensures computational feasibility, and further low-dimensional structure of the sought high-dimensional object is required.
Such a structure could be the dependence of the function on a much smaller (unknown) number of variables, see e.g. \cite{DeVore:11}. 
It could also mean {\em sparsity} with respect to {\em some} (a priori) unknown dictionary. In particular, dictionaries 
comprized of {\em rank-one tensors} $g(x_1,\ldots, x_d)=g_1(x_1)\cdots g_d(x_d)=: (g_1\otimes \cdots \otimes g_d)(x)$
open very promising perspectives and have recently attracted substantial
attention.

As a simple example consider $g(x)= \bigotimes_{i=1}^d g_i(x_i)$ on the unit
cube $\Omega = [0,1]^d$, where the $g_i$ are sufficiently smooth. 
Employing  for each factor $g_i$  a standard spline
approximation of order $s$ with $n$ knots yields an $L_\infty$-accuracy of order $n^{-s}$,
which gives rise to an overall accuracy of the order of $d n^{-s}$ 
at the expense of $dn=:N$ degrees of freedom.
Hence, assuming that $\norm{g}_{\infty}$ does not depend on $d$, 
an accuracy $\varepsilon$ requires 
\begin{equation}
\label{epsN}
{
N=N(\varepsilon,d)\sim  d^{\frac{1+s}{s}} \varepsilon^{-1/s}  }
\end{equation}
degrees of freedom.
  In contrast, it would take
the order of $N= n^d$ degrees of freedom to realize an accuracy of order $n^{-s} = N^{-d/s}$ when using 
a standard tensor product spline approximation, which means that in this case $N(\varepsilon,d)\sim \varepsilon^{-d/s}$.  
Thus, while the first approximation -- using a {\em nonlinear parametrization}
of a reference basis -- breaks the curse of dimensionality, 
the second one obviously does not.

Of course, $u$ being a simple tensor is in general an unrealistic assumption, but the curse of dimensionality can still be significantly mitigated
when $f$ is {\em well approximable} by relatively short sums of rank-one tensors.
By this we mean that for some norm $\|\cdot\|$ we have
\begin{equation}
\label{canonical}
\Bignorm{ u - \sum_{j=1}^{r(\varepsilon)} g_{1,j}\otimes \cdots \otimes g_{j,d}}
 \leq \varepsilon
\end{equation}
where the rank $r(\varepsilon)$ grows only moderately as $\varepsilon$ decreases.
In our initial example, in these terms we had $r(\varepsilon)=1$ for all $\varepsilon >0$. 
Assuming that all the factors $g_{j,i}$ in the above approximation are sufficiently smooth,
the count (\ref{epsN}) applied to each summand with target accuracy $\varepsilon/r$ shows that
now at most
\begin{equation}
\label{epsNr}
{
N(\varepsilon,d,r) \lesssim r^{1+\frac 1s}d^{\frac{1+s}{s}} \varepsilon^{-\frac{1}{s}} }
\end{equation}
degrees of freedom are required, which is still acceptable. 
This is clearly a very crude reasoning because it does not take a possible additional
decay in the rank-one summands into account. 

This argument, however, already indicates that good approximability in the sense of \eqref{canonical}
is not governed by classical regularity assumptions.
Instead, the key is to exploit an approximate global low-rank structure of $u$.
This leads to a highly nonlinear approximation problem, where one aims to identify suitable
lower-dimensional tensor factors, which can be interpreted as a $u$-dependent dictionary.

This discussion, although admittedly somewhat oversimplified, immediately raises several questions which we will briefly discuss as they guide subsequent developments.

{\em Format of approximation:} The hope that $r(\varepsilon)$ in (\ref{canonical}) can be rather small
is based on the fact that the rank-one tensors are allowed to ``optimally adapt'' to the approximand $u$.
The format of the approximation used in \eref{canonical} is sometimes called {\em canonical} since it is a formal
direct generalization of classical Hilbert Schmidt expansions for $d=2$. However, a closer look reveals a number
of well-known pitfalls.
In fact, they are already encountered in the {\em discrete} case. The collection of sums of ranks one tensors
of a given length is not closed, and the best approximation problem is not well-posed, see e.g.~\cite{Silva:08}.
There appears to be no reliable computational strategy that can be proven to yield near-minimal rank approximations
for a given target accuracy in this format.
In this work, we therefore employ different tensor formats that allow us to obtain provably near-minimal rank
approximations, as explained later.

{\em A two-layered problem:}
Given a suitable tensor format, even if a best tensor approximation is known in the 
infinite-dimensional setting of the {continuous} problem, the resulting lower-dimensional 
factors still need to be approximated.
Since finding these factors is part of the solution process, the determination of
efficient discretizations for these factors will need to be intertwined with
the process of finding low-rank expansions.
We have chosen here to organize this process through selecting low-dimensional orthonormal
wavelet bases for the tensor factors. However, other types of basis expansions would be
conceivable as well.

The issue of the total complexity of tensor approximations, taking the approximation of the involved
lower-dimensional factors into account, is addressed in \cite{Griebel:11,Schneider:13}.

\subsection{Conceptual Preview}\label{ssect:preview}

The problem of finding a suitable format of tensor approximations has been
extensively studied  in the literature over that past years, however, mainly in the discrete or finite-dimensional setting, see e.g.~\cite{Kolda:09,Hackbusch:09-1,Oseledets:09,Grasedyck:10,Oseledets:11}.
Some further aspects in a function space setting have been addressed
e.g.~in \cite{Uschmajew:10,Falco:10,Uschmajew:11}.
For an overview and further references we also refer to \cite{Hackbusch:12} and the recent survey
\cite{Grasedyck:13}.
 A central question in these works is:
 given a tensor, how can one in a stable manner obtain low-rank
 approximations, and how accurate are they when compared with best tensor 
 approximations in the respective format?
 
We shall heavily draw on these findings in the present paper, but under the following somewhat different perspectives.
First of all, we are interested in the {\em continuous infinite-dimensional} setting, i.e.,  in sparse tensor approximations of a 
{\em function} which is a priori not given in any finite tensor format but which one may expect to be 
well approximable by simple tensors in a way to be made precise later. We shall not discuss here the question under which concrete conditions this is actually the case. 
Moreover, the objects to be recovered are not given explicitly but only {\em implicitly} as a solution
to an operator equation
\beqn
\label{opeq}
Au =f,
\eeqn
where $A: V\to V'$ is an isomorphism of some Hilbert space $V$ onto its dual $V'$. One may think of $V$, in
the simplest instance, as a high-dimensional $\spL{2}$ space, or as a Sobolev space. More generally, as in the
context of parametric diffusion problems, $V$ could be a tensor product of a Sobolev space and an $\spL{2}$ space. Accordingly, we shall always assume that we have a Gelfand triplet
\beqn
\label{Gelfand}
V\subset  H \equiv H' \subset V',
\eeqn
in the sense of dense continuous embeddings, where we assume that $H$ is a tensor product Hilbert space, that is,
\begin{equation}\label{eq:htensorspace}
 H = H_1 \otimes \cdots \otimes H_d 
\end{equation}
with lower-dimensional Hilbert spaces $H_i$. A typical example would be $H = \spL{2}(\Omega^d) = \spL{2}(\Omega)\otimes \cdots\otimes \spL{2}(\Omega)$ for a domain $\Omega$ of small spatial dimension.

The main contribution of this work is to put forward a strategy that addresses the main obstacles 
identified above and results in an algorithm which, under mild assumptions, can be rigorously proven to provide
for any target accuracy $\varepsilon$ an approximate solution of {\em near-minimal} rank and representation complexity of the involved tensor factors. 
Specifically, (i) it is based on stable tensor formats relying on optimal subspaces; (ii)   successive solution  updates involve
a {\em combined refinement} of ranks and factor discretizations; (iii) (near-)optimality is achieved, thanks to (i), through
accompanying suitable {\em subspace correction} and {\em coarsening} schemes.

The following comments on the main ingredients are to provide some orientation. 
 A first essential step is to choose a {\em universal} basis for functions of
a {\em single variable} in $H_i$. Here, we focus on wavelet bases, but other systems like the trigonometric system for periodic problems are conceivable as well.
As soon as functions of a single variable, especially the factors in our rank-one tensors,
 are expanded in such a basis, the whole problem of approximating $u$ reduces to approximating its {\em infinite}
 coefficient  tensor $\bu$ induced by the expansion
 $$
 u=\sum_{\nu\in \nabla^d} u_\nu\, \Psi_\nu\,,\quad
 \Psi_\nu := \psi_{\nu_1}\otimes \cdots\otimes \psi_{\nu_d},\quad \bu = (u_\nu)_{\nu\in\nabla^d},
 $$
 see below. The original operator equation \eref{opeq} is then equivalent to an infinite system
 \beqn
 \label{bopeq}
 \bA \bu =\bbf,\quad \mbox{where}\quad \bA =\big(\langle A\Psi_\nu,\Psi_{\nu'}\big\rangle\big)_{\nu,\nu'\in\nabla^d},\,\,
 \bbf = \big(\langle f,\Psi_\nu\rangle\big)_{\nu\in\nabla^d}.
  \eeqn
 For standard types of Sobolev spaces $V$ it is well understood how to rescale the tensor product basis $\{ \Psi_\nu \}_{\nu\in\nabla^d}$ in such a way that it becomes a
 {\em Riesz basis} for $V$.
 This, in turn, together with the fact that $\kappa_{V\to V'}(A):= \|A\|_{V\to V'}\|A^{-1}\|_{V'\to V}$ is finite, allows one to show
 that $\kappa_{\ell_2\to\ell_2}(\bA)$ is finite{, see \cite{actanum}}. Hence one can find a positive $\omega$ such that $\|{\bf I}- \omega \bA\|_{\ell_2\to\ell_2}\leq \rho< 1$, 
 i.e., the operator ${\bf I}- \omega \bA$ is a contraction so that the iteration 
 \beqn
 \label{iteration}
 \bu_{k+1} := \bu_k +\omega (\bbf - \bA \bu_k),\quad k=0,1,2,\ldots,
 \eeqn
 converges for any initial guess to the solution $\bu$ of \eref{bopeq}.
  
  Of course, \eref{iteration} is only an idealization because the full coefficient sequences $\bu_k$ cannot be computed. Nevertheless, adaptive wavelet methods can be 
  viewed as realizing \eref{iteration} {\em approximately}, keeping possibly few wavelet coefficients ``active''
 while still preserving enough accuracy to ensure convergence to $\bu$ (see e.g.
 \cite{Cohen:01,Cohen:02}).

 In the present high-dimensional  context this kind of adaptation is no longer feasible. Instead, we propose here a ``much more nonlinear'' 
 adaptation concept. Being able to keep   increasingly accurate approximations on a path towards   near-minimal rank 
 approximations with properly sparsified tensor factors relies crucially on suitable correction mechanisms.
 An important contribution of this work is to identify and analyze just such methods.
  Conceptually, they are embedded in a properly perturbed numerical realization
 of \eref{iteration}  of the form
{ 
\beqn
 \label{practicaliter}
 \bu_{k+1} = {\rm C}_{\varepsilon_2(k)}\big({\rm P}_{\varepsilon_1(k)}(\bu_k + \omega
 (\bbf - \bA \bu_k))\big), \quad k=0,1,2,\ldots,
 \eeqn
 where ${\rm P}_{\varepsilon_1(k) }$, ${\rm C}_{\varepsilon_2(k) }$ are
 certain {\em reduction} operators and the $\varepsilon_i(k)$, $i=1,2$,} are
 suitable tolerances which decrease for increasing $k$. 
 
 More precisely, the purpose of ${\rm P}_{\varepsilon }$ is to ``correct'' the current tensor expansion and, 
 in doing so, reduce the rank subject to an accuracy tolerance $\varepsilon$.
 We shall always refer to such a rank reduction operation as a \emph{recompression}.
 For this operation to work as desired, it is essential that the employed  tensor format is
 stable in the sense that the best approximation problem for any given ranks is well-posed. As explained above, this excludes the use of the canonical format. Instead we use
  the so-called {\em hierarchical Tucker} (HT) format, since on the one hand it inherits the stability of the Tucker format \cite{Falco:10},
   as a classical best subspace method,  while on the other hand it better ameliorates the curse of dimensionality
 that the Tucker format may still be prone to. In \S \ref{sect:prelim} we collect the relevant prerequisites.
  This draws to a large extent on known results for the finite-dimensional case, but requires
  proper formulation and extension of these notions and facts for the current sequence space setting. 
 The second reduction operation ${\rm C}_{\epsilon }$, in turn, is a {\em coarsening} scheme that
 reduces the number of degrees of freedom used by
 the wavelet representations of the tensor factors, again subject to some accuracy constraint $\epsilon$. 
 
\subsection{What is New?} 
 
The use of rank reduction techniques in iterative schemes is in principle not new, 
see e.g.\ \cite{Beylkin:02,Beylkin:05-1,Hackbusch:08,Khoromskij:11-1,Kressner:11,Matthies:12,Ballani:13} and the further
references given in \cite{Grasedyck:13}.
To our knowledge, corresponding approaches can be subdivided roughly into two categories.
In the first one, iterates are always truncated to a fixed tensor rank. This allows one to control
the complexity of the approximation, but convergence of such iterations can be guaranteed only under
very restrictive assumptions (e.g.\ concerning highly effective preconditioners).
In the second category, schemes achieve a desired target accuracy by instead prescribing an error tolerance for the rank truncations, but the corresponding ranks arising during the iteration are not controlled.
A common feature of both groups of results is that they operate on a \emph{fixed discretization} of the underlying
continuous problems.

In contrast, the principal novelty of the present approach can be sketched as follows.
The first key element is to show that based on a known error bound for a given approximation to the unknown solution, a judiciously chosen recompression produces a near-minimal rank approximation to the \emph{solution of the continous problem} for a slightly larger accuracy tolerance.
Moreover, the underlying projections are stable with respect to certain sparsity measures.
As pointed out before, this reduction needs to be intertwined with a sufficiently accurate but possibly coarse approximation of the tensor factors. A direct coarsening of the full
wavelet coefficient tensor would face the curse of dimensionality, and thus would be practically infeasible. The second critical element is therefore to introduce certain
lower-dimensional quantities, termed tensor {\em contractions}, from
which the degrees of freedom to be discarded in the coarsening are identified. 
This notion of contractions also serves to define suitable sparsity classes with respect to wavelet coefficients, facilitating a computationally efficient, rigorously founded combination of tensor recompression and coefficient coarsening.

 These concepts culminate in the main result of this paper, which can be summarized in an admittedly oversimplified way as follows.\vspace{6pt}
 
 \noindent
 {\bf Meta-Theorem:} {\it Whenever the solution to \eref{bopeq} has certain tensor-rank approximation rates
 and when the involved tensor factors have certain   best $N$-term approximation rates, then a judicious
 numerical realization of the iteration \eref{practicaliter} realizes these rates. Moreover, up to logarithmic
 factors, the computational complexity  is optimal. More specifically, for the {smallest $k$} such that
the approximate solution $\bu_k$ satisfies $\|\bu_k -\bu\|_{\ell_2}\leq \tau$, $\bu_k$ has HT-ranks that can be {bounded, up to multiplication by a uniform constant,
by the smallest possible HT-ranks} needed to realize accuracy $\tau$.
}\vspace{6pt}
  
{In the theorem that we will eventually prove we admit classes of operators with unbounded ranks, in which case the rank bounds contain a
factor of the form $\abs{\log \tau}^c$, where $c$ is a fixed exponent.}

 To our knowledge this is the first result of this type, where convergence to the solution of the \emph{infinite-dimensional}
 problem is guaranteed under realistic assumptions, and all ranks arising during the process remain proportional
 to the respective smallest possible ones. 
{A rigorous proof of {\em rank near optimality}, using an iteration of the above type, is to be contrasted to approaches based on greedy approximation as studied e.g.\ in \cite{Cances:11}, where approximations in the (unstable) canonical format are constructed through successive greedy updates.
This does, in principle, not seem to offer much hope for finding minimal or near-minimal rank approximations, as the greedy search operates far from orthonormal bases, and errors committed early in the iteration cannot easily be corrected.
Although variants of the related \emph{proper generalized decomposition}, as studied in \cite{Falco:12}, can alleviate some of these difficulties, e.g.\ by employing different tensor formats, the basic issue of controlling ranks in a greedy procedure remains.}

\subsection{Layout}

The remainder of the paper is devoted to the development of the ingredients and their complexity analysis needed to make the statements in the above Meta-Theorem precise.
 Trying to carry out this program raises some issues which we will briefly address now, as they guide the 
 subsequent developments.

After collecting in \S \ref{sect:prelim} some preliminaries,   
\S \ref{sect:compress} is devoted to  a pivotal element of our approach, namely the development 
and analysis of suitable {\em recompression and coarsening schemes}
 that yield an approximation in the HT-format that is, for a given target accuracy, of {\em near-minimal} rank  with possibly sparse tensor factors
 (in a sense to be made precise later).
 
Of course, one can hope that the solution of \eref{opeq} is particularly tensor sparse  in the sense that relatively low HT-ranks
already provide high accuracy if the data $f$ are tensor sparse, and if the operator $A$ (resp.\ $\bA$) 
is tensor sparse in the sense that its application does not
increase ranks too drastically. Suitable models of operator classes that allow us to properly weigh tensor sparsity and wavelet expansion sparsity are introduced and analyzed in \S \ref{sect:op}.  
The approximate application of such operators with certified output accuracy builds on the findings in \S \ref{sect:compress}.
 
Finally, in \S \ref{sect:scheme} we formulate an {\em adaptive iterative algorithm} and analyze its complexity.
Starting from the coarsest possible approximation $\bu^0 =0$,
approximations in the tensor format are built successively, where the error tolerances
in the iterative scheme are updated for each step in such a way that two goals are achieved.
On the one hand, the tolerances are sufficiently stringent to guarantee the convergence
of the iteration up to any desired target accuracy. On the other hand, we ensure that at each stage of the
iteration, the approximations remain sufficiently coarse to realize
the Meta-Theorem formulated above. Here we specify concrete tensor approximability assumptions
on $\bu$, $\bbf$ and $\bA$ that allow us to make its statement precise.

\section{Preliminaries}\label{sect:prelim}
In this section we set the notation and collect the relevant ingredients for stable tensor formats in
the infinite-dimensional setting.  
In the remainder of this work, we shall use for simplicity the abbreviation
$\norm{\cdot}:=\norm{\cdot}_{\spl{2}}$, with the $\spl{2}$-space on the
appropriate index set.

Our basic assumption is that we have a Riesz basis $\{  \Psi_\nu \}_{\nu\in\nabla^d}$ for $V$, where
$\nabla$ is a countable index set. In other words, we require that the index set has Cartesian product structure. 
Therefore any $u\in V$ can be identified with its basis coefficient sequence 
 $\mathbf{u} := {(u_\nu)_{\nu\in\nabla^d}}$ {in the unique representation $u=\sum_{\nu\in\nabla^d}u_\nu\Psi_\nu$,}
with uniformly equivalent norms.
Thus, $d$ will in general correspond to the spatial dimension of the domain of functions under consideration.
In addition it can be important to reserve the option of grouping some of the variables in a 
possibly smaller number $m\leq d$ of portions of variables, i.e.,   
  $m\in \N$ and $d = d_1 + \ldots + d_m$ for $d_i\in\N$.

A canonical point of departure for the construction of $\{ \Psi_\nu\}$ is a collection of Riesz bases for each component Hilbert space $H_i$ (see \eqref{eq:htensorspace}),
which we denote by $\{ \psi^{H_i}_\nu\}_{\nu\in\nabla^{H_i}}$.
To fit in the above context, we may assume without loss of generality that all $\nabla^{H_i}$ are identical, denoted
by $\nabla$.
The precise structure of $\nabla$ is irrelevant at this point;
however, in the case that the $\psi^{H_i}_\nu$ are wavelets, each $\nu =(j,k)$ encodes a dyadic level $j=\abs{\nu}$
and a spatial index $k=k(\nu)$.
This latter case is of particular interest, since for instance when $V$ is a Sobolev space,
a simple rescaling of $\psi^{H_1}_{\nu_1} \otimes \cdots\otimes \psi^{H_d}_{\nu_d}$ 
yields a Riesz basis $\{ \Psi_\nu\}$ for $V\subseteq H$ as well.

A simple scenario would be $V=H=\spL{2}([0,1]^d)$, which is the situation considered in our numerical illustration
in \S\ref{sect:numexp}.
A second example are elliptic diffusion equations with stochastic coefficients. In this case, $V= \spH{1}_0(\Omega) \otimes \spL{2}([-1,1]^\infty)$, and $H = \spL{2}(\Omega \times [-1,1]^\infty)$. Here a typical choice of bases for $\spL{2}([-1,1]^\infty)$ are tensor products of polynomials on $[-1,1]$, while one can take a wavelet basis for $\spH{1}_0(\Omega)$, obtained by rescaling a standard $\spL{2}$ basis.
A third representative scenario concerns diffusion equations on high-dimensional product domains $\Omega^d$. Here, for instance, $V = \spH{1}(\Omega^d)$ and $H = \spL{2}(\Omega^d)$. We shall comment on some additional difficulties that arise in the application of operators in this case in Remark \ref{rmrk:sobolev}.
 
We now regard $\bu$ as a tensor of order $m$ on $\nabla^d =
\nabla^{d_1}\times\cdots\times \nabla^{d_m}$ {and look for representations or approximations of $\bu$ in terms of rank-one
tensors 
$$
\bV^{(1)}\otimes \cdots \otimes \bV^{(m)} := \big(V^{(1)}_{\nu_1}\cdots V^{(m)}_{\nu_m}\big)_{\nu=(\nu_1,\ldots,\nu_m)\in\nabla^d}.
$$}
Rather than looking for
approximations or representations in the canonical format 
$$
\bu = \sum_{k=1}^r a_k \bU_k^{(1)}\otimes \cdots \otimes \bU_k^{({m})},
$$
we will employ tensor representations of {a format that is perhaps best motivated as follows. Consider for each $i=1,\ldots,m$
  (finitely or infinitely many) pairwise orthonormal sequences  $\bU^{(i)}_k =(U^{(i)}_{\nu_i,k})_{\nu_i\in\nabla^{d_i}}\in \ell_2(\nabla^{d_i})$,
$k=1,\ldots,r_i$,
that is, 
$$
\langle \bU^{(i)}_k,\bU^{(i)}_l\rangle := \sum_{\nu_i\in\nabla^{d_i}} U^{(i)}_{\nu_i,k}U^{(i)}_{\nu_i,l} = \delta_{k,l},\quad i=1,\ldots,m.
$$
We stress that here and in the sequel $r_i=\infty$ is admitted.
The matrices $\bU^{(i)} = \big(U^{(i)}_{\nu_i,k}\big)_{\nu_i\in \nabla^{d_i},1\leq k\leq r_i}$ are 
often termed \emph{orthonormal mode frames}.
It will be convenient to use the notational convention 
$\kk{k} = (k_1,\ldots, k_t)$,
$\kk{n} = (n_1,\ldots,n_t)$, $\rr{r} = (r_1,\ldots,r_t)$, and so forth,
  for multiindices in $\N^t_0$, $t\in\N$.
Defining for
$\rr{r} \in\N_0^m$
\begin{equation*}
  \KK{m}(\rr{r}) := \left\{
  \begin{array}{ll}
   \bigtimes_{i=1}^m \zinterval{1}{{r}_i} & \text{if $\min \rr{r} >
  0$,}\\
\emptyset & \text{if $\min \rr{r} =
  0$} \,,
 \end{array} \right.
\end{equation*}  
   and noting that 
 $\ell_2(\nabla^d)=\bigotimes_{j=1}^m\ell_2(\nabla^{d_j})$ is a tensor product Hilbert space, the tensors
\begin{equation}
\label{UUdef}
\UU_\kk{k} := \bU^{1)}_{k_1}\otimes \cdots \otimes \bU^{(m)}_{k_m} ,
\quad \kk{k}\in \KK{m}(\rr{r}),
\end{equation}
form an orthonormal basis for the subspace of $\ell_2(\nabla^d)$, generated by the system $\UU :=  (\UU_\kk{k})_{\kk{k}\in\KK{m}(\rr{r})}$.
Hence,  for any $\bu\in\ell_2(\nabla^d)$ the orthogonal projection
\begin{equation}
\label{P-T}
\Proj_{\UU}\bu =\sum_{\kk{k}\in \KK{m}(\rr{r})}  a_\kk{k} \UU_\kk{k}, \quad a_\kk{k} = \langle \bu,    
 \UU_\kk{k} \rangle,\, \kk{k}\in\KK{m}(\rr{r}),
\end{equation} 
is the best approximation to $\bu\in\ell_2(\nabla^d)$ from the subspace spanned by $\UU$.  
The uniquely defined order-$m$ tensor $\mathbf{a}$ with entries $\langle \bu,\UU_\kk{k}\rangle$, $\kk{k}\in \KK{m}(\rr{r})$,  is referred
to as \emph{core tensor}. 
Moreover, when the $\bU^{(i)}_k$, $k\in \N$
are bases for all of $\ell_2(\nabla^{d_i})$, that is, $\KK{m}(\rr{r}) =\N^m$, one has, of course, $\Proj_\UU\bu =\bu$, while for any $\rr{s}\leq \rr{r}$, componentwise,
the ``box-truncation''
\begin{equation}
\label{oproj}
\Proj_{\UU,\rr{s}} \bu := \sum_{\kk{k}\in \KK{m}(\rr{s})} \langle \bu,\UU_\kk{k}\rangle\UU_\kk{k}
\end{equation}
is a simple mechanism of further reducing the ranks of an approximation from the subspace spanned by $\UU$ at the expense of a minimal loss of accuracy.}

{The existence of best approximations and their realizability through linear projections suggests approximating a given tensor in $\ell_2(\nabla^d)$
by expressions of the form
}
\begin{equation}\label{eq:tuckerd}
  \mathbf{u} =  \sum_{k_1=1}^{r_1} \cdots \sum_{k_m=1}^{r_m}
    a_{k_1,\ldots,k_m} \,(\mathbf{U}^{(1)}_{k_1} \otimes \cdots \otimes
    \mathbf{U}^{(m)}_{k_m} )\, \,,
\end{equation}
{even without insisting on the \emph{$i$th mode frame} $\bU^{(i)}$
 to have pairwise orthonormal column vectors $\mathbf{U}^{(i)}_k\in\spl{2}(\nabla^{d_i})$, $k = 1,\ldots,r_i$.
However, these columns can always be orthonormalized, which results in a corresponding modification of the core tensor $\mathbf{a}
= (a_{\kk{k}})_{\kk{k}\in\KK{m}(\kk{r})}$; {for fixed mode frames, the latter is uniquely determined}.}
 When writing sometimes for convenience $(\bU^{(i)}_k)_{k\in\N}$, although the $\bU^{(i)}_k$
may be specified through \eqref{eq:tuckerd} only for $k\leq r_i$, it will always be understood to mean $\bU^{(i)}_k=0$, for $k> r_i$.

If {the core tensor} $\mathbf{a}$ is represented directly by its entries, \eref{eq:tuckerd} corresponds to
the so-called \emph{Tucker format} \cite{Tucker:64,Tucker:66} or \emph{subspace representation}.
The {\em hierarchical Tucker} format \cite{Hackbusch:09-1}, as well
as the special case of the {\em tensor train} format \cite{Oseledets:11}, correspond
to representations in the form \eqref{eq:tuckerd} as well, {but use a further structured tensor decomposition for the core tensor $\mathbf{a}$ that can exploit a stronger type of \emph{information sparsity}.}
{{For $m=2$  the \emph{singular value decomposition} (SVD) or its infinite dimensional counterpart, the \emph{Hilbert-Schmidt decomposition}, yield $\bu$-dependent mode frames that even give a diagonal core tensor. Although this is no longer possible for $m > 2$, the SVD remains the main work horse behind Tucker as well as hierarchical Tucker formats. For the convenience of the reader, we summarize below the relevant facts for these tensor representations in a way tailored to the present needs.}

\subsection{Tucker format}\label{ssect:tucker}

It is instructive to consider  first  the simpler case of the Tucker format in more detail.

\subsubsection{Some Prerequisites}

As mentioned before, for a general $\mathbf{u}\in \spl{2}(\nabla^d)$, the sum in
\eqref{eq:tuckerd} may be infinite. 
{For each
$i\in\{1,\ldots,m\}$ we consider the mode-$i$ matricization of $\mathbf{u}$,
that is, the infinite matrix
$(u^{(i)}_{\nu,\tilde\nu})_{\nu\in\nabla^{d_i},\tilde\nu\in\nabla^{d-d_i}}$ with
entries $u^{(i)}_{\nu_i,\check\nu_i} := u_\nu$ for $\nu\in\nabla^{d}$, which defines a
Hilbert-Schmidt operator}
\begin{equation}
\label{eq:modei_matricization} {
T^{(i)}_{\bu} \colon \spl{2}(\nabla^{d-d_i}) \to \spl{2}(\nabla^{d_i})\,,\;
(c_{\tilde\nu})_{\tilde\nu\in\nabla^{d-d_i}} \mapsto \Bigl( \sum_{\tilde \nu\in
\nabla^{d-d_i}} u^{(i)}_{\nu,\tilde\nu} c_{\tilde \nu}
\Bigr)_{\nu\in\nabla^{d_i}} \,.  }
\end{equation}
We define the {\em rank vector} $\rank (\mathbf{u})
$ by its entries
\begin{equation}
\label{eq:def_multilin_rank} {
 \rank_{i}(\mathbf{u}) := \dim \range T^{(i)}_\bu \,,\quad i=1,\ldots,m \,,  }
\end{equation}
{see \cite{Hitchcock}.}
It is referred to as the \emph{multilinear rank}
of $\mathbf{u}$. We denote by
\begin{equation}
\label{def:R-T}
\mathcal{R}= \mathcal{R}_{\Tcal} := (\N_0 \cup \{ \infty \})^m
\end{equation}
the set of admissible rank vectors in the Tucker format.
For such rank vectors $\rr{r} \in {\mathcal{R}}$, we introduce the notation
\[    \abs{\rr{r}}_\infty := \max_{j=1,\ldots,m} \,{r}_j \,.   
 \]
{Given any $\rr{r} \in \mathcal{R}$, we can then define  the set}
\begin{equation}\label{eq:tucker_tensorset}
  \Tcal(\rr{r}) := \bigl\{ 
  \mathbf{u}\in\spl{2}(\nabla^d)  \colon \rank_i(\mathbf{u}) \leq {r}_i ,\,
  i=1,\ldots,m\bigr\} \,,
\end{equation}
{of those sequences whose multilinear rank is bounded componentwise by $\kk{r}$. It is easy to see that
the elements of $\Tcal(\rr{r})$ possess a representation of the form \eqref{eq:tuckerd}. Specifically,
for   any system of orthonormal mode frames $\VV =
\bigl(\mathbf{V}^{(i)}\bigr)_{i=1}^m$  with $r_i$ columns (where $r_i$ could be infinity),  
the $\VV$-{\em rigid} Tucker class
\beqn
\label{T-rigid}
\Tcal(\VV,r):= \{ \Proj_{\VV} \bv: \bv\in \ell_2(\nabla^d)\}
\eeqn
is contained in $\Tcal(\rr{r})$.  
}

{The actual computational complexity of the elements of $\Tcal(\rr{r})$ can be quantified by} 
\begin{equation}  {
 \supp_i (\mathbf{u}) := \bigcup_{z \in \range T^{(i)}_\bu}  \supp z \,.  }
\end{equation}
{It is not hard to see that these quantities are controlled by the  ``joint support'' of the $i$th mode frame, that is,
$ \supp_i (\mathbf{u}) \subseteq \bigcup_{k\leq r_i}\bU^{(i)}_k$.
Note} that if $\#\supp_i(\bu) < \infty$,  {one} necessarily also {has} $\rank_i(\bu)
<\infty$.
 
The following result, which can be found e.g. in 
\cite{Uschmajew:10,Falco:10,Hackbusch:12}, ensures the existence of best approximations in $ \Tcal(\rr{r})$
also for infinite ranks.

\begin{theorem}
\label{thm:tucker_bestapprox}
Let $\mathbf{u}\in \spl{2}(\nabla^d)$
and $0 \leq {r}_i \leq
\rank_i(\mathbf{u})$, then there exists $\mathbf{v} \in \Tcal(r)$ such that
\begin{equation*}
  \norm{\mathbf{u} - \mathbf{v}} =
  \min_{\rank(\mathbf{w})\leq \rr{r}} \norm{\mathbf{u} - \mathbf{w}} \,.
\end{equation*}
\end{theorem}

{The matricization $T^{(i)}_\bu$ of a given $\bu\in\ell_2(\nabla^d)$, defined in \eqref{eq:modei_matricization},
allows one to invoke the SVD or Hilbert-Schmidt decomposition.}
{By the spectral theorem, for each $i$ there exist a nonnegative real sequence
$(\sigma^{(i)}_n)_{n\in\N}$, where $\sigma^{(i)}_n$ are the eigenvalues of
{$\bigl((T^{(i)}_{\bu})^* T^{(i)}_{\bu}\bigr)^{1/2}$}, as well as orthonormal bases $\bU^{(i)}=\{
\mathbf{U}^{(i)}_n\}_{n\in \N}$ for a subspace of $\spl{2}(\nabla^{d_i})$ and $\{
\mathbf{V}^{(i)}_n\}_{n\in\N}$ for $\spl{2}(\nabla^{d-d_i})$
(again tacitly assuming that $\bU^{(i)}_n=\bV^{(i)}_n =0$ for
$n> \dim\,{\rm range}(T_\bu^{(i)})$), such that
\begin{equation}\label{eq:hilbertschmidt_decomp}
 T^{(i)}_{\bu} = \sum_{n\in\N} \sigma^{(i)}_n \langle \mathbf{V}^{(i)}_n, \cdot\rangle
 \mathbf{U}^{(i)}_n \,.
\end{equation}
The $\sigma^{(i)}_k$ are   referred to as \emph{mode-$i$ singular values}.}

To simplify notation {in a summary of the properties of the particular orthonormal mode frames  $\bU^{(i)}$, $i=1,\ldots,m$, 
defined by \eqref{eq:hilbertschmidt_decomp}, 
we define for any vector} $\kk{x} = (x_i)_{i=1,\ldots,m}$ 
and for $i\in\zinterval{1}{m}$, 
\begin{equation}\label{eq:delentry}
\begin{aligned}
  \kk{\check x}_i &:= (x_1, \ldots, x_{i-1},
  x_{i+1},\ldots,x_m)  \,, \\
  \kk{\check x}_i|_y &:= (x_1, \ldots, x_{i-1}, y,
  x_{i+1},\ldots,x_m)
\end{aligned}
\end{equation}
to refer to the corresponding vector with entry $i$ deleted or
entry $i$ replaced by $y$, respectively.
We shall also need the auxiliary quantities
\begin{equation} 
 \label{eq:tensor_nmodemat}
a^{(i)}_{pq} := \sum_{\kk{\check k}_i\in \mathsf{K}_{m-1}(\rr{\check
r}_i)}a_{\kk{\check k}_i|_p}a_{\kk{\check k}_i|_q}\,, 
\end{equation}
derived from the core tensor, where $i\in\{1,\ldots,m\}$ and $p,q\in
\zinterval{1}{r_i}$.

\subsubsection{Higher-Order Singular Value Decomposition}

{The representation \eqref{eq:hilbertschmidt_decomp} is the main building block
of the}
\emph{higher-order singular value decomposition} (HOSVD) \cite{Lathauwer:00},
for the Tucker tensor format \eqref{eq:tuckerd}. 
In the following theorem, we summarize its properties 
in the more general case of infinite-dimensional sequence spaces, where the
singular value decomposition is replaced by the spectral theorem for compact
operators. These facts could also be extracted from the treatment in
\cite[Section 8.3]{Hackbusch:12}.
 
\begin{theorem}
\label{thm:hosvd_properties}
For any $\mathbf{u}\in \spl{2}(\nabla^d)$ 
{the} orthonormal mode frames $\{ \mathbf{U}^{(i)}_k\}_{k\in\N}$, $i=1,\ldots,m$, with
$\mathbf{U}^{(i)}_k\in\spl{2}(\nabla^{d_i})${, defined by  \eqref{eq:hilbertschmidt_decomp}},
{and the corresponding core tensor $\mathbf{a}$ with entries $\ba_\kk{k} = \langle \bu,\UU_\kk{k}\rangle$, have the following properties:}
\begin{enumerate}[\rm (i)]
  \item For all $i\in\{1,\ldots,m\}$ we have
  $(\sigma^{(i)}_k)_{k\in\N}\in\spl{2}(\N)$, and $\sigma^{(i)}_k \geq
  \sigma^{(i)}_{k+1}\geq 0$ for all $k\in\N$, {where $\sigma^{(i)}_k$ are the mode-$i$ singular values in \eqref{eq:hilbertschmidt_decomp}}.
  \item For all $i\in\{1,\ldots,m\}$ and all $p,q\in\N$, we have $a^{(i)}_{pq}
  = \bigabs{\sigma^{(i)}_p}^2 \delta_{pq}$ {where the $a^{(i)}_{pq}$ are defined by \eqref{eq:tensor_nmodemat}}.
  \item For each $\rr{r} \in\N^m_0$, we have
   \begin{equation}\label{eq:tensor_hosvd_errest}
\Bignorm{\mathbf{u} - \sum_{\kk{k}\in\KK{m}(\rr{r})}
  a_\kk{k} \UU_\kk{k} } \leq \Bigl(\sum_{i=1}^m \sum_{k = r_i + 1}^{\infty}
\abs{\sigma^{(i)}_k}^2\Bigr)^{\frac{1}{2}} \leq \sqrt{m}
\inf_{\rank(\mathbf{w})\leq \rr{r}} \norm{\mathbf{u} - \mathbf{w}}
\,.
\end{equation}
\end{enumerate}
If in addition
$\supp \mathbf{u} \subseteq \Lambda_1\times \cdots\times \Lambda_m
\subset\nabla^d$ for finite $\Lambda_i\subset\nabla^{d_i}$, then $\supp
\mathbf{U}^{(i)}_k \subseteq \Lambda_i$ and we have $\supp \mathbf{a} \subseteq
\KK{d}(\rr{\bar r})$ with $\rr{\bar r}\in\N_0^m$ satisfying $\bar r_i
\leq \#\Lambda_i$ for $i=1,\ldots,m$.
\end{theorem}

\begin{proof}
The representation \eqref{eq:hilbertschmidt_decomp} converges in the
Hilbert-Schmidt norm and, as a consequence, we have
\begin{equation}
\label{eq:hilbertschmidt_decomp_matrix}  
\mathbf{u} = \Bigl( \sum_{n\in\N} 
 \sigma^{(i)}_n \mathbf{U}^{(i)}_{\nu_i,n}  \mathbf{V}^{(i)}_{\check\nu_i,n} 
\Bigr)_{\nu\in\nabla^d} \,, 
\end{equation} 
with convergence in $\spl{2}(\nabla^d)$. 
Furthermore, $\{\UU_\kk{n}\}_{\kk{n}\in\N^m}$ with $\UU_\kk{n} :=
\bigotimes_{j=1}^m \mathbf{U}^{(j)}_{n_j}$ is an orthonormal 
system in $\spl{2}(\nabla^d)$ (spanning a strict subspace 
of $\spl{2}(\nabla^d)$ when $\abs{\rank(\bu)}_\infty <\infty$). {For
$a_\kk{n} = \langle \mathbf{u}, \UU_\kk{n}\rangle$
we have thus shown} $\mathbf{a} = (a_\kk{n})\in\spl{2}(\N^m)$ and $\mathbf{u} =
\sum_{\kk{n}\in\N^m} a_\kk{n} \UU_\kk{n}$.
The further properties of the expansion can now be obtained
along the lines of \cite{Lathauwer:00}, see also
\cite{Hackbusch:12,Bachmayr:12}.
\end{proof}

In what follows we shall denote by
\beqn
\label{UU-T}
\UU(\bu) = \UU^{\Tcal}(\bu) := \{\bU^{(i)}: i=1,\ldots,m, \, \mbox{generated by HOSVD}\}
\eeqn
the particular system of orthonormal mode frames generated for a given $\bu$ by HOSVD.
It will occasionally be important to identify the specific tensor format to which a given system of mode frames refers, for which we use a corresponding superscript, such as in $\UU^{\Tcal}$ for the Tucker format.

Property (iii) in Theorem
\ref{thm:hosvd_properties} leads to a simple procedure for truncation to lower
multilinear ranks with an
explicit error estimate in terms of the mode-$i$ singular values. In this
manner, one does not necessarily obtain the best approximation for prescribed
rank, but the approximation is 
quasi-optimal in the sense that the error is at most by a factor $\sqrt{m}$
larger than the error of best approximation with the same multilinear rank.

{We now introduce the notation 
\begin{equation}
\label{eq:tucker_err}
  \err_{\rr{\tilde r}}(\bu) = \err^\Tcal_{\rr{\tilde r}}(\bu) :=
  \Bigl( \sum_{i = 1}^m 
    \sum_{k = \tilde r_{i}+ 1}^{\rank_i(\mathbf{u})} \bigabs{\sigma^{(i)}_{k}}^2 \Bigr)^{\frac12} \,,
    \quad  \rr{\tilde r} \in \N_0^m \,. 
\end{equation}
This quantity plays the role of a computable error estimate, as made explicit in the following direct consequence of {Theorem} \ref{thm:hosvd_properties}.}

\begin{corollary}
\label{cor:tensor_recompression_est}
For an HOSVD of $\mathbf{u} \in \spl{2}(\nabla^d)$, as in Theorem
  \ref{thm:hosvd_properties}, and for $\rr{\tilde r}$ with $0\leq \tilde r_i
  \leq \rank_i(\mathbf{u})$,    
we have
\begin{equation*}  
\|\mathbf{u} - \Proj_{\UU(\bu),\rr{\tilde r}}(\bu)\| \leq 
   \err^\Tcal_{\rr{\tilde r}} (\bu)  \leq \sqrt{m} \inf_{\mathbf{w}\in\Tcal(\rr{r})} 
  \|\mathbf{u} - \mathbf{w}\| \,,  
\end{equation*}
{where $\Proj_{\UU(\bu),\rr{\tilde r}}$ 
is defined in \eqref{oproj}. }
\end{corollary}

{While projections to subspaces spanned by the $\UU_\kk{k}(\bu)$, $\kk{k}\in\KK{m}(\rr{r})$, do in general not realize 
the best approximation from $\Tcal(\rr{r})$ (only from $\Tcal(\UU(\bu),\rr{r})$), exact best approximations are still orthogonal projections
based on suitable mode frames.
\begin{corollary}
\label{cor:tucker_projbest}
For $\bu\in\spl{2}(\nabla^d)$ and $\rr{r} = (r_i)_{i=1}^m\in \N_0^m$
with $0\leq r_i\leq \rank_i(\bu)$, $i=1,\ldots,m$, there exists an orthonormal  
  mode frame system $\bar\UU(\bu, \rr{r})$ such that
$$  
\norm{\bu - \Proj_{\bar\UU(\bu, \rr{r})} \bu} =
\min_{\bw\in\tucker{\rr{r}}}\norm{\bu - \bw},   
$$
with $\Proj_{\bar\UU(\bu, \rr{r})}$ given by \eqref{P-T}.
\end{corollary}
\begin{proof}
By Theorem \ref{thm:tucker_bestapprox}, a best approximation of  
ranks $\rr{r}$ for $\bu$, $$  \bar\bu \in \argmin\{ \norm{\bu - \bv} \colon
\rank_\alpha(\bu) \leq r_\alpha \} \,, $$ exists.
Defining $\bar\UU(\bu, \rr{r}):= \UU(\bar\bu)$ 
as the   orthonormal mode frame system
for $\bar\bu$, given by the HOSVD, we obtain the assertion. 
\end{proof}
}

\begin{remark}
\label{rmrk:hosvd_complexity}
Suppose that for a finitely supported 
vector $\mathbf{u}$ on $\nabla^d$, we have a possibly redundant 
representation
$$  \mathbf{u} = \sum_{\kk{k}\in\KK{m}(\rr{\tilde r})} \tilde a_\kk{k} 
   \bigotimes_{i=1}^m \mathbf{\tilde U}^{(i)}_{k_i} \,,  $$
where the vectors $\mathbf{\tilde U}^{(i)}_{k}$, $k=1,\ldots,\tilde r_i$
may be linearly dependent. Then by standard linear algebra procedures,
we can obtain a HOSVD of $\mathbf{u}$ with a number of arithmetic
operations that can be estimated by
\begin{equation}
 C m \abs{\rr{\tilde r}}_\infty^{m+1} 
 +  C \abs{\rr{\tilde r}}_\infty^2 \sum_{i=1}^m \#\supp_i(\mathbf{u}) \,.
\end{equation}
where  $C>0$ is an absolute constant (see, e.g., \cite{Hackbusch:12}).
\end{remark}

\subsection{The Hierarchical Tucker Format}\label{ssect:htucker}
The Tucker format as it stands, in general, still gives rise to an increase of degrees of freedom 
that is exponential in $d$.
One way to mitigate the curse of dimensionality is to further decompose the core tensor ${\bf a}$ in \eref{eq:tuckerd}.
We now briefly formulate the relevant notions concerning the {\em hierarchical {Tucker format}} in the present
sequence space context, following essentially the developments in \cite{Hackbusch:09-1,Grasedyck:10}, see also
\cite{Hackbusch:12}.

\subsubsection{Dimension Trees}

\begin{definition}\label{def:htucker_dimtree}
 Let $m\in\N$, $m\geq 2$. A set $\hdimtree{m} \subset 2^{\{ 1,\ldots,m\}}$ is
 called a {(binary)} dimension tree if the following hold:
 \begin{enumerate}[{\rm(i)}]
   \item $\{1,\ldots,m\} \in \hdimtree{m}$ and for each $i\in\{1,\ldots,m\}$, we have
   $\{i\} \in \hdimtree{m}$.
   \item Each $\alpha \in \hdimtree{m}$ is either a singleton or there exist unique
   disjoint $\alpha_1, \alpha_2 \in \hdimtree{m}$, called {\em children} of $\alpha$,   
such that $\alpha = \alpha_1 \cup \alpha_2$. 
 \end{enumerate}
Singletons $\{ i\} \in \hdimtree{m}$ are referred to as \emph{leaves}, $$\hroot{m} := \{1,\ldots,m\}$$
as \emph{root}, and elements
of $\Ical(\hdimtree{m}) := \hdimtree{m}\setminus\bigl\{\hroot{m}
,\{1\},\ldots,\{m\} \bigr\}$ as \emph{interior nodes}.
The set of leaves is denoted by $\leaf{m}$, where we
additionally set $\nonleaf{m} :=
\hdimtree{m}\setminus\mathcal{L}(\mathcal{D}_m) =
\Ical(\hdimtree{m})\cup\{ \hroot{m}\}$. When an enumeration of $\leaf{m}$ is
required, we shall always assume the ascending order with respect to the indices,
 i.e., in the form $\{\{1\},\ldots, \{m\}\}$.
\end{definition}

It will be convenient to introduce the two functions 
$$
\child{i}: \cD_m\setminus \cL(\cD_m)\to \cD_m\setminus \{\hroot{m}\},\quad 
\child{i}(\alpha):= \alpha_i\,,\qquad i=1,2\,,
$$
producing the ``left'' and ``right'' children of a non-leaf node $\alpha {\in \mathcal{N}(\cD_m)}$ which, in view of
Definition \ref{def:htucker_dimtree}, are well-defined up to their order, which we fix by the condition $\min \alpha_1 < \min\alpha_2$.

Note that for a binary dimension tree as defined above, $\# \hdimtree{m} = 2m-1$
and $\#\nonleaf{m} = m-1$.

\begin{remark}
The restriction to binary trees in Definition \ref{def:htucker_dimtree} is not
necessary, but leads to the most favorable complexity estimates for algorithms
operating on the resulting tensor format. With this restriction dropped, the Tucker format
\eqref{eq:tuckerd} can be treated in the same framework, with the $m$-ary
dimension tree consisting only of root and leaves, i.e., $\bigl\{
\hroot{m}, \{1\},\ldots,\{m\} \bigr\}$. {In principle, all subsequent results carry over to more general dimension trees (see \cite[Section 5.2]{FalHaNou}).}
\end{remark}

\begin{definition}
We shall refer to a family 
$$\UU = \bigl\{ \bU^{(\alpha)}_k \in
\spl{2}(\nabla^{\sum_{j\in\alpha} d_j}) \,\colon\, \alpha\in \hdimtree{m}\setminus\{\hroot{m}\bigr\},
k=1,\ldots,k_\alpha \} \,, $$ 
with $k_\alpha\in\N\cup \{\infty\}$ for each $\alpha\in\hdimtree{m}\setminus\{\hroot{m}\}$,
as \emph{hierarchical mode frames}.
In addition, these are called \emph{orthonormal} if
for all $\alpha\in\hdimtree{m}\setminus\{\hroot{m}\}$, we have 
$\langle \bU^{(\alpha)}_i, \bU^{(\alpha)}_j\rangle = \delta_{ij}$ for
$i,j=1,\ldots,k_\alpha$, and \emph{nested} if
\begin{eqnarray*}
&&  \overline{\linspan}\{ \bU^{(\alpha)}_k\colon
k=1,\ldots,k_\alpha \} \\ 
&& \quad\quad \subseteq
 \overline{\linspan}\{ \bU^{(\leftchild(\alpha))}_k\colon
k=1,\ldots,k_{\leftchild(\alpha)} \}  
\otimes
\overline{\linspan}\{ \bU^{(\rightchild(\alpha))}_k\colon
k=1,\ldots,k_{\rightchild{\alpha}} \} \,. 
\end{eqnarray*}
As for the Tucker format, 
we set $\mathbf{U}^{(i)} := \mathbf{U}^{(\{i\})}$,
and for $\kk{k}\in\N^m$ we {retain the notation}
$$     
\UU_\kk{k} := \bigotimes_{i=1}^m \bU^{(i)}_{k_i} \,. 
$$
\end{definition}
Again to express that $\UU$ is associated with the hierarchical format we sometimes write $\UU^{\Hcal}$.
{Of course, $\UU^{\Hcal}$ depends on the dimension tree $\cD_m$, which will be kept fixed in what follows.}

{To define hierarchical tensor classes and to construct specific $\bu$-dependent hierarchical mode frames  one can proceed as for the Tucker format.}
 Let $\hdimtree{m}$ be a dimension tree, let $\alpha\in\Ical(\hdimtree{m})$ be an
interior node, and $\beta := \{1,\ldots,m\}\setminus \alpha$.
For $\mathbf{u} \in \spl{2}(\nabla^d)$, we define the
Hilbert-Schmidt operator
\begin{equation}
\label{eq:matricization_def} 
 T^{(\alpha)}_\bu \colon
\spl{2}(\nabla^{\sum_{i\in\beta} d_i}) \to \spl{2}(\nabla^{\sum_{i\in\alpha} d_i})\,,\;
\mathbf{c} \mapsto \Bigl(
\sum_{(\nu_i)_{i\in\beta}} u_{\nu} c_{(\nu_i)_{i\in\beta}}
\Bigr)_{(\nu_i)_{i\in\alpha}} \,,
\end{equation}
and set
\begin{equation*}
\dd_\alpha(\bu) 
:= \dim \range T^{(\alpha)}_\bu ,\quad \alpha \in \cD_m\setminus \hroot{m}\,. 
\end{equation*}
To be consistent with our previous 
notation for leaf nodes $\{i\} \in \hdimtree{m}$,
we use the abbreviation $\dd_i(\mathbf{u}) := \dd_{\{i\}} (\mathbf{u})$.
Again, $\dd_\alpha(\bu)$ can be infinite. The root element of the dimension tree, $\hroot{m}=\{1,\ldots,m\}\in\hdimtree{m}$, is a special case. Here we define 
$$  
T^{(\hroot{m})}_\bu \colon \R \to \spl{2}(\nabla^d) ,\; t \mapsto t \,\bu 
$$
and correspondingly set  
\begin{equation*}
 \dd_{\hroot{m}}{(\bu)} := 1 \,,
  \quad \mathbf{U}^{(\hroot{m})}_1 := \mathbf{u}\,,
  \quad \mathbf{U}^{(\hroot{m})}_k := 0\,,\; k>1,
\end{equation*}
{if $\bu\neq 0$, and otherwise $\dd_{\hroot{m}}(\bu) :=0$.}
To be consistent with the Tucker format we denote by
$$
\rank (\bu) = {\rank_{\hdimtree{m}}(\bu)} := (\rank_\alpha (\bu) )_{\alpha\in\hdimtree{m}\setminus\{\hroot{m}\}}\, 
$$ 
the {\em hierarchical rank vector} associated with $\bu$.
{Since in what follows the dimension tree $\hdimtree{m}$ will  be kept fixed we  suppress the corresponding subscript in the rank vector.}

{
This  allows us to define for a given $\rr{r} = ({r}_\alpha)_{\alpha\in\hdimtree{m}\setminus\{\hroot{m}\}}
\in (\N_0\cup\{\infty\})^{\hdimtree{m}\setminus\{\hroot{m}\}}$, in analogy to \eqref{eq:tucker_tensorset},
the class 
\begin{equation}
\label{eq:hier_tensorset}
  \Hcal(\rr{r}) := \bigl\{ \mathbf{u} \in
  \spl{2}(\nabla^d) \colon  \rank_\alpha(\mathbf{u}) \leq {r}_\alpha \text{
  for all $\alpha \in \hdimtree{m}\setminus\{\hroot{m}\}$} \bigr\} \,.
\end{equation}
For $\Hcal(\rr{r})$ to be non-empty the rank vectors must satisfy certain compatibility conditions, see Proposition \ref{prop:nestedness} below.
As detailed later, the elements of $\Hcal(\rr{r})$ can be represented in terms of hierarchical mode frames in
the so called {\em hierarchical format} with ranks $\rr{r}$.}

{Now,  for a given $\bu\in\ell_2(\nabla^d)$, let} 
$\{ \mathbf{U}^{(\alpha)}_k \}_{k=1}^{\dd_\alpha(\bu)}$, $\mathbf{U}^{(\alpha)}_k\in \spl{2}(\nabla^{\sum_{i\in\alpha} d_i})$
 be the left singular 
vectors and $\sigma^{(\alpha)}_k$ be
the singular values of $T^{(\alpha)}_\bu$. 
In analogy to the Tucker format we   denote by 
{
\beqn
\label{UU-H}
\UU(\bu)= \UU^{\Hcal}(\bu):= \big\{\{\mathbf{U}^{(\alpha)}_k \}_{k=1}^{\dd_\alpha (\mathbf{u})}:\alpha \in \cD_m \big\}
\eeqn
}
 the system of orthonormal hierarchical mode frames 
with rank vectors $\rank(\bu)$.

The observation that the specific systems of hierarchical  mode frames $\UU(\bu)$   have the following {\em nestedness} property, including the root element,
will be crucial. The following fact has been established in a more generally applicable framework of
minimal subspaces in \cite{Hackbusch:12} (cf.~Corollary 6.18 and
Theorem 6.31 there).

\begin{proposition}
\label{prop:nestedness}
For $\bu\in\spl{2}(\nabla^d)$ and $\alpha \in \nonleaf{m}$, 
the mode frames $\{ \mathbf{U}^{(\alpha)}_k \}$ given by the left singular vectors of the
operators $T^{(\alpha)}_\bu$ defined in \eqref{eq:matricization_def} satisfy
\begin{eqnarray*}
&&\overline{\linspan}\{ \bU^{(\alpha)}_k\colon
k=1,\ldots,\dd_\alpha(\bu) \}   
 \subseteq \overline{\linspan}\{ \bU^{(\leftchild(\alpha))}_k\colon
k=1,\ldots,\dd_{\leftchild(\alpha)}(\bu) \} \\
&& \quad\quad \otimes\,
 \overline{\linspan}\{ \bU^{(\rightchild(\alpha))}_k\colon
k=1,\ldots,\dd_{\rightchild(\alpha)}(\bu) \}   \,,
\end{eqnarray*}
i.e., the family of left singular vectors 
of the operators $T^{(\alpha)}_\bu$  is comprized of  orthonormal 
and nested mode frames for $\bu$.
\end{proposition}

{Nestedness entails \emph{compatibility conditions} on the rank vectors $\rr{r}$. In fact, it readily follows from
Proposition \ref{prop:nestedness} that for $\alpha\in \cD_m\setminus\cL(\cD_m)$ one has $\rank_\alpha(\bu)\leq \rank_{c_1(\alpha)}(\bu)
\rank_{c_2(\alpha)}(\bu)$. For necessary and sufficient conditions on a rank vector $\rr{r}=(r_\alpha)_{\alpha\in \cD_m\setminus\{\hroot{m}\}}$
{for existence of corresponding nested hierarchical mode frames}, we refer to \cite[Section 11.2.3]{Hackbusch:12}. In what follows we denote by 
\begin{equation}
\label{def:R-H}
\mathcal{R}=\mathcal{R}_{\Hcal}\subset (\N_0 \cup \{\infty\})^{\cD_m\setminus\cL(\cD_m)}
\end{equation}
 the set of all hierarchical rank vectors satisfying the compatibility conditions for nestedness.
}

{
Following \cite{Falco:10,Hackbusch:12}, we can formulate now the analogue to 
Theorem \ref{thm:tucker_bestapprox}.
\begin{theorem}
\label{thm:htucker_bestapprox}
Let $\mathbf{u}\in \spl{2}(\nabla^d)$, let $\hdimtree{m}$ be a dimension tree,
and let $\rr{r} = (r_\alpha)\in \mathcal{R}_{\mathcal{H}}$ with $0 \leq r_\alpha \leq
\rank_\alpha(\mathbf{u})$ for $\alpha\in\hdimtree{m}\setminus\{\hroot{m}\}$, then there exists $\mathbf{v}
\in \Hcal(\rr{r})$ such that
\begin{equation*}
  \norm{\mathbf{u} - \mathbf{v}} =
  \min
  \bigl\{ \norm{\mathbf{u} - \mathbf{w}} \colon
   \rank_\alpha(\mathbf{w})\leq r_\alpha, \alpha\in\hdimtree{m}\setminus\{\hroot{m}\} \bigr\} \,.
\end{equation*}
\end{theorem}
}

{We recall next the   specific structure of the hierarchical format. Let $\UU$ be a system of hierarchical orthonormal mode frames.} By orthonormality and nestedness,
we obtain for each $\alpha \in \nonleaf{m}$ and $k=1,\ldots,\dd_\alpha(\bu)$  
the expansion
\begin{equation}
\label{eq:htucker_nestedness_expansion}
  \mathbf{U}^{(\alpha)}_k = \sum_{k_1 =
  1}^{\dd_{\leftchild(\alpha)}(\mathbf{u})}
    \sum_{k_2 = 1}^{\dd_{\rightchild(\alpha)}(\mathbf{u})} 
    \bigl\langle
    \mathbf{U}^{(\alpha)}_k , \mathbf{U}^{(\leftchild(\alpha))}_{k_1} \otimes
           \mathbf{U}^{(\rightchild(\alpha))}_{k_2}  \bigr\rangle  \, 
              \mathbf{U}^{(\leftchild(\alpha))}_{k_1} \otimes
           \mathbf{U}^{(\rightchild(\alpha))}_{k_2} \,.
\end{equation}
Defining the matrices $\mathbf{B}^{(\alpha,k)} \in \spl{2}(\N\times \N)$ with
entries
\begin{equation}
  B^{(\alpha,k)}_{k_1,k_2} :=  \bigl\langle
    \mathbf{U}^{(\alpha)}_{k} , \mathbf{U}^{(\leftchild(\alpha))}_{k_1} \otimes
           \mathbf{U}^{(\rightchild(\alpha))}_{k_2}  \bigr\rangle \,,
\end{equation}
\eqref{eq:htucker_nestedness_expansion} can be rewritten as 
\begin{equation}
\label{eq:htucker_moderecursion}
  \mathbf{U}^{(\alpha)}_k = \sum_{k_1 =
  1}^{\dd_{\leftchild(\alpha)}(\mathbf{u})} \sum_{k_2 =
  1}^{\dd_{\rightchild(\alpha)}(\mathbf{u})} B^{(\alpha,k)}_{k_1,k_2} \,
      \mathbf{U}^{(\leftchild(\alpha))}_{k_1} \otimes 
      \mathbf{U}^{(\rightchild(\alpha))}_{k_2},
\end{equation}
providing a decomposition
into vectors $\bU^{\child{i}(\alpha)}_k$, $i=1,2$, which now involve shorter
multiindices supported in the children $\child{i}(\alpha)$.
This decomposition can be iterated as illustrated by the next step. 
Abbreviating $\child{i,j}(\alpha)=\child{i}(\child{j}(\alpha))$,
one obtains
\begin{multline}
\label{htucker_m8}
  \mathbf{U}^{(\alpha)}_k =  \sum_{k_1 =
  1}^{\dd_{\leftchild(\alpha)}(\mathbf{u})} \sum_{k_2 =
  1}^{\dd_{\rightchild(\alpha)}(\mathbf{u})}
  \sum_{\substack{k_{i,1},k_{i,2} k_{j,1},k_{j,2}\\ (i,j)\in\{1,2\}^2}}  
  B^{(\alpha,k)}_{(k_1,k_2)}   \\
\times   B^{(\child{1}(\alpha),k_1
)}_{k_{i,1},k_{j,1}}B^{(\child{2}(\alpha),k_2 )}_{k_{i,2},k_{j,2}}
\bU^{(\child{i,1}(\alpha))}_{k_{i,1}}\otimes 
\bU^{(\child{j,1}(\alpha))}_{k_{j,1}}\otimes 
\bU^{(\child{i,2}(\alpha))}_{k_{i,2}}\otimes
\bU^{(\child{j,2}(\alpha))}_{k_{j,2}}.
\end{multline}
Applying this recursively, any $\mathbf{u}
\in \spl{2}(\nabla^d)$ can be expanded in the form
\begin{equation}
\label{eq:htucker_tdecomp}
  \mathbf{u} = \sum_{k_1 = 1}^{\dd_1(\mathbf{u})} \cdots \sum_{k_m =
  1}^{\dd_m(\mathbf{u})} a_{k_1,\ldots,k_m} \, \mathbf{U}^{(1)}_{k_1}
  \otimes\cdots \otimes \mathbf{U}^{(m)}_{k_m} \,,
\end{equation}
where the core tensor
$\mathbf{a}$ has a further decomposition in terms of the
matrices $\mathbf{B}^{(\alpha,k)}$ for all non-leaf nodes $\alpha$ and
$k=1,\ldots,\dd_\alpha(\mathbf{u})$. This decomposition can be given explicitly as
follows: For each $(k_\alpha)_{\alpha\in\hdimtree{m}}$, we define the
auxiliary expression
$$   \hat B_{(k_\alpha)_{\alpha\in\hdimtree{m}}} := 
   \prod_{\beta\in \nonleaf{m}}
      B^{(\beta,k_\beta)}_{(k_{\leftchild(\beta)},k_{\rightchild(\beta)})} \,.
$$
We now use this to give an entrywise definition of the tensor
$\hsum{\hdimtree{m}}(\{ \mathbf{B}^{(\alpha,k)}\})\in\spl{2}(\N^m)$,
for each tuple of leaf node indices
$(k_\beta)_{\beta\in\mathcal{L}(\hdimtree{m})} \in 
\N^{\#\mathcal{L}(\hdimtree{m})}$, as
\begin{multline}
\label{def:hsum}
  \Bigl( \hsum{\hdimtree{m}}\bigl(\{
\mathbf{B}^{(\alpha,k)} \colon \alpha\in\nonleaf{m},\,
k=1,\ldots,\dd_\alpha(\bu)\}\bigr) \Bigr)_{(k_\beta)_{\beta\in\mathcal{L}(\hdimtree{m})}}
\\
= \sum_{\substack{(k_\delta)_{\delta\in 
    \Ical(\hdimtree{m})} \\
  k_\delta=1,\ldots,\dd_\delta(\bu)}}
      \hat B_{(k_\delta)_{\delta\in\hdimtree{m}}}
      \,.
\end{multline}
Note that the quantity on the right hand side involves a summation over all
indices corresponding to non-leaf nodes. Since the summands depend on all
indices, this leaves precisely the indices corresponding to leaf nodes as free
parameters, as on the left hand side (recall that the index for the root of
the tree is restricted to the value $1$).
The tensor defined in \eqref{def:hsum} then equals the core tensor
$\mathbf{a}$, which is thus represented as 
\beqn
\label{eq:a-rep}
\mathbf{a} = \hsum{\hdimtree{m}}\bigl(\{\mathbf{B}^{(\alpha,k)}
 \colon \alpha\in\nonleaf{m},\,k=1,\ldots,\dd_\alpha(\bu)\}\bigr) \,.  
 \eeqn 
 This representation is illustrated explicitly for $m=4$ in Example
\ref{ex:htucker_ex4} below.

\begin{example}\label{ex:htucker_ex4}
{Consider} $m = 4$, $\hdimtree{4} = \bigl\{ \{1,2,3,4\}, \{1,2\}, \{3,4\},
\{1\},\{2\},\{3\},\{4\} \bigr\}$.
For this example, we use the abbreviation $r_\alpha :=
\dd_{\alpha}(\mathbf{u})$ and derive from \eref{htucker_m8}
  the expansion
\begin{multline*}
 \mathbf{u} =
 \sum_{k_1 = 1}^{r_1} \sum_{k_2 =
 1}^{r_2} \sum_{k_3 = 1}^{r_3} \sum_{k_4 =
 1}^{r_4}
    \sum_{k_{\{1,2\}} = 1}^{r_{\{1,2\}}} \sum_{k_{\{3,4\}} =
    1}^{r_{\{3,4\}}} B^{(\{1,2,3,4\},1)}_{(k_{\{1,2\}},k_{\{3,4\}})} \\
    \times B^{(\{1,2\},k_{\{1,2\}})}_{(k_1,k_2)} \,
    B^{(\{3,4\},k_{\{3,4\}})}_{(k_3,k_4)} \, 
    \mathbf{U}^{(1)}_{k_1} \otimes \mathbf{U}^{(2)}_{k_2} \otimes
    \mathbf{U}^{(3)}_{k_3}  \otimes \mathbf{U}^{(4)}_{k_4} \,,
\end{multline*}
that is, for the core tensor we have the decomposition
\[
a_{k_1,k_2,k_3,k_4} =  \sum_{k_{\{1,2\}} = 1}^{r_{\{1,2\}}} \sum_{k_{\{3,4\}} =
    1}^{r_{\{3,4\}}} B^{(\{1,2,3,4\},1)}_{(k_{\{1,2\}},k_{\{3,4\}})}
    B^{(\{1,2\},k_{\{1,2\}})}_{(k_1,k_2)} 
    B^{(\{3,4\},k_{\{3,4\}})}_{(k_3,k_4)} \,.
\] 
\end{example}

\begin{example}\label{ex:tt_ex4}
A tensor train (TT) representation for $m=4$ as in Example \ref{ex:htucker_ex4}
would correspond to 
$\hdimtree{4} = \bigl\{ \{1,2,3,4\}, \{1\}, \{2,3,4\}, \{2\}, \{3,4\}, \{3\},
\{4\} \bigr\}$, i.e., a degenerate instead of a balanced binary tree.
More precisely, the special case of the hierarchical Tucker format resulting from this
type of tree has also be considered under the name \emph{extended TT format} \cite{Oseledets:09-1}.
\end{example}

\subsubsection{Hierarchical Singular Value Decomposition}
 
 {For any given $\bu\in \ell_2(\nabla^d)$ the decomposition \eqref{eq:htucker_tdecomp},
 with $\mathbf{a}$ defined by \eqref{eq:a-rep},}
 can be regarded as a generalization of the HOSVD, which we shall refer to
as {\em hierarchical singular value decomposition} or $\Hcal$SVD.
{The next theorem summarizes the main properties of this
decomposition in the present setting.}  The finite-dimensional versions of the following claims 
have been established in  \cite{Grasedyck:10}.
All arguments given there carry over to the infinite-dimensional case as in the proof of
Theorem \ref{thm:hosvd_properties}.

\begin{theorem}
\label{thm:hiersvd_properties}
Let $\mathbf{u}\in \spl{2}(\nabla^d)$, where $d = d_1 + \ldots + d_m$, and let
$\hdimtree{m}$ be a dimension tree.
Then $\mathbf{u}$ can be represented in the form
\begin{equation*}
  \mathbf{u} =  \sum_{\kk{k}\in\N^m}
    a_\kk{k} \UU_\kk{k}\,, 
    \quad \mathbf{a} = \hsum{\hdimtree{m}}\bigl(\{ \mathbf{B}^{(\alpha,k)} \colon 
    \alpha \in \nonleaf{m}, \, k = 1,\ldots,
    \rank_\alpha(\mathbf{u})\}\bigr)
\end{equation*}
with $\mathbf{a} \in \spl{2}(\nabla^d)$, 
$\mathbf{B}^{(\alpha,k)}\in\spl{2}(\N \times \N)$, {defined by \eqref{def:hsum},}  for
$\alpha\in\nonleaf{m}$, 
$k\in \N$, and where the following hold:
\begin{enumerate}[{\rm (i)}]
  \item $\langle\mathbf{U}^{(i)}_k, \mathbf{U}^{(i)}_l \rangle =
  \delta_{kl}$ for $i=1,\ldots,m$ and $k,l\in\N$;
  \item $\rank_{\hroot{m}}(\bu) = 1$,
  $\norm{\mathbf{B}^{(\hroot{m},1)}} = \norm{\bu}$, and
  $\mathbf{B}^{(\hroot{m},k)} = 0$ for $k>1$;
  \item $\langle\mathbf{B}^{(\alpha,k)}, \mathbf{B}^{(\alpha,l)}\rangle =
  \delta_{kl}$ for $\alpha\in\Ical(\hdimtree{m})$ and $k,l\in\N$;
  \item for all $i\in\{1,\ldots,m\}$ we have
  $(\sigma^{(i)}_k)_{k\in\N}\in\spl{2}(\N)$, and $\sigma^{(i)}_k \geq
  \sigma^{(i)}_{k+1}\geq 0$ for all $k\in\N$;
  \item for all $i\in\{1,\ldots,m\}$  we have $a^{(i)}_{pq}
  = \bigabs{\sigma^{(i)}_p}^2 \delta_{pq}$, $1\leq p,q\leq \dd_i(\bu)$.
\end{enumerate}
\end{theorem}

{
\subsubsection{Projections}\label{sssect:proj-H}
}
{As in the case of the Tucker format} it will be important to associate suitable
orthogonal projections with a given system $\VV$ of nested orthonormal mode frames. 
{Recall that   $\rr{r} = (r_\alpha)_{\alpha\in\hdimtree{m}\setminus\{\hroot{m}\}}\in \mathcal{R}_{\mathcal{H}}$} always stands for a rank vector for the hierarchical Tucker format, {satisfying the compatibility conditions implied by Proposition \ref{prop:nestedness}}. Again $r_\alpha =\infty$ is permitted.
 We begin with introducing an analog to \eqref{T-rigid}, with a slightly more involved definition.
 The {\em hierarchical $\VV$-rigid} tensor class of rank $\rr{r}$ is given by
\beqn
\label{H-rigid}
\Hcal(\VV,\rr{r}):= \big\{\bw : \overline{\range}\, T^{(\alpha)}_{\bw} 
   \subseteq \overline{\linspan} \{  \bV^{(\alpha)}_k \colon k=1,\ldots,r_\alpha \} \,,
       \alpha\in\hdimtree{m}\setminus\{\hroot{m}\}\big\},
\eeqn
where $ T^{(\alpha)}_{\bw}$ is defined by \eqref{eq:matricization_def}. Clearly $\Hcal(\VV,\rr{r}) \subset \htucker{\rr{r}}$.

In analogy to {\eqref{oproj}} we address next a truncation of hierarchical ranks to $\rr{\tilde r}\leq \rr{r}$
for elements in $\Hcal(\VV,\rr{r})$, when $\VV$ is a given
system of orthonormal and nested mode
frames with ranks $\rr{r}$. {We assume first that $\rr{\tilde r}$ belongs also to $\mathcal{R}_{\mathcal{H}}$.}
{The main point is that} an approximation with restricted mode frames can still be realized 
through an operation represented as a 
\emph{sequence} of projections involving the given mode frames from $\VV$.  
However, the order in which these projections are applied now matters.

In a way the proof of Lemma \ref{lmm:htucker_projfull} {below} already indicates how to proceed,
namely restricting first on lower ``levels'' of the dimension tree. To make this precise
we denote by    $\hdimtree{m}^\ell$ the collection of
  elements of $\hdimtree{m}$ that have distance exactly $\ell$ to the
root (i.e., $\hdimtree{m}^0 = \{ \hroot{m} \}$, $\hdimtree{m}^1 = \{
\leftchild(\hroot{m}),\rightchild(\hroot{m}) \}$ and so forth).
Let $L$ be the maximal integer such that $\hdimtree{m}^L\neq \emptyset$.
For $\ell=1,\ldots,L$, let $\bar\cD_m^\ell:= \bigcup\{i\in\alpha:\alpha\in \cD_m^\ell\}$.
Then, given $\VV$, and abbreviating
$$
\Proj_{\VV,\alpha,\rr{\tilde r}} :=  \sum_{k=1}^{\tilde
   r_\alpha} \langle \bV^{(\alpha)}_{k} ,\cdot \rangle \bV^{(\alpha)}_k ,
$$
we define
\begin{equation*}
   P_{\VV,\ell,\rr{\tilde r}} := \Bigl(
   \bigotimes_{i\in\{1,\ldots,m\}\setminus \bar \cD_m^\ell} \id_{i}
   \Bigr) \otimes 
   \Bigl( \bigotimes_{\alpha \in \hdimtree{m}^\ell} \Proj_{\VV,\alpha,\rr{\tilde r}}
 \Bigr)
   \,,
\end{equation*}
with $\id_i$ denoting the identity operation on the $i$-th tensor mode.
Then, as observed in \cite{Grasedyck:10}, the truncation operation with mode frames $\VV$ restricted to
ranks $\rr{\tilde r}$ can be represented as
\begin{equation}
\label{eq:hsvdproj_hier}
 \Proj_{\VV,\rr{\tilde r}}  := P_{\VV,L,\rr{\tilde r}} \,
     \cdots \,P_{\VV,2,\rr{\tilde r}}\, P_{\VV,1,\rr{\tilde r}}  \,.
\end{equation}
Here the order is important because the projections $\Proj_{\VV,\alpha,\rr{\tilde r}}, \Proj_{\VV,\beta,\rr{\tilde r}}$ corresponding to
$\alpha,\beta\in\hdimtree{m}$ with $\alpha\subset \beta$ do not 
necessarily commute. Therefore a different order of projections
may in fact lead to an end result that has ranks larger than $\rr{\tilde r}$,
cf.\ \cite{Grasedyck:10}.

Specifically, given $\bu\in \ell_2(\nabla^d)$, we can choose $\VV=\UU(\bu)$ 
 provided by the $\Hcal$SVD, see \eqref{UU-H}. Hence  
$\Proj_{\UU(\bu),\rr{\tilde r}}\bu$ gives the truncation of $\bu$ based
on the $\Hcal$SVD. For this particular truncation an error estimate, in terms of
the error of best approximation with rank $\rr{\tilde r}$, is given
in Theorem \ref{thm:hier_tensor_recompression_est} {below}.

\begin{remark}\label{rem:non-nested}
By \eqref{eq:hsvdproj_hier}, we have a representation of
$\tilde\bu:=\Proj_{\UU(\bu),\rr{\tilde r}} \bu$ in terms of a sequence of
non-commuting orthogonal projections.
{When $\rr{\tilde r}\leq \rr{r}$ does not belong to $\mathcal{R}_{\mathcal{H}}$ the operator
defined by \eqref{eq:hsvdproj_hier} is still a projection which, however, modifies 
the mode frames for those nodes $\alpha\in \mathcal{N}(\cD_m)$ for which the rank compatibility conditions are violated. The resulting projected mode frames are then nested, that is,} $\tilde\bu$ may again be represented in terms of the orthonormal and
nested mode frames $\tilde\UU := \UU(\tilde\bu)$. 
\end{remark}

The situation simplifies if we consider the projection to a fixed \emph{nested} system of
mode frames, without a further truncation of ranks that could entail non-nestedness.

\begin{lemma}
\label{lmm:htucker_projfull}
Let $\VV$ be a family of orthonormal and nested
hierarchical mode frames with ranks $\rr{r}$. 
Then there exists a linear projection $\Proj_{\VV} \colon \spl{2}(\nabla^d) \to \Hcal(\VV,\rr{r})$
such that the unique best approximation in $\Hcal(\VV,\rr{r})$ of any $\bu\in\spl{2}(\nabla^d)$
is given by $\Proj_{\VV}\bu$, that is,
$$  \norm{\bu - \Proj_{\VV} \bu} =
\min_{\bw\in\htucker{\VV,\rr{r}}}\norm{\bu - \bw} \,.  $$
\end{lemma}
 
\begin{proof}
The sought projection is given by $\Proj_\VV = P_{\VV,1,\rr{r}}$, since
$$
P_{\VV,L,\rr{r}} \,
     \cdots \,P_{\VV,2,\rr{r}}\, P_{\VV,1,\rr{r}}  = P_{\VV,1,\rr{r}}
$$
holds as a consequence of the nestedness property.
\end{proof}

{
\subsubsection{Best approximation}\label{ssec:best-appr}
}

{
In analogy to \eqref{eq:tucker_err}, we define the   error estimate
\begin{equation}
\label{eq:htucker_err}
 \err_\rr{\tilde r}(\bu) = \err^\Hcal_\rr{\tilde r}(\bu)
   := \Bigl( \sum_{\alpha}
    \sum_{k = \tilde r_{\alpha}+ 1}^{\dd_\alpha(\mathbf{u})}
    \bigabs{\sigma^{(\alpha)}_{k}}^2 \Bigr)^{\frac12} \,.
\end{equation}
Here the sum over $\alpha$ extends over $\hdimtree{m}\setminus \{ \hroot{m},\rightchild(\hroot{m})  \}$ if $\tilde r_{\leftchild(\hroot{m})} \leq \tilde r_{\rightchild(\hroot{m})}$, and otherwise over $\hdimtree{m}\setminus \{ \hroot{m},\leftchild(\hroot{m})  \}$.}
We then have the following analogue of Corollary
\ref{cor:tensor_recompression_est}, see \cite{Grasedyck:10}.
\smallskip

\begin{theorem}
\label{thm:hier_tensor_recompression_est}
{For a given $\mathbf{u} \in \spl{2}(\nabla^d)$ let $\UU^{\Hcal}(\bu)=\UU(\bu)$ the hierarchical orthonormal 
system of mode frames generated by the} $\Hcal$SVD of $\mathbf{u}$ as in
Theorem \ref{thm:hiersvd_properties}. 
{Then for hierarchical ranks $\rr{\tilde r} = (\tilde
r_\alpha)\in \Rcal_\Hcal$,} we have{
\begin{equation*}
 \norm{\mathbf{u} - \Proj_{\UU(\bu),\rr{\tilde r}}\bu} \leq \err^\Hcal_\rr{\tilde r}(\bu)
  \leq \sqrt{2 m - 3} \, \inf
 \bigl\{\norm{\mathbf{u} - \mathbf{v}} \colon  \mathbf{v} \in
 \Hcal(\rr{\tilde r}) \bigr\} \,.
\end{equation*}}
\end{theorem}

\begin{corollary}\label{cor:htucker_projbest}
For $\bu\in\spl{2}(\nabla^d)$ and $\rr{r} = (r_\alpha)_{\alpha\in\hdimtree{m}} {\in \Rcal_\Hcal}$
with $0\leq r_\alpha\leq \rank_\alpha(\bu)$, there exist orthonormal and
nested hierarchical mode frames $\bar\UU(\bu, \rr{r})$ such that
$$  \norm{\bu - \Proj_{\bar\UU(\bu, \rr{r})} \bu} =
\min_{\bw\in\htucker{\rr{r}}}\norm{\bu - \bw}   $$
with $\Proj_{\bar\UU(\bu, \rr{r})}$ as in Lemma \ref{lmm:htucker_projfull}.
\end{corollary}

\begin{proof}
By Theorem \ref{thm:htucker_bestapprox}, a best approximation of hierarchical
ranks $\rr{r}$ for $\bu$, $$  \bar\bu \in \argmin\{ \norm{\bu - \bv} \colon
\rank_\alpha(\bu) \leq r_\alpha \} \,, $$ exists.
Defining $\bar\UU(\bu, \rr{r}):= \UU(\bar\bu)$  
as the nested and orthonormal mode frames
for $\bar\bu$, given by the $\Hcal$SVD, we obtain the assertion with Lemma
\ref{lmm:htucker_projfull}.
\end{proof}

\begin{remark}\label{rmrk:hiersvd_complexity}
{Suppose that, in analogy to Remark \ref{rmrk:hosvd_complexity}, a
  compactly supported vector $\bu$ on $\nabla^d$ is given in
a possibly redundant hierarchical representation
$$ 
  \mathbf{u} = \sum_{\kk{k}\in\KK{m}(\rr{\tilde r})} \tilde a_\kk{k} 
   \bigotimes_{i=1}^m \mathbf{\tilde U}^{(i)}_{k_i}\,,
   \quad  \tilde{\mathbf{a}} = \hsum{\hdimtree{m}}(\{
    \mathbf{\tilde B}^{(\alpha,k_\alpha)}\}) \,,
$$
where  the summations in the expansion of $\tilde{\mathbf{a}}$
range over $k_\alpha = 1,\ldots,\tilde r_\alpha$ for each $\alpha$,
and where the vectors $\mathbf{\tilde U}^{(i)}_{k}$, $k=1,\ldots,\tilde r_i$,
and $\mathbf{\tilde B}^{(\alpha,k)}$, $k = 1,\ldots,\tilde r_\alpha$,
may be linearly dependent. Employing standard linear algebra procedures, 
an $\Hcal$SVD of $\mathbf{u}$
can be computed from such a representation,
using a number of operations that can be 
estimated by
\begin{equation}  C m \,\bigl(\max_{\alpha\in\hdimtree{m}\setminus\{\hroot{m}\} }\tilde
r_\alpha\bigr)^4 + C \bigl(\max_i{\tilde r_i}\bigr)^2 \sum_{i=1}^m
\#\supp_i(\bu),
\end{equation}
where $C>0 $ is a fixed constant, cf.\ \cite[Lemma 4.9]{Grasedyck:10}.}
\end{remark}

\section{Recompression and Coarsening}\label{sect:compress}
As explained in \S\ref{ssect:preview}, iterations of the form \eref{practicaliter} provide
updates $\bv = \bu_k +\omega(\bbf - \bA \bu_k)$
which differ from the unknown $\bu$ by some known tolerance. However, even when using a ``tensor-friendly'' structure
of the operator $\bA$ or a known ``tensor-sparsity'' of the data $\bbf$, the arithmetic operations 
leading to the update $\bv$ do not give any clue as to whether the resulting ranks are close to minimal.
Hence, one needs a mechanism that realizes a {\em subspace correction} leading to tensor
representations with ranks at least close to minimal ones. This consists in deriving from the {\em known}
$\bv$ a {\em near best} approximation to the {\em unknown} $\bu$ where the notion of near best in terms of ranks is made precise below.
Specifically, suppose that  $\bv\in \e2$ is an approximation of $\bu\in
\spl{2}(\nabla^d)$ which for some $\eta>0$ satisfies
\begin{equation}
\label{c1}
\norm{\bu - \bv}_{\e2}\leq \eta.
\end{equation}
We shall show next how to derive from $\bv$ a near-minimal rank tensor approximation to
$\bu$.
Based on our preparations in {\S\ref{sect:prelim}}, the following developments apply 
to both formats $\Fcal\in \{\Tcal,\Hcal\}$, in fact, to any format $\Fcal$ with associated mode frame systems $\UU=\UU^{\Fcal}$
(see \eqref{UU-T}, \eqref{UU-H}) for which one can formulate suitable projections $\Proj_\VV^{\Fcal}, \Proj_{\VV,\rr{\tilde r}}^{\Fcal}$
with analogous properties. Accordingly, 
\begin{equation}
\label{eq:rankvectorsets}
{ \Rank = \Rank_{\Fcal}\,,\quad\Fcal\in\{\Tcal,\Hcal\}  }
\end{equation}
denotes the respective set of {admissible} rank vectors
{$\Rank_{\Tcal}$, $\Rank_{\Hcal}$, defined in \eqref{def:R-T}, \eqref{def:R-H}, respectively.}
A crucial role in what follows is played by the following immediate consequence of Corollaries \ref{cor:tucker_projbest}, \ref{cor:htucker_projbest}
combined with Corollary \ref{cor:tensor_recompression_est} and Theorem \ref{thm:hier_tensor_recompression_est}.

\begin{remark}
\label{rem:near-best}
Let for a given $\bv\in \ell_2(\nabla^d)$ the mode frame system $\UU(\bv)$ be  either $\UU^\Tcal(\bv)$
or $\UU^\Hcal(\bv)$. Then, for any  rank vector $\rr{r}\leq \rank(\bv)$, $\rr{r}\in \Rank$, one has
\begin{equation}
\label{UU-near-best}
{  \norm{\mathbf{v} -  \Psvd{\mathbf{v}}{\rr{r}} \mathbf{v}} 
\leq {  \err_\rr{r} (\bv) }
\leq \constsvd  \norm{\mathbf{v} - \Pbest{\mathbf{v}}{\rr{r}} \mathbf{v}} =  \constsvd \min_{\rank(\mathbf{w})\leq \rr{r}} \norm{\mathbf{u} - \mathbf{w}} ,   }
\end{equation}
where $\constsvd =\sqrt{m}$ when $\Fcal =\Tcal$, and $\constsvd =\sqrt{2m-3}$ when $\Fcal =\Hcal$.
 \end{remark}

{As mentioned earlier, for $\Fcal = \Hcal$ the above notions depend on the dimension tree $\cD_m$. Since $\cD_m$ is fixed
we dispense with a corresponding notational reference.} 
\subsection{Tensor Recompression}\label{ssect:recompress}

  Given $\mathbf{u}\in\spl{2}(\nabla^d)$, in what follows
by  $\UU(\bu)$ we either mean  $\UU^{\Tcal}(\bu)$ or  $\UU^{\Hcal}(\bu)$, see \eqref{UU-T}, \eqref{UU-H}.

We introduce next two notions of ``minimal ranks'' {$\rsvd(\bu,\eta), \rbest(\bu, \eta)$} for a given target accuracy
$\eta$, one for the specific
mode frame system $\UU(\bu)$ provided by either HOSVD or $\Hcal$SVD, and one for 
the respective best mode frame systems.

\begin{definition}
For each $\eta>0$ we choose $\rsvd(\mathbf{u}, \eta)\in \Rank$
such that
\begin{equation*}
{  \err_{\rsvd(\mathbf{u}, \eta)}(\bu) }  \leq \eta \,,
\end{equation*}
{and hence $ \norm{\mathbf{u} - \Psvd{\mathbf{u}}{\rsvd(\mathbf{u}, \eta)} \mathbf{u} } \leq \eta$,}
with minimal $\abs{\rsvd(\mathbf{u}, \eta)}_\infty$, that is,
\begin{equation*}
\rsvd(\bu,\eta) \in \argmin\bigl\{|\rr{r}|_\infty: \rr{r}\in \Rank,\;  {  \err_{\rsvd(\mathbf{u}, \eta)}(\bu) } \leq \eta\bigr\}  \,.
\end{equation*}
Similarly,
for each $\eta>0$ we choose $\rbest(\mathbf{u}, \eta)\in
\Rank$ such that
\begin{equation*}
  \norm{\mathbf{u} - \Pbest{\mathbf{u}}{\rbest(\mathbf{u}, \eta)} \mathbf{u} }
  \leq  \eta,
\end{equation*}
with minimal $\abs{\rbest(\mathbf{u}, \eta)}_\infty$, that is (see Corollary \ref{cor:tucker_projbest} 
and Remark \ref{rem:near-best}),
\begin{equation}
\label{c3}
\rbest(\bu,\eta) \in \argmin\,\bigl\{|\rr{r}|_\infty : \rr{r}\in \Rank,\,\,
\exists\,\, \bw\in {\Fcal(\rr{r})},\,\,
\|\bu -\bw\|_{\e2}\leq \eta\bigr\} \,.
\end{equation}
\end{definition}

Recall that the projections $\Psvd{\bv}{\rr{r}}=\Psvd{\bv}{\rr{r}}^{\Fcal}$ to $\Fcal(\rr{r})$ are given either by \eqref{P-T} or \eqref{eq:hsvdproj_hier} when $\Fcal \in \{\Tcal,\Hcal\}$, respectively.  
 In both cases, they will be used to define  computable coarsening operators for any given 
$\bv$ (of finite support in $\nabla^d$).
In fact, setting
\begin{equation}\label{eq:tucker_recomp_def}
\hatPsvd{\eta} \bv:= \Psvd{\bv}{\rsvd(\bv,\eta)}\bv\,,
\end{equation}
we have by definition
\begin{equation}
\label{propC}
\norm{\bv- \hatPsvd{\eta}\bv}\leq  \err_{\rsvd(\bv,\eta)} (\bv) \leq \eta,\qquad
\abs{\rank(\hatPsvd{\eta}\bv)}_\infty = \abs{\rsvd(\bv,\eta)}_\infty.
\end{equation}

\begin{lemma}
\label{lmmc1}
Fix any $\alpha >0$. 
For any $\bu, \bv,\eta$ satisfying \eqref{c1}, i.e.~$\norm{\bu -\bv}\leq \eta$,
one has
\begin{equation}
\label{c6}
\norm{ \bu - \hatPsvd{\constsvd(1+\alpha)\eta}\bv } \leq (1+
\constsvd(1+\alpha))\eta
\end{equation}
while
\begin{equation}
\label{c5}
\abs{\rank(\hatPsvd{\constsvd(1+\alpha)\eta}\bv)}_\infty = 
\abs{\rsvd(\bv,\constsvd(1+\alpha)\eta) }_\infty
\leq \abs{\rbest(\bu,\alpha\eta)}_\infty\,.
\end{equation}
 \end{lemma}
{In other words,} {the ranks of $\hatPsvd{\constsvd(1+\alpha)\eta}\bv$ are bounded by the minimum ranks required to realize a somewhat higher accuracy.}
\begin{proof} 
Bearing Remark \ref{rem:near-best} in mind,
given $\bu$,  one has for the projection $\Pbest{\bu}{\rbest(\bu,\alpha\eta)}$
\begin{multline}
\label{c4}
\norm{ \bv - \Pbest{\bu}{\rbest(\bu,\alpha\eta)}\bv } 
\leq \norm{ (\id -
\Pbest{\bu}{\rbest(\bu,\alpha\eta)})(\bv -\bu)}  \\
 + \norm{\bu -
\Pbest{\bu}{\rbest(\bu,\alpha\eta)}\bu}  
\leq (1+\alpha)\eta.
\end{multline}
On the other hand, we know that for any $\rr{r}\in \Rank$, 
$$
\norm{ \bv - \Psvd{\bv}{\rr{r}}\bv } 
{ \leq \err_\rr{r}(\bv) }
\leq \constsvd \, \inf_{\bw\in \Fcal(\rr{r})}
 \norm{\bv -\bw }\,, $$
so that, by \eqref{c4}, for $\rr{r}=\rbest(\bu,\alpha\eta)$ we have
$$
\norm{ \bv - \Psvd{\bv}{\rbest(\bu,\alpha\eta)}\bv } 
{  \leq \err_{\rbest(\bu,\alpha\eta)}(\bv) }
\leq \constsvd (1+\alpha)\eta\,. $$
{Since, by definition, $\abs{\rank(\hatPsvd{\constsvd(1+\alpha)\eta}\bv)}_\infty$ is minimal to achieve the accuracy bound $\constsvd(1+\alpha)\eta$, \eqref{c5} follows.} Estimate \eqref{c6} follows by triangle inequality.
\end{proof}

Thus, appropriately coarsening $\bv$ yields an approximation to $\bu$ of still
the same quality up to a fixed (dimension-dependent) constant, where the rank of 
this new approximation is bounded by a minimal rank of a {\em best} Tucker or hierarchical Tucker approximation to $\bu$ for somewhat higher accuracy.

Let us reinterpret this in terms of minimal ranks, i.e., for $r\in\N_0$ and $\Fcal\in \{\Tcal, \Hcal\}$, 
let $$
{\sigma_{r}(\bv) = \sigma_{r,\Fcal}(\bv):= \inf\,\bigl\{\norm{ \bv- \bw} \,:\; \bw \in
\Fcal(\rr{r}) \text{ with $\rr{r}\in\Rank$, $\abs{\rr{r}}_\infty \leq r$} \}\,. }
$$
We now consider corresponding approximation classes. 

\begin{definition}
We   call a positive, strictly
increasing $\ga = \bigl(\ga(n)\bigr)_{n\in \N_0}$ with $\ga(0)=1$
and $\ga(n)\to\infty$, as $n\to\infty$,
a \emph{growth sequence}.
For a given growth sequence $\ga$, we define
$$
\Acal(\ga)= \AF{\ga}:= { \bigl\{\bv\in {\e2} : \sup_{r\in\N_0} 
\ga({r})\,
\sigma_{r,\Fcal}(\bv)=:\abs{\bv}_{\AF{\ga}}{ <\infty }\bigr\} }
$$ 
and
$\norm{\bv}_{\AF{\ga}}:= \norm{\bv} + \abs{\bv}_{\AF{\ga}}$.
We call the growth sequence $\ga$ \emph{admissible} if
$$
\garatio:= \sup_{n\in\N} \ga(n)/\ga(n-1)<\infty\,,
$$
which corresponds to a restriction to at most exponential growth.
\end{definition}

In the particular case when $\ga(n)\sim n^{s}$ for some $s >0$, 
{$\norm{\bv}_{\AF{\ga}}:= \norm{\bv} + \abs{\bv}_{\AF{\ga}}$ is a quasi-norm and
$\AF{\ga}$ is a linear space.}

{
\begin{remark}
\label{rem:howtoread}
For the subsequent developments it will be helpful to keep the following 
way of reading $\bv\in \AF{\ga}$ in mind: a given target accuracy $\varepsilon$ can
be realized at the expense of
ranks of the size $\gamma^{-1}(\abs{\bv}_{\AF{\ga}}/\varepsilon)$ so that a rank bound
of the form $\gamma^{-1}(C\abs{\bv}_{\AF{\ga}}/\varepsilon)$, where $C$ is any constant, marks
a near-optimal performance. 
\end{remark}
}

\begin{theorem}
\label{propc1}
Let $\constsvd$ be as in Remark \ref{lmmc1}, and let $\alpha > 0$.
Assume that $\bu\in \AF{\ga}$ and that $\bv\in \e2$
satisfies $\norm{\bu-\bv} \leq \eta$.
Then, defining $\bw_\eta:= \hatPsvd{\constsvd(1+\alpha)\eta}\bv$, one has
\begin{equation}
\label{c8}
\abs{\rank(\bw_\eta)}_\infty 
  \leq  \ga^{-1}\big(\garatio\norm{\bu}_{\AF{\ga}}/(\alpha\eta)\big),\quad
\norm{\bu - \bw_\eta}\leq (1+ \constsvd(1+\alpha))\eta,
\end{equation}
and
\begin{equation}
\label{c9}
\norm{\bw_\eta}_{\AF{\ga}}\leq C 
\norm{\bu}_{\AF{\ga}},\quad \eta >0,
\end{equation}
where $C=\alpha^{-1}(1+\constsvd(1+\alpha)) + 1$.
\end{theorem}
\begin{proof}
The second relation in \eqref{c8} has already been shown in {Lemma \ref{lmmc1}}.
We also know from \eqref{c5} that $\abs{\rank(\bw_\eta)}_\infty\leq 
\abs{\rbest(\bu,\alpha\eta)}_\infty$. 
Thus the first relation in \eqref{c8} is clear if
$\abs{\rbest(\bu,\alpha\eta)}_\infty = 0$.
Assume that $\abs{\rbest(\bu,\alpha\eta)}_\infty>1$. Then for
{$r' :=\abs{\rbest(\bu,\alpha\eta)}_\infty -1$}, by definition of
$\abs{\cdot}_{\AF{\ga}}$ we have
\begin{equation}
\label{c10}
{\abs{\bu}_{\AF{\ga}} \geq \ga(\abs{r'}_\infty)\sigma_{
r',\Fcal}(\bu) \geq
\ga(\abs{r'}_\infty)\,\alpha\eta \geq \garatio^{-1}
\ga(\abs{\rbest(\bu,\alpha\eta)}_\infty)\,\alpha\eta.}
\end{equation}
Also, when $\abs{\rbest(\bu,\alpha\eta)}_\infty=1$, we have
$$
\sigma_{0,\Tcal}(\bu)=\norm{\bu}>\alpha\eta =\ga(0)\,\alpha\eta \geq
\garatio^{-1}\ga(\abs{\rbest(\bu,\alpha\eta)}_\infty)\,\alpha\eta \,. $$
Therefore
$$
\abs{\rbest(\bu,\alpha\eta)}_\infty  \leq \ga^{-1}\big(\garatio
\norm{\bu}_{\AF{\ga}}/(\alpha\eta)\big), $$
which is the first relation in \eqref{c8}.

As for the remaining claim, we need to estimate
{$\ga(r)\sigma_{r,\Fcal}(\bw_\eta)$ for $r\in\N_0$. Whenever
$r \geq \abs{\rbest(\bu,\alpha\eta)}_\infty$ we have, by
\eqref{c5}, $\sigma_{r,\Fcal}(\bw_\eta)=0$. It thus suffices to consider
$r < \abs{\rbest(\bu,\alpha\eta)}_\infty$.} By \eqref{c6},
\begin{align*}
{\inf_{\rr{r}\in\Rank\colon \abs{\rr{r}}_\infty\leq r}\norm{\bw_\eta - \Pbest{\bu}{\rr{r}}\bu } }
 & \leq \norm{\bw_\eta - \bu}  + { \inf_{\rr{r}\in\Rank\colon \abs{\rr{r}}_\infty\leq r} \norm{\bu -\Pbest{\bu}{\rr{r}}\bu} } \\
 & \leq (1+\constsvd(1+\alpha))\eta + { \sigma_{r,\Fcal}(\bu). }
\end{align*}
{Since for $r < \abs{\rbest(\bu,\alpha\eta)}_\infty$ we have  $\sigma_{r,\Fcal}(\bu)>\alpha\eta$, while
$\sigma_{\abs{\rbest(\bu,\alpha\eta)}_\infty,\Fcal}(\bu) \leq \alpha\eta$,  
we conclude that
\begin{align*}
\ga(r)\,\sigma_{r,\Fcal}(\bw_\eta)
 &\leq \ga(r) \frac{(1+\constsvd(1+\alpha))\alpha\eta}{\alpha}
 + \ga(r)\,\sigma_{r,\Fcal}(\bu) \nonumber\\
 &\leq \left( \frac{1+\constsvd(1+\alpha)}{\alpha} + 1\right)
   \ga(r)\sigma_{r,\Fcal}(\bu)\nonumber\\
 &\leq  \left( \frac{1+\constsvd(1+\alpha)}{\alpha} +
    1\right)\abs{\bu}_{\AF{\ga}},
\end{align*}
which shows \eqref{c9}.}
\end{proof}

\subsection{Coarsening of Mode Frames}\label{sec:modeframe_coarsening}
We now turn to a second type of operation for reducing the complexity of given
coefficient sequences in tensor representation, an operation that coarsens mode frames
by discarding basis indices whose contribution is negligible. We shall use the following standard notions
for best $N$-term approximations.
\begin{definition}
For $\hat d \in \N$ and $\Lambda \subset \nabla^{\hat d}$, we define the
restrictions{
\begin{equation*}
  \Restr{\Lambda} \mathbf{v} := \mathbf{v} \odot \chi_\Lambda\,,\quad \mathbf{v}\in
  \spl{2}(\nabla^{\hat d}) \,,
\end{equation*} 
where $\odot$ denotes the Hadamard (elementwise) product.}
The compressibility of $\bv$ can again be described through   {\em approximation classes}.
For $s >0$, we denote by $\As(\nabla^{\hat d})$ the set of
$\mathbf{v}\in\spl{2}(\nabla^{\hat d})$ such that
\begin{equation*}
  \norm{\mathbf{v}}_{\As(\nabla^{\hat d})} := \sup_{N\in\N_0} (N+1)^s
  {\inf_{\substack{\Lambda\subset\nabla^{\hat d}\\
  \#\Lambda\leq N}} \norm{\mathbf{v} - \Restr{\Lambda} \mathbf{v}} } < \infty \,.
\end{equation*}
Endowed with this (quasi-)norm, $\As(\nabla^{\hat d})$ becomes a (quasi-)Banach space.
When no confusion can arise, we shall suppress the index set dependence 
and write $\As = \As(\nabla^{\hat d})$.
\end{definition}

\begin{remark}
\label{rem:howtoread2}
{The same comment as in Remark \ref{rem:howtoread} applies. Thinking of the growth sequence to be
$\gamma_s(n)=(n+1)^s$, realizing an accuracy $\varepsilon$ at the expense of $(C  \norm{\mathbf{v}}_{\As(\nabla^{\hat d})}
/\varepsilon)^{1/s}$ terms, where $C$ is a constant independent of $\varepsilon$, signifies 
an ``optimal work-accuracy balance" over the class $\As(\nabla^{\hat d})$.}
\end{remark}

We deliberately restrict the discussion to polynomial decay rates here since this corresponds to finite
Sobolev or Besov regularity. However, with appropriate modifications,
the subsequent considerations can be adapted also to approximation classes corresponding to
more general growth sequences.

\subsubsection{Tensor Contractions}

Searching through a sequence $\bu\in \ell_2(\nabla^d)$ (of finite support)
would suffer from the curse of dimensionality.
Being content with {\em near best} $N$-term approximations one can get around this
by introducing, for each given $\bu\in \ell_2(\nabla^d)$,
the following quantities formed from {certain contractions of the tensor $\bu \otimes \bu$ which are given by ${\rm diag}(T^{(i)}_\bu(T^{(i)}_\bu)^*)$}.
 
\begin{definition}
\label{def:contractions}
Let $\bu\in\spl{2}(\nabla^d)$. For $i\in\zinterval{1}{m}$ we define,
using the notation \eqref{eq:delentry},
\begin{equation*}
 \pi^{(i)}(\mathbf{u}) = \bigl( \pi^{(i)}_{\nu_i} (\mathbf{u})
 \bigr)_{\nu_i\in\nabla^{d_i}} 
   :=\biggl(  \Bigl(\sum_{\check{\nu}_i} \abs{u_{\nu}}^2 \Bigr)^{\frac{1}{2}}  \biggr)_{\nu_i \in\nabla^{d_i}} \,.
\end{equation*}
\end{definition}

With a slight abuse of terminology, we shall refer to these $\pi^{(i)}(\cdot)$ simply as \emph{contractions}. 
Their direct computation would involve high-dimensional summations over the index sets $\nabla^{d-d_i}$.
However, the following observations show how this can be avoided.
This makes essential use of the particular orthogonality properties of the tensor formats.
\begin{proposition}
Let $\bu\in\spl{2}(\nabla^d)$.
\begin{enumerate}[{\rm (i)}] 
\item We have $\norm{\bu} = \norm{\pi^{(i)}(\bu)}$, $i=1,\ldots,m$.
\item Let $\Lambda^{(i)}\subseteq \nabla^{d_i}$, then
\begin{equation}
\label{telesc}
 \norm{\bu - \Restr{\Lambda^{(1)} \times\cdots\times\Lambda^{(m)}} \bu} \leq
    \Bigl(\sum_{i=1}^m \sum_{\nu\in\nabla^{d_i}\setminus\Lambda^{(i)}}
    \abs{\pi^{(i)}_\nu (\bu)}^2 \Bigr)^{\frac{1}{2}} \,.
\end{equation}
\item Let in addition $\mathbf{U}^{(i)}$ and $\mathbf{a}$ be mode
frames and core tensor, respectively, as in Theorems \ref{thm:hosvd_properties}
or \ref{thm:hiersvd_properties},
and let $(\sigma^{(i)}_k)$ be the corresponding sequences of mode-$i$ singular
values. Then
\begin{equation}
\label{contract-mode}
 \pi^{(i)}_\nu (\mathbf{u}) = \Bigl( \sum_{k}
 \bigabs{\mathbf{U}^{(i)}_{\nu, k}}^2 \bigabs{\sigma^{(i)}_{k}}^2
 \Bigr)^{\frac{1}{2}} \,,\quad \nu\in\nabla^{d_i} \,.
\end{equation}
\end{enumerate}
\end{proposition}

\begin{proof}
Property (i) is clear, and (iii) is a simple consequence of the orthogonality
properties of mode frames and core tensor stated in Theorems \ref{thm:hosvd_properties}
and \ref{thm:hiersvd_properties}. Abbreviating $\mathbf{\tilde u}:=
 \Restr{\Lambda^{(1)} \times\cdots\times\Lambda^{(m)}} \bu$,
property (ii) follows, in view of  (i), from
\begin{eqnarray*} 
\norm{\mathbf{\tilde u}- \mathbf{u}}^2 
  & \leq &\norm{\mathbf{u} - \Restr{\Lambda^{(1)} \times
 \nabla^{d_2}\times\cdots\times \nabla^{d_m}} \mathbf{u}}^2 +
 \ldots + \norm{\mathbf{u} - \Restr{\nabla^{d_1} \times \cdots\times
 \nabla^{d_{m-1}} \times \Lambda^{(m)}} \mathbf{u} }^2\nonumber   \\
  & = &\sum_{i=1}^m
\sum_{\nu\in \nabla^{d_i} \setminus \Lambda^{(i)}}
 \bigabs{\pi^{(i)}_{\nu}(\mathbf{u})}^2 \,. \qedhere
\end{eqnarray*}
\end{proof}

The following subadditivity property is an
immediate consequence of the triangle inequality.

\begin{proposition}\label{prp:contraction_subadd}
Let $N\in\N$ and $\bu_n \in \spl{2}(\nabla^d)$, $n=1,\ldots,N$.
Then for each $i$ and each $\nu\in\nabla^{d_i}$, we have
\begin{equation*}
  \pi^{(i)}_\nu\Bigl(\sum_{n=1}^N \bu_n\Bigr) \leq
     \sum_{n=1}^N \pi^{(i)}_\nu (\bu_n).
\end{equation*}
\end{proposition}

Relation \eref{contract-mode} allows us to realize (in practice, of course, for finite ranks $\rank(\bu)$
and finitely supported mode frames $\bU^{(i)}$) best $N$-term approximations of the contractions
$\pi^{(i)}(\bu)$ through those of the mode frames $\bU^{(i)}_k$. Moreover, expressing coarsening errors 
in terms of tails of contraction sequences requires finding good Cartesian index sets. To see how to
determine them {consider a non-increasing rearrangement
\beqn
\label{dim-sorting}
\pi^{(i_1)}_{ \nu^{i_1,1}}(\bu)\geq \pi^{(i_2)}_{ \nu^{i_2,2}}(\bu)\geq \cdots \geq \pi^{(i_j)}_{ \nu^{i_j,j}}(\bu)\geq \cdots,\quad \nu^{i_j,j}\in \nabla^{d_{i_j}}, 
\eeqn
of the entire set of contractions for all tensor modes,}
$$
{ \bigl\{\pi^{(i)}_\nu(\bu) : \nu\in \nabla^{d_i},\, i=1,\ldots,m \bigr\}.  }
$$
Next, retaining only the $N$ largest from the latter total ordering {\eqref{dim-sorting}} and redistributing them to the {respective}
dimension bins
\beqn
\label{dim-bin}
\Lambda^{(i)}(\bu;N):= \bigl\{\nu^{i_j,j}: i_j= i,\, j=1,\ldots, N\bigr\}, \quad i=1,\ldots, m,
\eeqn
the product set
\beqn
\label{dim-bin-product}
\Lambda(\mathbf{u};N) := \bigtimes_{i=1}^m
\Lambda^{(i)}(\mathbf{u};N)
\eeqn
can be obtained at a cost that is roughly $m$ times the analogous low-dimensional cost.
By construction, one has
\beqn
\label{Lambda-N}
\sum_{i=1}^m \# \Lambda^{(i)}(\mathbf{u};N) \leq N
\eeqn
and
\beqn
\label{min-contract}
\sum_{i=1}^m\sum_{\nu\in \nabla^{d_i}\setminus \Lambda^{(i)}(\bu;N)} |\pi^{(i)}_\nu(\bu)|^2
= \min_{\hat \Lambda}\Big\{\sum_{i=1}^m\sum_{\nu\in \nabla^{d_i}\setminus \hat \Lambda^{(i)} } |\pi^{(i)}_\nu(\bu)|^2\Big\},
\eeqn
where $\hat\Lambda $ ranges over all product sets
$ \bigtimes_{i=1}^m \hat\Lambda^{(i)}$ with $\sum_{i=1}^m \#\hat
\Lambda^{(i)} \leq N$.

\begin{proposition}
\label{prp:tensor_coarsening_est}
For any  $\mathbf{u} \in \spl{2}(\nabla^d)$ one has
\begin{equation}
\label{eq:tensor_coarsening_errest}  
{ \norm{\mathbf{u} - \Restr{\Lambda(\mathbf{u};N)}\bu} \leq \Bigl( \sum_{i=1}^m\sum_{\nu\in \nabla^{d_i}\setminus \Lambda^{(i)}(\bu;N)}
\bigabs{\pi^{(i)}_\nu(\mathbf{u})}  \Bigr)^{\frac12} =: \cerr_N(\bu) \,,  }
\end{equation}
and for any $\hat\Lambda =
\bigtimes_{i=1}^m \hat\Lambda^{(i)}$ with $\Lambda^{(i)} \subset\nabla^{d_i}$
satisfying $\sum_{i=1}^m \#\hat
\Lambda^{(i)} \leq N$, one has
\begin{equation}
\label{eq:tensor_coarsen_qopt}
 \norm{\mathbf{u} - \Restr{\Lambda(\mathbf{u};N)}\bu} \leq \cerr_N(\bu) \leq \sqrt{m} \norm{\mathbf{u} -
 \Restr{\hat\Lambda}\mathbf{u}} \,.
\end{equation}
\end{proposition}

\begin{proof}
The bound \eqref{eq:tensor_coarsening_errest} is immediate from \eqref{telesc}.
 Let now $\hat\Lambda$ be as in the hypothesis, then  one obtains by \eqref{eq:tensor_coarsening_errest} and
 \eqref{min-contract},
\begin{align*} 
 \norm{\mathbf{u} - \Restr{\Lambda(\mathbf{u};N)}\bu}^2 &  \leq
 \sum_{i=1}^m \sum_{\nu\in\nabla^{d_i}\setminus \hat\Lambda^{(i)}}
 \bigabs{\pi^{(i)}_{\nu}(\mathbf{u})}^2 
= 
\norm{\mathbf{u} - \Restr{\hat\Lambda^{(1)} \times
 \nabla^{d_2}\times\cdots\times \nabla^{d_m}} \mathbf{u}}^2 +
 \ldots \\
 & +
\norm{\mathbf{u} - \Restr{\nabla^{d_1} \times \cdots\times \nabla^{d_{m-1}}
\times \hat\Lambda^{(m)}} \mathbf{u} }^2 
 \leq m \norm{\mathbf{u} - \Restr{\hat\Lambda} \mathbf{u} }^2\,.
\qedhere
\end{align*}
\end{proof}

Note that the sorting \eqref{dim-sorting} used in the construction of \eqref{dim-bin-product}
can be replaced by a quasi-sorting by binary binning; we shall return to this point in
the proof of Remark \ref{rmrk:opapprox_tucker_ops_est}.
With the above preparations at hand we define the {\em coarsening operator}
\begin{equation}
\label{coarse-op}
  \Cctr_{\mathbf{u}, N} \mathbf{v} := \Restr{\Lambda(\mathbf{u};N)} \mathbf{v} 
  \,,\quad \mathbf{v} \in \spl{2}(\nabla^d)   \,.
\end{equation}
While $ \Cctr_{\mathbf{u}, N}$ is computationally feasible, it is not necessarily strictly optimal. 
However, we remark {that  
 for each} $N\in\N$, there exists $\bar\Lambda(\mathbf{u};N) =
\bigtimes_i \bar\Lambda^{(i)}(\mathbf{u};N)$ such that the {\em best tensor coarsening} operator
\begin{equation}
\label{C-best}
  \Cbest_{\mathbf{u},N} \mathbf{v} := 
   \Restr{\bar\Lambda(\mathbf{u};N)} \mathbf{v} \,, \quad \mathbf{v} \in
  \spl{2}(\nabla^d)\,,
\end{equation}
realizes
\begin{equation}
\label{C-best2}
  \norm{\mathbf{u} - \Cbest_{\mathbf{u},N} \mathbf{u} }
    = \min_{\sum_i \# \supp_i(\mathbf{w})\leq N} \norm{\mathbf{u} - \mathbf{w}}
    \,.
\end{equation}

The next observation is that the contractions are stable under the projections $\Proj_{\UU^\Fcal (\bu),\rr{r}}$,
$\Fcal\in \{\Tcal, \Hcal\}$.

\begin{lemma}
\label{lmm:contraction_monotonicity}
Let $\mathbf{u}\in \e2$ {and $\Rank =\Rank_\Fcal$, as in \eqref{eq:rankvectorsets}, given by \eqref{def:R-T}, \eqref{def:R-H}, respectively. 
Then for $i\in\{1,\ldots,m\}$, $\nu \in
\nabla^{d_i}$, and any rank vector $\rr{r}\in \Rank$,} with $\rr{r}\leq\rank(\bu)$
componentwise, we have
\begin{equation*}
  \pi^{(i)}_\nu(\Psvd{\bu}{\rr{r}} \bu) \leq \pi^{(i)}_\nu(\mathbf{u}) \,,
\end{equation*}
where $\UU(\bu)$ either stands for $\UU^\Tcal(\bu)$ or $\UU^\Hcal(\bu)$   for the Tucker and    hierarchical Tucker format, respectively, see
\eqref{UU-T}, \eqref{UU-H}.
\end{lemma}

\begin{proof}
We consider first the Tucker format.
Using the {orthonormality} of the mode frames $\UU(\bu)$, we obtain
\begin{equation}\label{eq:contrmon1}
{   \bigl(\pi^{(i)}_\nu(\Psvd{\bu}{\rr{r}} \bu) \bigr)^2 = \sum_{\kk{\check
 k}_i\in\KK{K}(\rr{\check r}_i)}
    \Bigl( \sum_{k_i=1}^{r_i} U^{(i)}_{\nu,k_i}
 a_\kk{k} \Bigr)^2,\quad {\nu\in \nabla^{d_i}}  \,.  }
\end{equation}
For any fixed $\kk{\check k}_i$, we have
\begin{equation}\label{eq:contrmon2}
  \sum_{k_i = 1}^{r_i} \sum_{l_i = 1}^{r_i} U^{(i)}_{\nu,k_i}
  U^{(i)}_{\nu,l_i} 
  a_{\kk{\check k}_i|_{k_i}} a_{\kk{\check
  k}_i|_{l_i}} 
  = \Bigl(  \sum_{k_i = 1}^{r_i} U^{(i)}_{\nu,k_i} a_{\kk{\check
  k}_i|_{k_i}} \Bigr)^2 \geq 0\,, \quad {\nu\in \nabla^{d_i}} .
\end{equation}
{Combining this with \eqref{eq:contrmon1} and abbreviating}
$\rr{R} := \rank(\bu)$, we obtain
$$
\bigl(\pi^{(i)}_\nu(\Psvd{\bu}{\rr{r}} \bu) \bigr)^2
 \leq \sum_{k_i = 1}^{r_i} \sum_{l_i = 1}^{r_i} U^{(i)}_{\nu,k_i}
  U^{(i)}_{\nu,l_i} \sum_{\kk{\check k}_i \in \KK{K}(\rr{\check R}_i)}
    a_{\kk{\check k}_i|_{k_i}} a_{\kk{\check k}_i|_{l_i}}\,, \quad {\nu\in \nabla^{d_i}} .
$$
By Theorem \ref{thm:hosvd_properties}(ii),
the right hand side equals
$$
\sum_{k_i=1}^{r_i} \bigabs{\sigma^{(i)}_{k_i} U^{(i)}_{\nu,k_i}}^2 
  \leq \bigl( \pi^{(i)}_\nu (\mathbf{u}) \bigr)^2 \,.
$$
This shows the assertion for the Tucker format.
The proof for the hierarchical Tucker format follows similar lines. We treat
additional summations arising in the core tensor in the same way as the 
summation over $\kk{\check k}_i$ above, and apply  the same argument as in
\eqref{eq:contrmon2} recursively. 
\end{proof}

As a next step, we shall combine the coarsening procedure of this subsection 
with the tensor recompression considered earlier.
To this end, we define
{$N(\bv,\eta) := \min\bigl\{ N\colon \cerr_N(\bv) \leq \eta
  \bigr\}$, where $\cerr_N$ is defined   in \eqref{eq:tensor_coarsening_errest},}
as well as 
\begin{equation}\label{eq:tensorcoarsen_def}
\hatCctr{\eta} (\mathbf{v}) := \Cctr_{\bv, N(\bv;\eta)}
\mathbf{v}\,,
\end{equation}
in order to switch from {$N$-term approximation to the corresponding thresholding procedures}. As a consequence of \eqref{eq:tensor_coarsening_errest},
we have
\beqn
\label{near-best-coarsen}
  \norm{\mathbf{v} - \Cctr_{\mathbf{v},N} \mathbf{v}} 
  {\leq  \cerr_N(\bv) } \leq \constcrs
  \norm{\mathbf{v} - \Cbest_{\mathbf{v},N} \mathbf{v}} ,\quad \constcrs =\sqrt{m}.
\eeqn
Our general
 assumption on the approximability of mode frames is
that $\pi^{(i)}(\mathbf{u}) \in \As$ which, as mentioned before, reflects finite Sobolev or Besov regularity
of  {the functions whose wavelet coefficients are given by} the lower-dimensional tensor factors.

\subsubsection{Combination of Tensor Recompression and Coarsening}

Recall that we use $\|\cdot\|_{\As}$, $\|\cdot\|_{\AF{\ga}}$ to quantify sparsity of
wavelet expansions of mode frames, and low-rank approximability, respectively.
The following main result of this section applies again to both the Tucker
and the hierarchical Tucker format.  It extends Theorem \ref{propc1} in combining
tensor recompression and wavelet coarsening, and shows that both reduction techniques combined
are optimal up to uniform constants, and stable in the respective sparsity norms.

\begin{theorem}
\label{lmm:combined_coarsening}
For a given $\bv\in \ell_2(\nabla^d)$, let the mode frame system $\UU(\bv)$ be  either $\UU^\Tcal(\bv)$
or $\UU^\Hcal(\bv)$ (see \eref{UU-T}, \eqref{UU-H}).
Let $\mathbf{u}, \mathbf{v} \in \spl{2}(\nabla^d)$ with
$\mathbf{u}\in\AF{\ga}$, $\pi^{(i)}(\mathbf{u}) \in \cA^s$ for
$i=1,\ldots,m$, and $\norm{\mathbf{u}-\mathbf{v}} \leq \eta$. As before let
 $\constsvd = \constcrs =
\sqrt{m}$ for the Tucker format, while for the $\Hcal$-Tucker format 
$\constsvd = \sqrt{2m-3}$ and $\constcrs = \sqrt{m}$. Then for
\begin{equation}
\label{weta}
\mathbf{w}_{\eta} := \hatCctr{\constcrs
  (\constsvd+1)(1+\alpha)\eta} \bigl(\hatPsvd{\constsvd
(1+\alpha)\eta} (\mathbf{v}) \bigr) \,,
\end{equation}
we have
\begin{equation}
\label{eq:combinedcoarsen_errest}
  \norm{\bu - \bw_\eta} \leq \bigl( 1 + \constsvd(1+\alpha) + \constcrs
  (\constsvd+1)(1+\alpha)\bigr) \eta ,
\end{equation}
as well as
\begin{equation}\label{eq:combinedcoarsen_rankest}
 \abs{\rank(\bw_\eta)}_\infty \leq \ga^{-1}\bigl(\garatio
 \norm{\bu}_{\AF{\ga}}/(\alpha \eta)\bigr)\,,\qquad \norm{\bw_\eta}_{\AF{\ga}}
 \leq C_1 \norm{\bu}_{\AF{\ga}} ,
\end{equation}
with $C_1 = (\alpha^{-1}(1+\constsvd(1+\alpha)) + 1)$ and
\begin{equation}\label{eq:combinedcoarsen_suppest}
\begin{aligned}
 \sum_{i=1}^m \#\supp_i (\mathbf{w}_\eta) &\leq {2 \eta^{-\frac{1}{s}} m\, \alpha^{-\frac{1}{s}} } \Bigl(
 \sum_{i=1}^m \norm{ \pi^{(i)}(\bu)}_{\As} \Bigr)^{\frac{1}{s}}  \,, \\
   \sum_{i=1}^m \norm{\pi^{(i)}(\bw_\eta)}_{\As} &\leq C_2 \sum_{i=1}^m
   \norm{\pi^{(i)}(\bu)}_{\As} ,
 \end{aligned}
\end{equation}
with $C_2 = 2^s(1+ {3^s}) + 2^{{4s}} { \alpha^{-1} \bigl( 1+  \constsvd(1+\alpha) + \constcrs  (\constsvd + 1)(1+\alpha) \bigr)  }  m^{\max\{1,s\}}$.
\end{theorem}

\begin{proof}
Taking \eqref{c6} in Lemma \ref{lmmc1} and the definition \eqref{eq:tensorcoarsen_def} into account, the 
relation \eqref{eq:combinedcoarsen_errest} follows from the triangle inequality.
 
The statements in \eqref{eq:combinedcoarsen_rankest} follow from Theorem 
\ref{propc1}. Note that the
additional mode frame coarsening considered here does not affect these estimates.

{For the proof of \eqref{eq:combinedcoarsen_suppest}}, we can
proceed similary to \cite[Corollary 5.2]{Cohen:01} (see also \cite[Theorem
4.9.1]{Cohen:03}).
We set $\hat \bw := \hatPsvd{\constsvd (1+\alpha)\eta} (\mathbf{v})$.
Let $N \in \N$ be minimal such that $\norm{\bu - \Cbest_{\bu,N}\bu} \leq \alpha
\eta$. Then
\begin{align*}
  \norm{\hat\bw - \Cbest_{\bu,N} \hat\bw} &\leq 
       \norm{(\id - \Cbest_{\bu,N})(\bu - \hat\bw)}
       + \norm{\bu - \Cbest_{\bu,N} \bu} \\
     &\leq \norm{\bu - \hat\bw} + \norm{\bu - \Cbest_{\bu,N} \bu} \leq
     \bigl({1+}\constsvd(1+\alpha)+\alpha\bigr)\eta \,,
\end{align*}
where we have used {Lemma \ref{lmmc1}} to bound the first summand on the right hand side.
Consequently, by \eqref{near-best-coarsen},
\begin{multline}
\label{eq:combinedcoarsen_approx_nterms}
  \norm{\hat\bw - \Cctr_{\hat\bw,N} \hat\bw} \leq { \cerr_N(\hat \bw) \leq }
   \constcrs \norm{\hat\bw - \Cbest_{\hat\bw,N} \hat\bw}  \\
   \leq \constcrs \norm{\hat\bw - \Cbest_{\bu,N} \hat\bw}
   \leq \constcrs
  \bigl({1+}\constsvd(1+\alpha)+\alpha\bigr)\eta \,.
\end{multline}
Furthermore, {note that without loss of generality, we may assume $N \geq m$.} Keeping the definition \eqref{C-best} and the optimality \eqref{C-best2}
 in mind, \eqref{telesc} yields 
\begin{align*} 
{ \alpha \eta <  \norm{\bu - \Cbest_{\bu,N-1}\bu} } &\leq \inf_{\sum_i \#\Lambda_i \leq { N-1 }} \Bigl(
   \sum_{i=1}^m \norm{\pi^{(i)}(\bu) - \Restr{\Lambda_i} \pi^{(i)}(\bu)}^2
   \Bigr)^{\frac{1}{2}} \\
    &\leq \sum_{i=1}^m \inf_{\#\Lambda_i \leq {(N-1)}/m} 
     \norm{\pi^{(i)}(\bu) - \Restr{\Lambda_i} \pi^{(i)}(\bu)} \\
     &\leq \bigl({(N-1)}/m\bigr)^{-s}  \sum_{i=1}^m \norm{\pi^{(i)}(\bu)}_{\As}\\
     &   {\leq 2^s \bigl({N}/m\bigr)^{-s}\sum_{i=1}^m \norm{\pi^{(i)}(\bu)}_{\As}} \,. 
\end{align*}
Using the latter estimate and noting that, {by  
\eqref{eq:combinedcoarsen_approx_nterms},} the coarsening operator $\hatCctr{\constcrs
  (1+\constsvd)(1+\alpha)\eta}$ retains at most $N$ terms, we conclude that
\begin{equation}
\label{eq:combinedcoarsen_suppineq_interm}
  \sum_{i=1}^m \#\supp_i(\bw_\eta) \leq N \leq {2m\,\alpha^{-\frac{1}{s}} \eta^{-\frac{1}{s}} 
       }    \Bigl( \sum_{i=1}^m \norm{
       \pi^{(i)}(\bu)}_{\As} \Bigr)^{\frac{1}{s}} \,,
\end{equation}
and hence the first statement in \eqref{eq:combinedcoarsen_suppest}. 
{Now let  $\hat N = \sum_{i=1}^m \hat N_i$ with $\hat N_i := \#\supp_i(\bw_\eta)$, where we may also assume $\hat N_i > 0$ without loss of generality}.
{Resolving \eqref{eq:combinedcoarsen_suppineq_interm}} for $\eta$,
one can rewrite \eqref{eq:combinedcoarsen_errest}   as
\begin{equation}
\label{eq:combinedcoarsen_convest1}
  \norm{\bu - \bw_\eta} \leq  { {\hat N}^{-s} } {C(\alpha)}  
    \, m^s \, \Bigl( \sum_{i=1}^m \norm{\pi^{(i)}(\bu)}_{\As} \Bigr) \,,
\end{equation}
where  {$C(\alpha):= 2^s\alpha^{-1}(1+  \constsvd(1+\alpha) + \constcrs
  (\constsvd + 1)(1+\alpha))$}.
Let $\mathbf{\hat u}_i$ be the best $\hat N_i$-term approximation to
$\pi^{(i)}(\bu)$, then
\begin{align*}
  \norm{\pi^{(i)}(\bw_\eta)}_{\As} 
     &\leq 2^s \bigl( \norm{\mathbf{\hat u}_i}_{\As} 
        +  \norm{\mathbf{\hat u}_i  - \pi^{(i)}(\bw_\eta)}_{\As} \bigr) \\
     &\leq 2^s \bigl( \norm{\pi^{(i)}(\bu)}_{\As} 
        +  {(2 \hat N_i + 1)^s} \norm{\mathbf{\hat u}_i  - \pi^{(i)}(\bw_\eta)} \bigr)
        \\
     & \leq 2^s \bigl( \norm{\pi^{(i)}(\bu)}_{\As} \\
       & \quad  +  {(2 \hat N_i + 1)^s} (\norm{\mathbf{\hat u}_i  - \pi^{(i)}(\bu)} 
         + \norm{\pi^{(i)}(\bu) - \pi^{(i)}(\bw_\eta)}  ) \bigr)\\
     & \leq 2^s\bigl((1+{3^s})\norm{\pi^{(i)}(\bu)}_{\As} + {(2 \hat N_i + 1)^s} \norm{\pi^{(i)}(\bu) - \pi^{(i)}(\bw_\eta)}  \bigr) \,,
\end{align*}
where we have used 
  that $\norm{\mathbf{\hat u}_i  - \pi^{(i)}(\bu)} \leq {\hat N_i}^{-s}
\norm{\pi^{(i)}(\bu)}_{\As}$ as well as {$\#\supp(\mathbf{\hat u}_i - \pi^{(i)}(\bw_\eta)) \leq 2 \hat N_i$.}
Moreover, as a consequence of the Cauchy-Schwarz
inequality, we have the componentwise estimate
\begin{equation*}
  \abs{\pi^{(i)}_\nu(\bu) - \pi^{(i)}_\nu(\bw_\eta)}
    \leq \pi^{(i)}_\nu(\bu - \bw_\eta) \,,
\end{equation*}
{which yields
\begin{align*}
  \norm{\pi^{(i)}(\bu) - \pi^{(i)}(\bw_\eta)} &\leq \norm{\pi^{(i)}(\bu -
  \bw_\eta)} = \norm{\bu - \bw_\eta} \,.
\end{align*}
 Combining this with \eqref{eq:combinedcoarsen_convest1}, we obtain}
\begin{multline*}
\norm{\pi^{(i)}(\bw_\eta)}_{\As}\leq 2^s (1+{3^s})\norm{\pi^{(i)}(\bu)}_{\As}  \\  + {2^{s}} C(\alpha)\, m^s
{ \hat N^{-s} } (2 \hat N_i + 1)^s \Bigl(
  \sum_{k=1}^m \norm{\pi^{(k)}(\bu)}_{\As} \Bigr).
\end{multline*}
 Summing over $i=1,\ldots,m$ and noting that 
 $${\hat N^{-s}} \sum_{i=1}^m {(2\hat N_i + 1)^s}
    \leq {2^{2s}} m^{\max\{0,1-s\}}\,, $$
 we arrive at the second assertion in \eqref{eq:combinedcoarsen_suppest}.
\end{proof}

\section{Adaptive Approximation of Operators}\label{sect:op}

Whether the solution to an operator equation actually exhibits some tensor- and expansion sparsity is expected to depend 
strongly on the structure of the involved operator. The purpose of this  section is formulate a class of  operators which
are ``tensor-friendly'' in the sense that their approximate application does not increase the rank too much. 
Making this precise requires some {\em model assumptions} which at this point we feel are relevant in that
a wide range of interesting cases is covered. But of course, many possible variants would be conceivable as well.
In that sense the main issue in the subsequent discussion is to identify the essential structural mechanisms that would still work under somewhat different model assumptions. 

We shall approach this on two levels. First we consider operators with an {\em exact} low rank structure. Of course, assuming that the operator
is a single tensor product of operators acting on functions of a smaller number of variables would be far too restrictive and also
concern a trivial scenario, since ranks would be preserved. More interesting are {\em sums} of tensor products such as the $m$-dimensional {\em Laplacian}
$$
\Delta = \partial_{x_1}^2 + \cdots + \partial_{x_m}^2,
$$ 
where strictly speaking each summand $\partial_{x_j}$ is a {tensor product of the} {\em identity operators} acting on all but the $j$th variable
{with the} second order partial derivative with respect to the $j$th variable. Hence the wavelet representation $\bA$ of $\Delta$ in
 an $L_2$-orthonormal wavelet basis has the form
\beqn
\label{Laplacian}
\bA = \bA_1\otimes \id_2\otimes \cdots \otimes \id_m + \cdots + \id_1\otimes \cdots \otimes \id_{m-1}\otimes \bA_m,
\eeqn
where $\bA_j$ is the wavelet representation of $\partial_{x_j}$.  There is, however,  an issue concerning the scaling of the wavelet bases.
For $L_2$-orthonormalized wavelets $\bA$ is not bounded, an issue to be taken up later again
in Remark \ref{rmrk:sobolev}. 

At a second stage it is important to cover also operators which do not have an explicit low-rank structure but can be approximated 
in a quantified manner by low-rank operators. A typical example are potential terms, such as those arising in electronic structure calculations, see, e.g., \cite{Bachmayr:12} and the references cited there, {as well as the rescaled versions of operators of the type \eqref{Laplacian}, mentioned above}.

\subsection{Operators with Explicit Low-Rank Form}\label{ssect:ex-low-rank}

We start with a technical observation that will be used at several points.
Given operators $\bB^{(i)} = (b^{(i)}_{\nu_i,\tilde \nu_i})_{\nu_i,\tilde \nu_i\in \nabla^{d_i}}: \ell_2(\nabla^{d_i})\to \ell_2(\nabla^{d_i})$, recall that
their tensor product $\bB= \bB^{(1)}\otimes \cdots \otimes \bB^{(m)}$ is given
by
$
\bB_{\nu,\tilde\nu}= b^{(1)}_{\nu_1,\tilde\nu_1}\cdots b^{(m)}_{\nu_m,\tilde\nu_m}$ so that, whenever $\bv= \bv^1\otimes \cdots \otimes \bv^m$, $\bv^j\in \ell_2(\nabla^{d_j})$, we have $\bB\bv = (\bB^{(1)}\bv^1)\otimes \cdots \otimes (\bB^{(m)}\bv^m)$.
Observing that for any $\bv\in \ell_2(\nabla)$
\beqn
\label{eq:fac}
\bB\bv = \Big(\id_1 \otimes \bB^{(2)}\otimes \cdots \otimes \bB^{(m)} \Big)\Big( \big(\bB^{(1)} \otimes \id_2 \otimes \cdots\otimes \id_m\big)\bv\Big) ,
\eeqn
we conclude
\begin{equation*}
\norm{\bB\bv} \leq  \bignorm{\bB^{(2)}\otimes \cdots \otimes \bB^{(m)}} \,\bignorm{ \pi^{(1)}\bigl(( \bB^{(1)} \otimes \id_2 \otimes \cdots\otimes \id_m)\bv\bigr)}
\end{equation*}
More generally, one obtains by the same argument
\begin{multline}
\label{simp-est}
\norm{\bB\bv} \leq \bignorm{\bB^{(1)}\otimes \cdots \otimes \bB^{(i-1)}\otimes \bB^{(i+1)}\otimes\cdots\otimes \bB^{(m)}} \\
\times\bignorm{\pi^{(i)}\bigl((\id_1\otimes \id_{i-1} \otimes \bB^{(i)}\otimes \id_{i+1}\otimes\cdots\otimes \id_{m} )\bv\bigr)}\,.
\end{multline}
Similarly, one derives from \eqref{eq:fac} the inequality
\begin{multline}
\label{eq:pi-est}
\pi^{(i)}(\bB \bv)_{\nu_i} \leq 
\big\|\bB^{(1)}\otimes \cdots \otimes \bB^{(i-1)}\otimes \bB^{(i+1)}\otimes\cdots\otimes \bB^{(m)}\big\| \\
\times \pi^{(i)}\big((\id_1\otimes\cdots\otimes \id_{i-1}\otimes \bB^{(i)} \otimes \id_{i+1}\otimes \cdots\otimes \id_m)\bv\big)_{\nu_i},\quad \nu_i\in\nabla^{d_i}.
\end{multline}

\subsubsection{Tucker Format}\label{ssect:Tuck-op}
We shall be concerned first with (wavelet representations of) operators 
$\bA = (a_{\nu,\tilde \nu})_{\nu, \tilde \nu\in \nabla^d} : \ell_2(\nabla^d)\to  \ell_2(\nabla^d)$ composed
of tensor products of operators according to the Tucker format. 
For a given rank vector $\rr{R}\in\N^m$ throughout this section we assume that 
$\mathbf{A}\colon \spl{2}(\nabla^d)\to \spl{2}(\nabla^d)$ is bounded and has the form
\begin{equation}
\label{eq:tuck-op}
\mathbf{A} = \sum_{\kk{n}\in \KK{m}(\rr{R})} c_\kk{n} \bigotimes_{i=1}^m
\mathbf{A}^{(i)}_{n_i}\,,
\end{equation}
where
$\mathbf{A}^{(i)}_{n_i} \colon \spl{2}(\nabla^{d_i})\to \spl{2}(\nabla^{d_i})$
for $i\in\zinterval{1}{m}$ and $n_i\in\zinterval{1}{R_i}$.

\begin{example}
In particular, any operator of the form
$$  \bA_1 \otimes \id_2 \otimes\cdots \otimes \id_m 
  \,+\, \ldots \,+\, \id_1 \otimes \cdots \otimes \id_{m-1} \otimes \bA_m
$$
can be written in the form \eqref{eq:tuck-op} with $\rr{R} = (2,\ldots,2)$, $\bA_{1}^{(i)} = \id_i$,  $\bA^{(i)}_2=\bA_i$ for $i=1,\ldots, m$, and core tensor
$$  
   c_{2,1,\ldots,1}=\ldots = c_{1,\ldots,1,2,1,\ldots,1} = \ldots =
     c_{1,\ldots,1,2} = 1 \,,
  \qquad c_\kk{n} = 0\, \text{ otherwise.}
$$
\end{example}

The $\mathbf{A}^{(i)}_{n_i}$ are in general infinite matrices and not necessarily sparse in the strict sense.
We shall further require, however, that they are {\em nearly sparse} as will be quantified next.
To this end, suppose that for each $\mathbf{A}^{(i)}_{n_i}$ we have a sequence of approximations
(in the spectral norm) such that for a given sequence  $\varepsilon^{(i)}_{n_i,p} $, {$p\in \N_0$}, 
of tolerances,
\beqn
\label{eq:Aepsp}
\norm{\mathbf{A}^{(i)}_{n_i} - \mathbf{\tilde A}^{(i)}_{n_i,[p]}} \leq
\varepsilon^{(i)}_{n_i,p}, \quad p\in\N_0.
\eeqn
Moreover, it will be important to apply such sparsified versions of the $\bA_{n_i}^{(i)}$ 
to vectors which are supported on the elements of a {\em partition} $\{ \Lambda^{(i)}_{n_i,[p]} \}_{p\in\N_0}$
of $\nabla^{d_i}$. 

We shall then consider {\em approximations} $ \mathbf{\tilde A}$ to $\bA$
of the form
\begin{equation}\label{eq:approx_operator_generalform}
  \mathbf{\tilde A} = \sum_{\kk{n}\in \KK{m}(\rr{R})} c_\kk{n}
  \bigotimes_{i=1}^m \mathbf{\tilde A}^{(i)}_{n_i} \,,\quad
    \mathbf{\tilde A}^{(i)}_{n_i} := \sum_{p\in\N_0} \mathbf{\tilde
    A}^{(i)}_{n_i,[p]} \Restr{\Lambda^{(i)}_{n_i,[p]}}\,,
\end{equation}
where as before $\Restr{\Lambda}$ denotes the restriction of a given input sequence to $\Lambda$.
{The partitions $\Lambda_{n_i}^{(i)}$ will later be identified for a class of
matrices studied in the context of wavelet methods \cite{Cohen:01,Stevenson:02}.
In particular, choosing them in dependence on a given input sequence $\bv$ facilitates an adaptive
approximate application of $\mathbf{A}$ to $\bv$.}
The following lemma describes the accuracy of such approximations.

\begin{lemma}
\label{lmm:tensor_apply_est}
Let $\mathbf{v} \in \spl{2}(\nabla^d)$ {and let} $\mathbf{A}\colon \spl{2}(\nabla^d)\to \spl{2}(\nabla^d)$ have the form
\eqref{eq:tuck-op} for some $\rr{R}\in \N^m$, while $\tilde\bA$, given by \eqref{eq:approx_operator_generalform}, satisfies \eqref{eq:Aepsp}. Then
we have
\begin{equation}
\label{eq:tensor_apply_est}
 \norm{\mathbf{A}\mathbf{v} - \mathbf{\tilde A}\mathbf{v}}
   \leq   \sum_{i=1}^m \sum_{n_i=1}^{R_i}
   \sum_{p\in\N_0} 
   C^{(i)}_{\mathbf{\tilde A}} \varepsilon^{(i)}_{n_i,[p]} \,
   \bignorm{\Restr{\Lambda^{(i)}_{n_i,[p]}}\,\pi^{(i)}(\mathbf{v})}  \,,
\end{equation}
where 
\begin{equation*}
C^{(i)}_{\mathbf{\tilde A}} = \max_{n_i=1,\ldots,R_i}
\Bignorm{\sum_{\kk{\check n}_i} c_\kk{n} \Bigl(\bigotimes_{j=1}^{i-1}
\mathbf{\tilde A}^{(j)}_{n_j} \Bigr) \otimes \Bigl( \bigotimes_{j=i+1}^{m}
\mathbf{A}^{(j)}_{n_j} \Bigr)} \,.
\end{equation*}
\end{lemma}

\begin{proof}
The usual insertion-triangle inequality argument for estimating differences of products yields, upon using \eqref{simp-est}
and the definition of the constants $C^{(i)}_{\mathbf{\tilde A}}$,
\begin{align*}
 \norm{\mathbf{A}\mathbf{v} - \mathbf{\tilde A}\mathbf{v}} 
   &\leq \Bignorm{\sum_{n_1} (\mathbf{A}^{(1)}_{n_1} - \mathbf{\tilde
   A}^{(1)}_{n_1})
   \otimes \Bigl( \sum_{\kk{\check n}_1} c_n \mathbf{A}^{(2)}_{n_2} \otimes \cdots
   \otimes \mathbf{A}^{(m)}_{n_m}\Bigr) \, \mathbf{v}}  \\
       & \qquad + \ldots +  \Bignorm{\sum_{n_m} \Bigl(\sum_{\kk{\check n}_m} c_n
       \mathbf{\tilde A}^{(1)}_{n_1} \otimes \cdots 
         \otimes \mathbf{\tilde A}^{(m-1)}_{n_{m-1}}\Bigr) \otimes
       (\mathbf{A}^{(m)}_{n_m} - \mathbf{\tilde A}^{(m)}_{n_m}) \, \mathbf{v}}
       \\
   &\leq C^{(1)}_{\mathbf{\tilde A}}  \sum_{n_1}
   \bignorm{[(\mathbf{A}^{(1)}_{n_1} -
   \mathbf{\tilde A}^{(1)}_{n_1})  \otimes \id \otimes \cdots \otimes \id]  \mathbf{v}}   \\
    &\qquad + \ldots + 
    C^{(m)}_{\mathbf{\tilde A}}  \sum_{n_m} \bignorm{[\id \otimes \cdots \otimes \id
   \otimes 
   (\mathbf{A}^{(m)}_{n_m} - \mathbf{\tilde A}^{(m)}_{n_m})]\bv} \,.
\end{align*}
The assertion \eqref{eq:tensor_apply_est} follows now, using \eqref{eq:Aepsp}, from 
\begin{align*}
 \norm{[(\mathbf{A}^{(1)}_{n_1} - \mathbf{\tilde A}^{(1)}_{n_1}) \otimes \id \otimes \cdots \otimes \id ] \bv}
    & \leq  \sum_{p}  \bignorm{[(\mathbf{A}^{(1)}_{n_1} - \mathbf{\tilde
    A}^{(1)}_{n_1,[p]}) \Restr{\Lambda^{(1)}_{n_1,[p]}} \otimes \id \otimes
    \cdots \otimes \id] \mathbf{v}}  \\
   &\leq \sum_p \varepsilon^{(1)}_{n_1,p} \bignorm{\Restr{\Lambda^{(1)}_{n_1,[p]}} \pi^{(1)}(\bv)}
\end{align*}
and analogous estimates for the other summands.
\end{proof}

\begin{remark}\label{rmrk:operator_constants}
The constants $C^{(i)}_{\mathbf{\tilde A}}$, depending on the operator and
its approximation, may introduce a dependence on $m$.
In certain cases, this dependence is exponential. For instance, in the case of
an operator of the form $\mathbf{A} = \mathbf{B} \otimes \mathbf{B}\otimes
\cdots \otimes \mathbf{B}$ with $\norm{\mathbf{\tilde B}}\leq
\norm{\mathbf{B}}$, we obtain $C^{(i)}_{\mathbf{\tilde A}} =
\norm{\mathbf{B}}^{m-1}$.
This constant can therefore also strongly depend on an appropriate scaling of the
problem under consideration.
However, in the case of an operator
$$  \mathbf{A} = \mathbf{B}\otimes \id \otimes \cdots \otimes \id \,+\,
  \id\otimes\mathbf{B}\otimes \id \otimes \cdots \otimes \id \,+\, \ldots 
  \,+\, \id \otimes \cdots \otimes \id\otimes\mathbf{B} \,, $$
we obtain instead $C^{(i)}_{\mathbf{\tilde A}} \leq (m-1)
\norm{\mathbf{B}}$.
\end{remark}

\begin{definition}\label{def:scompressibility}
Let $\Lambda$ be a countable index set and let $s^* > 0$. We call an operator
$\mathbf{B}\colon \spl{2}(\Lambda)\to \spl{2}(\Lambda)$ 
\emph{$s^*$-compressible} if for any $0 < s < s^*$, there exist 
summable positive sequences $(\alpha_j)_{j\geq 0}$, $(\beta_j)_{j\geq 0}$ 
and for each $j\geq 0$, there exists $\mathbf{B}_j$ with at most $\alpha_j 2^j$
nonzero entries per row and column, such that $\norm{\mathbf{B} - \mathbf{B}_j}
\leq \beta_j 2^{-s j} $. 
For a given $s^*$-compressible operator $\mathbf{B}$, we denote the
corresponding sequences by $\alpha(\mathbf{B})$, $\beta(\mathbf{B})$.

Moreover, we say that a {\em family} of operators $\{ \mathbf{B}(n)
\}_n$ is \emph{equi-$s^*$-compressible} if all $\mathbf{B}(n)$ are $s^*$-compressible
with the same choice of sequences $(\alpha_j)$, $(\beta_j)$ 
and in addition, for all $\lambda \in\Lambda$ the number of nonzero elements in
the rows and columns of the approximations $\mathbf{B}(n)_j$ can be estimated
jointly for all $n$ in the form
$$ \# \Bigl(
\bigcup_n \bigl\{ \lambda'\in\Lambda \colon (\mathbf{B}(n)_j)_{\lambda,\lambda'} \neq 0 \vee
(\mathbf{B}(n)_j)_{\lambda',\lambda} \neq 0 \bigr\} \Bigr) \leq \alpha_j 2^j \,.
$$
\end{definition}  

\begin{example}
To give a structural example, let us assume that $\{
\psi_\lambda\}_{\lambda\in\nabla}$ is an orthonormal wavelet basis on $\R$.
As before, let $\abs{\lambda}$ denote the level of the basis function $\psi_\lambda$.
For $c,\sigma, \beta >0$, we denote by $\mathcal{M}_{c, \sigma,\beta}$ the class
of infinite matrices for which 
$$  
 \abs{b_{\lambda,\lambda'}}  \leq c \,
2^{-\abs{\abs{\lambda}-\abs{\lambda'}} \sigma} \bigl( 1 +
\dist(\supp\psi_\lambda, \supp\psi_{\lambda'}) \bigr)^{-\beta}  \,.   
$$
Such bounds are known to hold, for instance, for wavelet representations of the double layer potential   operator. 
Again, with a suitable rescaling of the wavelets, the representations of other potential types, as well as elliptic partial differential operators, exhibit the same decay of entries.
It is shown in \cite[Proposition 3.4]{Cohen:01} that (when specialized to the present case of
one-dimensional factors) any $\mathbf{B} \in \mathcal{M}_{c,\sigma,\beta}$ with
$\sigma > 1/2$, $\beta > 1$ is $s^*$-compressible with $s^* = \min\{\sigma-1/2, \beta-1\}$.

If $\mathbf{B}(n) \in \mathcal{M}_{c(n),\sigma(n),\beta(n)}$
with $c(n)$ and $\sigma(n)^{-1},\beta(n)^{-1}$ uniformly bounded, then from the
construction in the proof of \cite[Proposition 3.4]{Cohen:01} it can be seen that 
the $\mathbf{B}(n)$ are equi-$s^*$-compressible with 
$s^* = \min\{\inf_n\sigma(n)-1/2, \inf_n\beta(n)-1\}$, since the same
set of nonzero matrix entries can be used for each $n$.
\end{example}

The key property of $s^*$-compressible matrices in the context of adaptive methods is that they are not
only bounded in $\ell_2$ but also on the smaller approximation spaces, and thus preserve sparsity in a quantifiable manner.
We wish to establish such concepts next for the tensor setting.

To this end, assume that the components $\bA_{n_i}^{(i)}$ in $\bA$, given by \eqref{eq:tuck-op},  are $s^*$-compressible, and
let $\mathbf{A}^{(i)}_{n_i,j}$ be the corresponding
approximations according to Definition \ref{def:scompressibility}.
Quite in the spirit of the adaptive application of an operator in wavelet coordinates (see \cite{Cohen:01}),
for approximating $\bA \bv$ for a given $\bv\in \ell_2(\nabla^d)$,
the a-priori knowledge about $\bA$ in terms of $s^*$-compressibility is to be combined with a-posteriori information
on $\bv$. In fact,
given $\bv\in \ell_2(\nabla^d)$, we describe now how to construct for any $J\in\N$ approximations $\bw_J$ to 
the sequence $\bA\bv$ as follows. 
 For each $i$ and for $j\in\N$, let $\bar\Lambda^{(i)}_{j}$ be the
support of the best $2^j$-term approximation of $\pi^{(i)}(\mathbf{v})$ so that, in particular,
$\bar\Lambda^{(i)}_p \subset \bar\Lambda^{(i)}_{p+1}$.
If $\mathbf{A}^{(i)}_{n_i} = \id$, we simply set $\mathbf{\tilde A}^{(i)}_{n_i}
= \id$.  If $\mathbf{A}^{(i)}_{n_i} \neq \id$,  we  let $\bar\Lambda^{(i)}_{-1}:=\emptyset$ and
\begin{equation}
\label{eq:supps}
\Lambda^{(i)}_{[p]} := \left\{
\begin{array}{ll}
\bar\Lambda^{(i)}_{p} \setminus  \bar\Lambda^{(i)}_{p-1}, & p=0,\ldots,J  ,\\
\nabla^{d_i} \setminus \bar\Lambda^{(i)}_J, & p=J+1,\\
\emptyset, & p>J+1.
\end{array}\right.
\end{equation}
Moreover, let
\beqn
\label{eq:tildeA}
\mathbf{\tilde A}^{(i)}_{n_i,[p]} := \left\{
\begin{array}{ll} \mathbf{A}^{(i)}_{n_i,J-p},& p=0,\ldots,J,\\
0, & p> J.
\end{array}\right.
\eeqn
Note that {due to the particular choice of the sets $\Lambda^{(i)}_{[p]}$, the factors $\mathbf{\tilde A}^{(i)}_{n_i,[p]}$ formed according to \eqref{eq:approx_operator_generalform}
depend on the sequence $\bv$. However,} as a simple consequence of Definition \ref{def:scompressibility}, the $\mathbf{\tilde A}^{(i)}_{n_i}$ are bounded independently of $\mathbf{v}$.
\begin{lemma}
\label{lem:wJ}
Assume that the components $\bA_{n_i}^{(i)}$ of $\bA$ as in \eqref{eq:tuck-op} are $s^*$-compressible.
Given any $\bv\in\ell_2(\nabla^d)$, $J\in\N$, let
\begin{equation*}
\mathbf{\tilde A}_J := \sum_{\kk{n}\in \KK{m}(\rr{R})} c_\kk{n} \bigotimes_{i=1}^m
  \mathbf{\tilde A}^{(i)}_{n_i} \,,\quad
    \mathbf{\tilde A}^{(i)}_{n_i} := \sum_{p\in\N_0} \mathbf{\tilde
    A}^{(i)}_{n_i,[p]} \Restr{\Lambda^{(i)}_{n_i,[p]}}\,
\end{equation*} 
with $ \Restr{\Lambda^{(i)}_{n_i,[p]}}, \mathbf{\tilde A}^{(i)}_{n_i}$ defined by \eqref{eq:supps}, \eqref{eq:tildeA},
respectively. Then, whenever $\pi^{(i)}(\bv)\in \mathcal{A}^s$ for some $0<s<s^*$, the finitely supported sequence
$\bw_J := \tilde\bA_J \bv$ satisfies
\beqn
\label{eq:west}
  \norm{\bA\bv 
  - \mathbf{\tilde A}_J\mathbf{v}}
    \leq 
     {  2^{-s (J -1)}  }
    \sum_{i = 1}^m C^{(i)}_{\mathbf{\tilde A}}
     R_i \bigl( \max_{n}
    \norm{\bA^{(i)}_{n}} + \norm{\hat\beta^{(i)}}_{\spl{1}} \bigr)
     \norm{\pi^{(i)}(\mathbf{v})}_{\Acal^s} \,,
\eeqn
as well as 
\begin{equation}
\label{eq:suppJ}
  \#\supp_i (\mathbf{\tilde A}_J\mathbf{v})
    \leq      
     R_i \norm{\hat\alpha^{(i)}}_{\spl{1}} 2^J \,,
\end{equation}
where
the sequences $\hat\alpha$, $\hat\beta$ are defined as the
componentwise maxima of the sequences in Definition \ref{def:scompressibility}
for each $\mathbf{A}^{(i)}_{n_i}$, that is,
\beqn
\label{eq:maxsequences}
 \hat\alpha^{(i)}_j := \max_{n} \alpha_j(\bA^{(i)}_n)  \,,\qquad 
   \hat\beta^{(i)}_j := \max_{n} \beta_j(\bA^{(i)}_n) \,.
\eeqn
\end{lemma}
\begin{proof}
We apply Lemma \ref{lmm:tensor_apply_est} with $\mathbf{\tilde
A}^{(i)}_{n_i,[p]}$, defined    in \eqref{eq:tildeA},  and
$\Lambda^{(i)}_{n_i,[p]} := \Lambda^{(i)}_{[p]}$, according to \eqref{eq:supps}.
By $s^*$-compressibility, we have 
$
\norm{\bA^{(i)}_{n_i} - \mathbf{\tilde A}^{(i)}_{n_i,[p]}} \leq
\hat\beta^{(i)}_{J-p} 2^{- s(J-p)} =: \varepsilon^{(i)}_{n_i,p}$,
$p=0,\ldots,J,$ $\norm{\bA^{(i)}_{n_i}
- \mathbf{\tilde A}^{(i)}_{n_i,[J+1]}} = \norm{\bA^{(i)}_{n_i}}$, 
$\norm{\Restr{\Lambda^{(i)}_{[p]}} \pi^{(i)}(\bv)} = 0$ for $p>J+1$,
and therefore
\begin{eqnarray}
\label{eq:lem-appl}
\norm{\bA\bv - \bw_J}&\leq &\sum_{i=1}^m \sum_{n_i=1}^{R_i}
 C^{(i)}_{\mathbf{\tilde A}}\left\{   \sum_{p=0}^{J}
  \hat\beta^{(i)}_{J-p} 2^{- s(J-p)}\,
   \bignorm{\Restr{\Lambda^{(i)}_{n_i,[p]}}\,\pi^{(i)}(\mathbf{v})}\right. \nonumber\\
&&\left.  + \, \norm{\bA^{(i)}_{n_i}} \bignorm{ \Restr{\Lambda^{(i)}_{n_i,[J+1]}}\pi^{(i)}(\bv)}_{}\right\}.
\end{eqnarray}
By the choice of the $\Lambda^{(i)}_{[p]}$ and the definition of  $\norm{\cdot}_{\As}$, we obtain
{$\norm{\Restr{\Lambda^{(i)}_{[p]}} \pi^{(i)}(\bv)} \leq
2^{-s(p-1)} \norm{\pi^{(i)}(\bv)}_{\As}$} for $p = 0,\ldots, J+1$, 
which confirms \eqref{eq:west}.
Furthermore, 
\begin{equation}
\label{eq:tensor_apply_supportest}
  \#\supp_i (\mathbf{\tilde A}_J \mathbf{v}) 
    \leq  R_i(\hat\alpha^{(i)}_{J} 2^{J} 2^0
     + \hat\alpha^{(i)}_{J-1} 2^{J-1} 2^1 
     + \ldots + \hat\alpha^{(i)}_0 2^0 2^J) \leq
     R_i \norm{\hat\alpha^{(i)}}_{\spl{1}} 2^J \,,
\end{equation}
which is \eqref{eq:suppJ}.
\end{proof}

\begin{remark}
\label{rem:apply}
Whenever $\bv$ is finitely supported there exists a $p(\bv)\in \N_0$ such that
$\Lambda^{(i)}_{[p]}=\emptyset$ for $i=1,\ldots,m$, $p>p(\bv)$. Hence, the right hand side of 
\eqref{eq:lem-appl} can be computed for each $J\in \N_0$, where the sum over $p$ terminates
for $J\geq p(\bv)$ at $p(\bv)$. Further increasing $J$ will then decrease all summands 
on the right hand side of 
\eqref{eq:lem-appl}. Therefore, fixing any $s <s^*$ (close to $s^*$), we can find for any $\eta >0$
the integer $J(\eta)$ defined by
\begin{eqnarray}
\label{eq:Jeta}
J(\eta) &:= &\argmin_{J\in \N_0}
\biggl\{
\sum_{i=1}^m \sum_{n_i=1}^{R_i}
 C^{(i)}_{\mathbf{\tilde A}}\Bigl\{   \sum_{p=0}^{J}
  \hat\beta^{(i)}_{J-p} 2^{- s(J-p)}\,
   \bignorm{\Restr{\Lambda^{(i)}_{n_i,[p]}}\,\pi^{(i)}(\mathbf{v})}
   \nonumber\\
&&\qquad\qquad
 + \, \norm{\bA^{(i)}_{n_i}} \bignorm{ \Restr{\Lambda^{(i)}_{n_i,[J+1]}}\pi^{(i)}(\bv)} \Bigr\}
\leq \eta\biggr\}.
\end{eqnarray}
\end{remark}

To further examine the properties of $\tilde\bA_{J(\eta)}\bv$ for a given finitely supported $\bv$
let
\beqn
\label{eq:C-constants}
 C_{\hat\alpha}^{(i)} :=  \norm{\hat\alpha^{(i)}}_{\spl{1}}  \,, \qquad
   C_{\hat\beta}^{(i)} := \bigl( \max_{n} \norm{\bA^{(i)}_{n}} +
     \norm{\hat\beta^{(i)}}_{\spl{1}} \bigr) \,.  
\eeqn

\begin{theorem}
\label{lmm:tensor_scompr_est}
Under the assumptions of Lemma \ref{lem:wJ} on $\bA$ and any given finitely supported $\bv\in \ell_2(\nabla^d)$,
for any $\eta >0$ let
\beqn
\label{eq:weta}
\bw_\eta := \tilde\bA_{J(\eta)}\bv =: \tilde\bA_\eta\bv,
\eeqn
where $J(\eta)$ is defined by \eqref{eq:Jeta}. Then
\begin{align} 
\label{eq:approx-eta}
\norm{\mathbf{A}\mathbf{v} - \mathbf{\tilde
A}_\eta\mathbf{v}} & \leq \eta\,,   \\
   \#  \supp_i (\mathbf{\tilde A}_\eta \mathbf{v})  &\leq     
  { 4 } \,C_{\hat\alpha}^{(i)}\, R_i\, \eta^{-\frac{1}{s}}\, \Bigl(
  \sum_{j=1}^m C^{(j)}_{\hat\beta} C^{(j)}_{\mathbf{\tilde A}}  R_j
  \norm{\pi^{(j)}(\mathbf{v})}_{\Acal^s}  \Bigr)^{\frac{1}{s}} \,,
   \label{eq:tensor_apply_support} \\
   \norm{\pi^{(i)}(\mathbf{\tilde A}_\eta\mathbf{v})}_{\Acal^s} &\leq
      \frac{{2^{3s+2}}}{2^s-1} \,{\bigl(C^{(i)}_{\hat\alpha}\bigr)}^s 
      C^{(i)}_{\hat\beta} C^{(i)}_{\mathbf{\tilde A}} \,R_i^{1+s}
      \norm{\pi^{(i)}(\mathbf{v})}_{\Acal^s},
    \label{eq:tensor_apply_sparsity}
\end{align}
for all $i=1,\ldots,m$, where $C^{(i)}_{\mathbf{\tilde A}}$ is as in Lemma
\ref{lmm:tensor_apply_est}, and the constants $C^{(i)}_{\hat\alpha}$,
$C^{(i)}_{\hat\beta}$ are defined by \eqref{eq:C-constants} and 
 are independent of
$\mathbf{v}$, $\eta$, and $m$.
Moreover, 
\begin{equation}
 {\rank_i(\mathbf{\tilde A}_\eta \mathbf{v})} \leq R_i \rank_i(\mathbf{v}),\quad i=1,\ldots, m \,.
   \label{eq:tensor_apply_ranks}
\end{equation}
\end{theorem}

\begin{proof} \eqref{eq:approx-eta} follows from \eqref{eq:Jeta}. 
The bound \eqref{eq:tensor_apply_support} is an immediate consequence of \eqref{eq:suppJ}.
Choosing for a given finitely supported $\bv$ the mode frame system $\UU=\UU(\bv)$ according to HOSVD,
  \eqref{eq:tensor_apply_ranks} is clear, since with $\mathbf{U}^{(i)}$ and
$\mathbf{a}$ as in Lemma \ref{lmm:tensor_apply_est}, we obtain
\begin{equation}\label{eq:tensor_apply_Avrep}
 { \mathbf{\tilde A}_\eta  } \mathbf{v} = \sum_{\kk{n} \in \KK{m}(\rr{R})}
  \sum_{\kk{k}\in \N^m} \mathbf{d}_{(n_1,k_1),\ldots,(n_m,k_m)} 
   \bigotimes_{i=1}^m \mathbf{\tilde
  A}^{(i)}_{n_i} \mathbf{U}^{(i)}_{k_i} \,,
\end{equation}
where $\mathbf{d}_{(n_1,k_1),\ldots,(n_m,k_m)} = c_\kk{n} a_\kk{k}$.

Without loss of generality it suffices to prove
\eqref{eq:tensor_apply_sparsity} only for $i=1$, which allows us to temporarily simplify the notation 
by writing  $\Lambda_{[p]}$ for $\Lambda^{(1)}_{[p]}$.
Note first that for each $\nu_1\in\nabla^{d_1}$, using Proposition
\ref{prp:contraction_subadd} followed by \eqref{eq:pi-est} and \eqref{contract-mode}, we obtain
\begin{align}
\label{eq:tensor_scompr_contrest}
  \pi^{(1)}_{\nu_1}({ \mathbf{\tilde A}_\eta } \mathbf{v})  
    &\leq
    C^{(1)}_{\mathbf{\tilde A}} \sum_{n_1=1}^{R_1}
    \pi^{(1)}_{\nu_1}(\mathbf{\tilde A}^{(1)}_{n_1} \otimes \id \cdots\otimes
    \id \mathbf{v})\notag  \\
    & =  C^{(1)}_{\mathbf{\tilde A}} \sum_{n_1=1}^{R_1}
    \Bigl(\sum_k
    \bigabs{\sigma^{(1)}_k}^2 \bigabs{(\mathbf{\tilde A}^{(1)}_{n_1}
    \mathbf{U}^{(1)}_k)_{\nu_1}}^2 \Bigr)^\frac{1}{2} \,,
\end{align}
where we have used \eqref{contract-mode} in the last step for the mode frame system $\UU(\bv)$
from Theorem \ref{thm:hosvd_properties}.
In order to bound next the terms on the right hand side of \eqref{eq:tensor_scompr_contrest} let 
\begin{align*}
  \hat \Lambda_{n_1,[0]} &:= \supp\range \mathbf{A}^{(1)}_{n_1,0}
  \Restr{\Lambda_{[0]}} \,, \\
  \hat \Lambda_{n_1,[q]} &:= \Bigl( \bigcup_{j+\ell = q} \supp\range
  \mathbf{A}^{(1)}_{n_1,j} \Restr{\Lambda_{[\ell]}} \Bigr) \,\setminus\, \Bigl( \bigcup_{i
  < q} \hat\Lambda_{n_1,[i]} \Bigr) \,,\quad q > 0 \,.
\end{align*}
By the same argument as in \eqref{eq:tensor_apply_supportest}, we also obtain
\beqn
\label{eq:card}
\#\hat\Lambda_{n_1,[q]} \leq \norm{\hat\alpha^{(1)}}_{\spl{1}} 2^q.
\eeqn 
For {$q=0,\ldots,J$}, and each $k$, we have
$$
  \norm{\Restr{\hat\Lambda_{n_1,[q]}} \mathbf{\tilde A}^{(1)}_{n_1}
  \mathbf{U}^{(1)}_k} \leq {\sum_{\ell = 0}^{q-1}}
  \norm{\Restr{\hat\Lambda_{n_1,[q]}} \mathbf{\tilde A}^{(1)}_{n_1}
  \Restr{\Lambda_{[\ell]}} \mathbf{U}^{(1)}_k}
  { + \Bignorm{\mathbf{\tilde A}^{(1)}_{n_1}
  \sum_{\ell = q}^J \Restr{\Lambda_{[\ell]}} \mathbf{U}^{(1)}_k} },
$$
{On the one hand, by \eqref{eq:tildeA}, we obtain for $\ell=0,\ldots,q-1$, 
$$\Restr{\hat\Lambda_{n_1,[q]}} \mathbf{\tilde
A}^{(1)}_{n_1} \Restr{\Lambda_{[\ell]}} = \Restr{\hat\Lambda_{n_1,[q]}} (
\mathbf{A}^{(1)}_{n_1, J-\ell} - \mathbf{A}^{(1)}_{n_1, {q-\ell-1}})
\Restr{\Lambda_{[\ell]}}, $$ 
and hence
\begin{align*}
 \norm{\Restr{\hat\Lambda_{n_1,[q]}} \mathbf{\tilde A}^{(1)}_{n_1}
  \Restr{\Lambda_{[\ell]}}\mathbf{U}^{(1)}_k}
 &\leq  \bigl(
\norm{\mathbf{A}^{(1)}_{n_1} - \mathbf{A}^{(1)}_{n_1, J-\ell}} +
\norm{\mathbf{A}^{(1)}_{n_1} - \mathbf{A}^{(1)}_{n_1, q-\ell - 1}} \bigr)
    \norm{\Restr{\Lambda_{[\ell]}}
    \mathbf{U}^{(1)}_k} \\
    &\leq   \bigl( \hat\beta^{(1)}_{J-\ell} 2^{-s(J-\ell)} +
    \hat\beta^{(1)}_{q-\ell-1} 2^{-s(q-\ell -1)} \bigr)
    \norm{\Restr{\Lambda_{[\ell]}}
    \mathbf{U}^{(1)}_k} \\
    &\leq  \gamma_\ell  2^{-s(q-\ell -1)} 
    \norm{\Restr{\Lambda_{[\ell]}}
    \mathbf{U}^{(1)}_k}
\end{align*}
where we abbreviate $\gamma_\ell :=
\hat\beta^{(1)}_{J-\ell}+\hat\beta^{(1)}_{q-\ell-1}$}.
On the other hand,
\begin{align*}
  \Bignorm{\mathbf{\tilde A}^{(1)}_{n_1}
  \sum_{\ell = q}^J \Restr{\Lambda_{[\ell]}} \mathbf{U}^{(1)}_k} 
  &\leq \Bignorm{\sum_{\ell = q}^J \bigl[(\mathbf{A}^{(1)}_{n_1,J-\ell} - \mathbf{A}^{(1)}_{n_1}) + \mathbf{A}^{(1)}_{n_1}\bigr]\Restr{\Lambda_{[\ell]}} \mathbf{U}^{(1)}_k}   \\
   &\leq \sum_{\ell = q}^J {\hat\beta}^{(1)}_{J-\ell} 2^{-s(J-\ell)} \norm{\Restr{\Lambda_{[\ell]}} \mathbf{U}^{(1)}_k} + \norm{\mathbf{A}^{(1)}_{n_1}} \bignorm{\Restr{\bigcup_{j\geq q}\Lambda_{[j]} } \mathbf{U}^{(1)}_k}  \,.
\end{align*}
{Combining these estimates and applying the Cauchy-Schwarz inequality, yields
\begin{multline*}
 \norm{\Restr{\hat\Lambda_{n_1,[q]}} \mathbf{\tilde A}^{(1)}_{n_1}
  \mathbf{U}^{(1)}_k}  \\
    \leq \bigl( 3 \norm{\hat\beta^{(1)}}_{\spl{1}}+ \norm{\mathbf{A}^{(1)}_{n_1}} \bigr)^\frac12 
    \Bigl( \sum_{\ell=0}^{q-1} \gamma_\ell 2^{-2s(q-\ell-1)}  \norm{\Restr{\Lambda_{[\ell]}}
    \mathbf{U}^{(1)}_k}^2   \\
      + \sum_{\ell=q}^J \hat\beta^{(1)}_{J-\ell} 2^{-2s(J-\ell)}  \norm{\Restr{\Lambda_{[\ell]}}
    \mathbf{U}^{(1)}_k}^2  
       + \norm{\mathbf{A}^{(1)}_{n_1}} \bignorm{\Restr{\bigcup_{j\geq q}\Lambda_{[j]}} \mathbf{U}^{(1)}_k}^2    \Bigr)^\frac12  \,.
\end{multline*}
Again using \eqref{contract-mode}, as in \eqref{eq:tensor_scompr_contrest}, leads to
\begin{multline*}
 \norm{\Restr{\hat\Lambda_{n_1,[q]}} 
    \pi^{(1)}(\mathbf{\tilde A}^{(1)}_{n_1} \otimes \id_2 \cdots\otimes \id_m
 \mathbf{v})}^2   \\
    \leq 3C^{(1)}_{\hat\beta} 
   \Bigl( \sum_{\ell=0}^{q-1} \gamma_\ell 2^{-2s (q -\ell-1)} \norm{\Restr{\Lambda_{[\ell]}}
           \pi^{(1)}(\bv)}^2  \\
             + \sum_{\ell=q}^J \hat\beta^{(1)}_{J-\ell} 2^{-2s(J-\ell)} \norm{\Restr{\Lambda_{[\ell]}} \pi^{(1)}(\bv)}^2
           +  \norm{\mathbf{A}^{(1)}_{n_1}} \bignorm{\Restr{\bigcup_{j\geq q}\Lambda_{[j]}} \pi^{(1)}(\bv) }^2    \Bigr) \,,
\end{multline*}
and thus, since $\norm{\Restr{\Lambda_{[\ell]}} \pi^{(1)}(\bv)} \leq 2^{-s(\ell-1)} \norm{\pi^{(1)}(\bv)}_\As$ and $\bignorm{\Restr{\bigcup_{j\geq q}\Lambda_{[j]}}  \pi^{(1)}(\bv)} \leq 2^{-s(q-1)} \norm{\pi^{(1)}(\bv)}_\As$, for $q=0,\ldots,J$ we arrive at}
\begin{equation}
\label{eq:tensor_scompr_partdecay}
  \norm{\Restr{\hat\Lambda_{n_1,[q]}} \pi^{(1)}(\mathbf{\tilde A}^{(1)}_{n_1} 
  \otimes \id \cdots\otimes \id \mathbf{v})}
    \leq { 2^{-s q} \, 2^{2(s+1)} C^{(1)}_{\hat\beta} \norm{
    \pi^{(1)}(\mathbf{v})}_{\Acal^s} \,. }
\end{equation}
Recall that the sets $\hat\Lambda_{n_1,[q]}$ are disjoint with
$\#\hat\Lambda_{n_1,[q]} \leq \norm{\hat\alpha^{(1)}}_{\spl{1}} 2^q$.
By definition of the
$\As$-quasi-norm, we have 
\begin{multline*}
  \norm{\pi^{(1)}(\mathbf{\tilde A}^{(1)}_{n_1} 
  \otimes \id \cdots\otimes \id \mathbf{v})}_{\Acal^s}  \\
   \leq 
  { \sup_{q\in\N_0} \Big(\sum_{j<q} \#\hat\Lambda_{n_1,[j]} + 1 \Big)^s } \sum_{j\geq q}
    \norm{\Restr{\hat\Lambda_{n_1,[j]}} \pi^{(1)}(\mathbf{\tilde A}^{(1)}_{n_1} 
  \otimes \id \cdots\otimes \id \mathbf{v})}.
\end{multline*}
Hence from \eqref{eq:tensor_scompr_partdecay} we infer   
\begin{equation*}
  \norm{\pi^{(1)}(\mathbf{\tilde A}^{(1)}_{n_1} 
  \otimes \id \cdots\otimes \id \mathbf{v})}_{\Acal^s} \leq
  2^{3s+2}(2^s-1)^{-1}\norm{\hat\alpha^{(1)}}^s_{\spl{1}}
  C^{(1)}_{\hat\beta}  \, \norm{
  \pi^{(1)}(\mathbf{v})}_{\Acal^s} \,.
\end{equation*}
Since by the first inequality in \eqref{eq:tensor_scompr_contrest}, we have  
\begin{equation*}
  \norm{\pi^{(1)}({ \mathbf{\tilde A}_\eta } \bv)}_{ \As} \leq 
    C^{(1)}_{\mathbf{\tilde A}} R_1^s \sum_{n_1=1}^{R_1}
      \norm{\pi^{(1)}(\mathbf{\tilde A}^{(1)}_{n_1} \otimes \id
      \cdots\otimes \id \mathbf{v})}_{ \As}  \,,
\end{equation*}
we arrive at \eqref{eq:tensor_apply_sparsity}.  
\end{proof}

\begin{remark}\label{rmrk:equicompr_improvement}
The estimate \eqref{eq:tensor_apply_sparsity} corresponds to the worst case that
the sets $\hat\Lambda_{n_i,[q]}$ constructed in the proof are disjoint for
different $n_i$. If, on the contrary, the $\{\mathbf{A}^{(i)}_{n_i}\}_{n_i}$ are
equi-$s^*$-compressible, and hence these sets are the same for all $n_i$, we can
combine \eqref{eq:tensor_scompr_contrest} directly with \eqref{eq:tensor_scompr_partdecay} 
to obtain instead 
\begin{equation*} 
  \norm{\pi^{(i)}({ \mathbf{\tilde A}_\eta } \mathbf{v})}_{\Acal^s} 
    \lesssim
      {\bigl(C^{(i)}_{\hat\alpha}\bigr)}^s 
      C^{(i)}_{\hat\beta} C^{(i)}_{\mathbf{\tilde A}} R_i
      \norm{\pi^{(i)}(\mathbf{v})}_{\Acal^s} \,,  \end{equation*}
i.e., an improvement by a factor $R_i^s$. Similarly, in this case we also obtain
that by a modification of \eqref{eq:tensor_apply_supportest}, the estimate
\eqref{eq:tensor_apply_support} can be replaced by
\begin{equation*}
\# \supp_i ( {\mathbf{\tilde A}_\eta } \mathbf{v})  \lesssim
  C_{\hat\alpha}^{(i)} 
  \eta^{-\frac{1}{s}} \Bigl( \sum_{j=1}^m
  C^{(j)}_{\hat\beta} C^{(j)}_{\mathbf{\tilde A}} R_j
  \norm{\pi^{(j)}(\mathbf{v})}_{\Acal^s}  \Bigr)^{\frac{1}{s}} \,.
\end{equation*}
\end{remark}

\begin{remark}\label{rmrk:opapprox_tucker_ops_est}
If $r_i := \rank_i(\mathbf{v}) < \infty$, the number $\ops({ \mathbf{\tilde
A}_\eta } \mathbf{v})$ of arithmetic operations for evaluating ${ \mathbf{\tilde A}_\eta }\mathbf{v}$ as in {Theorem \ref{lmm:tensor_scompr_est}}, for a given HOSVD of
$\mathbf{v}$, can be estimated by
\begin{equation}\label{eq:tucker_opapprox_work}
  \ops({ \mathbf{\tilde A}_\eta } \mathbf{v}) \lesssim \prod_{i=1}^m
  R_i r_i + \eta^{-\frac{1}{s}} 
 \sum_{i=1}^m C_{\hat\alpha}^{(i)} 
 R_i r_i \Bigl( \sum_{j=1}^m C^{(j)}_{\hat\beta} C^{(j)}_{\mathbf{\tilde A}} R_j
 \norm{\pi^{(j)}(\mathbf{v})}_{\Acal^s} \Bigr)^{\frac{1}{s}}   
\end{equation}
with a constant independent of $\mathbf{v}$, $\eta$, and $m$.
\end{remark}

\begin{proof}
The sorting of entries of $\pi^{(i)}(\mathbf{v})$ required for
obtaining the index sets of best $2^j$-term approximations in Theorem
\ref{lmm:tensor_scompr_est} can be replaced by an \emph{approximate sorting} by
binary binning, requiring only $\#\supp_i(\mathbf{v})$ operations, as suggested
in \cite{Metselaar:02,Barinka:05}.
This only leads to a change in the generic constants in the resulting estimates. 

Let $\mathbf{v}$ have the HOSVD $\mathbf{v} = \sum_k a_k \UU_k$, then, on the one
hand, we need to form the core tensor for the result, which takes $\prod_{i=1}^m
R_i r_i$ operations, and evaluate the approximations to $\mathbf{\tilde
A}^{(i)}_{n_i} \mathbf{U}^{(i)}_{k_i}$ for $n_i=1,\ldots, R_i$ and
$k_i=1,\ldots,r_i$.
The number of operations for each of these terms can be estimated as in
\cite{Cohen:01}, which leads to \eqref{eq:tucker_opapprox_work}.
\end{proof}

As the first term on the right hand side of \eqref{eq:tucker_opapprox_work}
shows, the Tucker format still suffers from the curse of dimensionality due
to the complexity of the core tensors.

\subsubsection{Hierarchical Tucker Format}

For applying operators to coefficient sequences given in the 
hierarchical Tucker format, we need a representation of these operators 
with analogous hierarchical structure. That is, in the representation
\begin{equation}\label{eq:tucker_operator_rep}
  \mathbf{A} = \sum_{\kk{n}\in \KK{m}(\rr{R})} c_\kk{n} \bigotimes_{i=1}^m
\mathbf{A}^{(i)}_{n_i} \,,
\end{equation}
for the finitely supported tensor $\mathbf{c} = (c_\kk{n})\in \spl{2}(\N^m)$ we
  need in addition a hierarchical decomposition
\begin{equation}\label{eq:htucker_operator_core}
 \mathbf{c} = \hsum{\hdimtree{m}}(\{ \mathbf{C}^{(\alpha,\nu)}\colon \alpha\in
 \nonleaf{m},\, \nu\in\N\}) \,,
\end{equation}
see \eref{def:hsum}, \eref{eq:a-rep}.
Here for $\alpha \in \nonleaf{m}$, we
extend the definition of representation ranks in the representation of
$\mathbf{A}$ to each $\alpha\in\hdimtree{m}$ by setting
$R_{\{i\}} := R_i$ and 
\begin{equation}\label{eq:hier_operator_ranks}
 R_{\alpha} := \#\{\nu \colon \mathbf{C}^{(\alpha,\nu)}\neq 0 \}\,.
\end{equation}
In what follows, we assume $ \max_{\alpha\in \hdimtree{m}} R_\alpha <  \infty$
and $R_{\hroot{m}} = 1$.
According to Theorem \ref{thm:hiersvd_properties},  $\mathbf{v} \in \spl{2}(\nabla^d)$ 
has a representation 
$$  
\mathbf{v} = \sum_{\kk{k}\in\N^m} a_\kk{k} \UU_\kk{k}
\,,\quad \mathbf{a} = \hsum{\hdimtree{m}}(\{ \mathbf{B}^{(\alpha,k)}\}).
$$ 
If
$\max_{\alpha\in\hdimtree{m}} \rank_\alpha(\mathbf{v})< \infty$, then
${ \mathbf{\tilde A}_\eta } \mathbf{v}$ can be represented in the form
\eqref{eq:tensor_apply_Avrep}, with $\mathbf{d}$ again admitting
a hierarchical representation in terms of matrices 
$\mathbf{D}^{(\alpha,(\nu,k)))}$ on $\N^2\times \N^2$
with entries
$$  
\mathbf{D}^{(\alpha,(\nu,k)))}_{((\mu_1,l_1),(\mu_2,l_2))} :=
\mathbf{C}^{(\alpha,\nu)}_{(\mu_1,\mu_2)} \mathbf{B}^{(\alpha,k)}_{l_1,l_2} \,.
$$
That is, as in   \eqref{def:hsum}, we have an explicit representation 
$$ 
\mathbf{d} =
\hsum{\hdimtree{m}}\bigl(\{ \mathbf{D}^{(\alpha,(\nu,k)))}
\colon \alpha \in \nonleaf{m},\,
\nu=1,\ldots,R_\alpha,\,k=1,\ldots,\rank_\alpha(\bv) \}\bigr) 
$$
in \eqref{eq:tensor_apply_Avrep}, where the indices in $k\in\N$ are replaced
 in the definition of $\hsum{\hdimtree{m}}(\cdot)$ by the indices $(\nu,k) \in  \N^2$.

\begin{example}
To give a specific example, we consider an operator of the form
$$  \bA_1 \otimes \id_2 \otimes\cdots \otimes \id_m 
  \,+\, \ldots \,+\, \id_1 \otimes \cdots \otimes \id_{m-1} \otimes \bA_m
$$
in the hierarchical format with dimension tree
$$  \hdimtree{m} = \bigl\{ \hroot{m}, \{1\}, \{2,\ldots,m\}, \{2\}, \{3,\ldots,m\},\ldots,\{m\} \bigr\}  $$
as in Example \ref{ex:tt_ex4}. Setting $\bA_{1}^{(i)} = \id_i$,  $\bA^{(i)}_2=\bA_i$ for $i=1,\ldots, m$,
we obtain a representation as in \eqref{eq:htucker_operator_core} with
$R_\alpha = 2$ for $\alpha \neq \hroot{m}$, and 
$$  \mathbf{C}^{(\hroot{m},1)} = \begin{pmatrix} 0 & 1 \\ 1 & 0 \end{pmatrix}\,,\quad
  \mathbf{C}^{(\alpha,1)} = \begin{pmatrix} 1 & 0 \\ 0 & 0 \end{pmatrix}\,,\;
   \mathbf{C}^{(\alpha,2)} = \begin{pmatrix} 0 & 1 \\ 1 & 0 \end{pmatrix}\,,\; \alpha\in \nonleaf{m}\setminus\{\hroot{m}\}\,. $$
\end{example}

The estimates in Theorem \ref{lmm:tensor_scompr_est} now directly carry over
to the hierarchical Tucker format, where as the only modification,
\eqref{eq:tensor_apply_ranks} is replaced by 
\begin{equation*}
  \rank_\alpha({ \mathbf{\tilde A}_\eta } \mathbf{v}) \leq R_\alpha 
     \rank_\alpha(\mathbf{v}) \,.
\end{equation*}
 
\begin{remark}\label{rmrk:opapprox_hier_ops_est}
If  for a given $\Hcal$SVD of $\mathbf{v}$, 
$r_\alpha := \rank_\alpha(\mathbf{v}) < \infty$, $\alpha \in \nonleaf{m}$, the number
$\ops({ \mathbf{\tilde A}_\eta } \mathbf{v})$ of arithmetic operations 
for evaluating ${ \mathbf{\tilde A}_\eta } \mathbf{v}$ as in Theorem
\ref{lmm:tensor_scompr_est} can be estimated by
\begin{multline}
\label{eq:hier_opapprox_work}
  \ops({ \mathbf{\tilde A}_\eta } \mathbf{v}) \lesssim
  \sum_{\alpha\in\nonleaf{m}}
      R_\alpha r_\alpha \prod_{q=1}^2 R_{\child{q}(\alpha)}
      r_{\child{q}(\alpha)}   \\
   +  \eta^{-\frac{1}{s}} \sum_{i=1}^m C_{\hat\alpha}^{(i)} 
 R_i r_i \Bigl( \sum_{j=1}^m C^{(j)}_{\hat\beta} C^{(j)}_{\mathbf{\tilde A}} R_j
 \norm{\pi^{(j)}(\mathbf{v})}_{\Acal^s} \Bigr)^{\frac{1}{s}}   \,,
\end{multline}
with a constant independent of $\mathbf{v}$, $\eta$, and $m$.
\end{remark}

Comparing the first summand on the right hand side of \eqref{eq:hier_opapprox_work}
to the one in \eqref{eq:tucker_opapprox_work}, we observe a substantial reduction
in complexity regarding the dependence on $m$ (and hence $d$).

\subsection{Low-Rank Approximations of
Operators}\label{sec:operator_lowrank_approx}

In many applications of interest, the involved  operators do not have an
explicit low-rank form, but   there exist
efficient approximations to these operators in low-rank representation. 

Such a case can be handled by replacing a given operator $\mathbf{A}$ by
such an approximation and then applying the construction for operators
given in low-rank form as in the previous subsections.

To make this precise,   we assume that for a suitable growth sequence $\gamma_{\mathbf{A}}$,
 there exist approximations $\mathbf{A}_N$ for $N\in\N$ with
\begin{equation}
\label{eq:operator_lowrank_approx}
  \sup_N \gamma_{\mathbf{A}}(N) \norm{\mathbf{A} - \mathbf{A}_N}
   =: M_{\mathbf{A}} < \infty \,,
\end{equation}
where each $\mathbf{A}_N$ has a representation \eqref{eq:tucker_operator_rep}
with $R_i \leq N$. Moreover,    in the case of the hierarchical Tucker format we assume in addition
that $R_\alpha \leq N$ with $R_\alpha$ as in
\eqref{eq:hier_operator_ranks}.

Moreover, we need to quantify the approximability
of the $\mathbf{A}_N$.  We assume that all tensor factors arising in each 
$\mathbf{A}_N$ are $s^*$-compressible, and that for the
approximations ${\mathbf{\tilde A}_{N,\eta}}$ of $\mathbf{A}_N$ according to Lemma
\ref{lmm:tensor_apply_est} and Theorem \ref{lmm:tensor_scompr_est} -- with constants
$C^{(i)}_{\mathbf{\tilde A}_N}$, $C^{(i)}_{\hat\alpha_N}$, $C^{(i)}_{\hat\beta_N}$ in Theorem
\ref{lmm:tensor_scompr_est} -- we have
\begin{equation}
\label{eq:opapprox_constant}
C_{\mathbf{A},\mathbf{\tilde A}} 
:= \sup_N  \bigl( \max_i C^{(i)}_{\hat\alpha_N} \bigr)^s \bigl( \max_i
C^{(i)}_{\mathbf{\tilde A}_N} C^{(i)}_{\hat\beta_N} \bigr)  <\infty \,.
\end{equation}
Under these conditions, we shall say that
the approximations $\bA_N$ to $\mathbf{A}$ are \emph{uniformly $s^*$-compressible}.

Under this assumption, the estimates for $\ops({ \mathbf{\tilde A}_\eta } \mathbf{v})$
obtained in Remarks \ref{rmrk:opapprox_tucker_ops_est} and \ref{rmrk:opapprox_hier_ops_est}
carry over to the present setting with additional low-rank approximation of the
operator.
Here for given $\eta > 0$ and $\mathbf{v}$, we choose {$N_\eta$ such that
$\norm{\mathbf{A} - \mathbf{A}_{N_\eta}} \leq \eta/2$ and $\mathbf{\tilde A}_{N_\eta,\eta}$ such that
$\norm{\mathbf{A}_{N_\eta} \mathbf{v} - \mathbf{\tilde A}_{N_\eta,\eta}  \mathbf{v}} \leq \eta/2$,}
which in summary yields for the Tucker format
\begin{multline}\label{eq:tucker_opapprox_work_2}
 \ops({\mathbf{\tilde A}_{N_\eta,\eta}}\mathbf{v}) \lesssim
  \bigl(\gamma^{-1}_\bA (2 M_{\mathbf{A}} / \eta)\bigr)^m 
     \prod_{i=1}^m \rank_i(\mathbf{v})  \\
 + C_{\mathbf{A},\mathbf{\tilde A}}^{\frac{1}{s}} \eta^{-\frac{1}{s}}
 \bigl( \gamma^{-1}_\bA (2 M_{\mathbf{A}} / \eta) \bigr)^{1 + s^{-1}}
 \sum_{i=1}^m \rank_i(\mathbf{v}) 
 \Bigl( 
  \sum_{j=1}^m \norm{\pi^{(j)}(\mathbf{v})}_{\Acal^s} 
 \Bigr)^{\frac{1}{s}} \,,
\end{multline}
and for the hierarchical Tucker format
\begin{multline}\label{eq:hier_opapprox_work_2}
 \ops({\mathbf{\tilde A}_{N_\eta,\eta}}\mathbf{v}) \lesssim
  \bigl(\gamma^{-1}_\bA (2 M_{\mathbf{A}} / \eta)\bigr)^3
  \sum_{\alpha\in\nonleaf{m}} \rank_\alpha(\mathbf{v}) 
  \prod_{q=1}^2 \rank_{\child{q}(\alpha)}
     (\mathbf{v}) \\
 +C_{\mathbf{A},\mathbf{\tilde A}}^{\frac{1}{s}} \eta^{-\frac{1}{s}}
 \bigl( \gamma^{-1}_\bA (2 M_{\mathbf{A}} / \eta) \bigr)^{1 + s^{-1}}
 \sum_{i=1}^m \rank_i(\mathbf{v}) \Bigl( \sum_{j=1}^m
 \norm{\pi^{(j)}(\mathbf{v})}_{\Acal^s} \Bigr)^{\frac{1}{s}}  \,.
\end{multline}
Note again the reduction in complexity in the first term of \eqref{eq:hier_opapprox_work_2}
over \eqref{eq:tucker_opapprox_work_2}.

\section{An Adaptive Iterative Scheme}\label{sect:scheme}
\subsection{Formulation and Basic Convergence Properties}\label{ssect:formulation}

\newcommand{\elliptA}{{\lambda_{\mathbf{A}}}}
\newcommand{\boundA}{{\Lambda_{\mathbf{A}}}}

We have now all prerequisites in place to formulate an adaptive method whose basic structure resembles the one
introduced in \cite{Cohen:02} for linear operator equations 
$\mathbf{A} \mathbf{u} = \mathbf{f}$, where $\mathbf{f}\in\spl{2}$ and   $\mathbf{A}$ is bounded and elliptic on $\spl{2}$,
that is, 
\begin{equation*}
  \langle \mathbf{A} \mathbf{v},\mathbf{v}\rangle_{\spl{2}} \geq 
   \elliptA \norm{\mathbf{v}}^2_{\spl{2}}\,,\quad
  \norm{\mathbf{A}\mathbf{v}}_{\spl{2}}
  \leq \boundA \norm{\mathbf{v}}_{\spl{2}} 
\end{equation*}
holds for fixed constants $\elliptA, \boundA > 0$.
The scheme can be regarded as a perturbation of
a simple Richardson iteration,
\begin{equation}
  \mathbf{v}_{i+1} := \mathbf{v}_{i} - \omega (\mathbf{A} \mathbf{v}_{i} -
  \mathbf{f}) \,,
\end{equation}
which  applies to both symmetric and nonsymmetric elliptic $\mathbf{A}$. 
In both cases, the parameter $\omega>0$ can be chosen such that
$\norm{\id - \omega\mathbf{A}} < 1$.

Based on the developments in the previous sections,
we have at hand numerically realizable
procedures $\apply$, $\rhs$, $\recompress$, and $\coarsen$, 
which for finitely supported
$\mathbf{v}$ and any tolerance $\eta > 0$ satisfy
\begin{equation}\label{eq:procedures_error}
 \begin{aligned}
  &\norm{\mathbf{A}\mathbf{v} - \apply(\mathbf{v}; \eta)}  \leq \eta\,, &
  &\norm{\mathbf{f} - \rhs(\eta)} \leq \eta \,,  \\
  &\norm{\mathbf{v} - \recompress(\mathbf{v}; \eta)} \leq \eta  \,,\quad & 
  &\norm{\mathbf{v} - \coarsen(\mathbf{v}; \eta)} \leq \eta \,.
 \end{aligned}
\end{equation}
Specifications of the complexities of these procedures
will be summarized in \S\ref{ssect:complexity}.
The adaptive scheme that we analyze in what follows is given in Algorithm \ref{alg:tensor_opeq_solve}.

\begin{algorithm}[!ht]
\caption{$\quad \mathbf{u}_\varepsilon = \solve(\mathbf{A},
\mathbf{f}; \varepsilon)$} \begin{algorithmic}[1]
\Require \begin{minipage}{12cm}$\omega >0$ and $\rho\in(0,1)$ such that
$\norm{\id - \omega\mathbf{A}} \leq \rho$,\\
$\theta, \kappa_1, \kappa_2, \kappa_3 \in (0,1)$ with $\kappa_1 +
\kappa_2 + \kappa_3 \leq 1$, and $\beta \geq 0$.\end{minipage}
\Ensure $\mathbf{u}_\varepsilon$ satisfying $\norm{\mathbf{u}_\varepsilon -
\mathbf{u}}\leq \varepsilon$.
\State $\mathbf{u}_0 := 0$, $\delta := \lambda_\mathbf{A}^{-1}
\norm{\mathbf{f}}$ 
\State $k:= 0$, $J := \min\{ j \colon \rho^j (1 + (\omega + \beta) j) \leq
\kappa_1\theta\}$\label{alg:jchoice}
\While{$\theta^k \delta > \varepsilon$}
\State $\mathbf{w}_0:=\mathbf{u}_k$, $j \gets 0$
\Repeat
\State $\eta_j := \rho^{j+1} \theta^k\delta$
\State $\mathbf{r}_j := \apply( \mathbf{w}_j ; \frac{1}{2}\eta_j)
- \rhs(\frac{1}{2}\eta_j)$ 
\State $\mathbf{w}_{j+1} := \recompress(\mathbf{w}_j - \omega \mathbf{r}_j ;
\beta \eta_j)$ \label{alg:tensor_solve_innerrecomp}
\State $j\gets j+1$.
\Until{($j \geq J \quad \vee \quad \lambda_{\mathbf{A}}^{-1} \rho
\norm{\mathbf{r}_{j-1}} + (\lambda_{\mathbf{A}}^{-1} \rho  + \omega + \beta)
\eta_{j-1} \leq \kappa_1\theta^{k+1} \delta$)} \label{alg:cddtwo_looptermination_line}
\State $\mathbf{u}_{k+1} := \coarsen\bigl(\recompress(\mathbf{w}_j;
\kappa_2 \theta^{k+1} \delta) ; \kappa_3\theta^{k+1}
\delta\bigr)$\label{alg:cddtwo_coarsen_line} 
\State $k \gets k+1$
\EndWhile
\State $\mathbf{u}_\varepsilon := \mathbf{u}_k$ 
\end{algorithmic}
\label{alg:tensor_opeq_solve}
\end{algorithm}

\begin{proposition}\label{prp:tensor_iteration_opeq_convergence}
Let the step size $\omega >0$ in Algorithm \ref{alg:tensor_opeq_solve} satisfy
$\norm{\id - \omega\mathbf{A}} \leq \rho < 1$.
Then the intermediate steps $\mathbf{u}_{k}$ of Algorithm
\ref{alg:tensor_opeq_solve} satisfy $\norm{\mathbf{u}_k - \mathbf{u}} \leq
\theta^k\delta$, and in particular,
the output $\mathbf{u}_\varepsilon$ of Algorithm \ref{alg:tensor_opeq_solve}
satisfies $\norm{\mathbf{u}_\varepsilon - \mathbf{u}} \leq \varepsilon$.
\end{proposition}

\begin{proof}
Since $\kappa_1+\kappa_2+\kappa_3\leq 1$, it suffices to show that for any $k$,
after the termination of the inner loop the error bound 
\begin{equation}\label{eq:cddtwo_innerloop_reduction}
\norm{\mathbf{w}_j -
\mathbf{u}} \leq \kappa_1 \theta^{k+1} \delta 
\end{equation} 
holds.
By the choice of $\omega$, we have
\begin{align*}
  \norm{ \mathbf{w}_{j+1} - \mathbf{u} } &\leq { \norm{(\id - \omega\mathbf{A})
  (\mathbf{w}_j - \mathbf{u}) } + \omega\norm{(\mathbf{A}\mathbf{w}_j - \bbf) -
  \mathbf{r}_j }   + \beta \eta_j}   \\
  &\leq \rho \norm{\mathbf{w}_j - \mathbf{u}} + (\omega + \beta) \eta_j \,,
\end{align*}
and recursive application of this estimate yields
\begin{equation*}
 \norm{\mathbf{w}_{j} -\mathbf{u}} \leq \rho^{j} \norm{\mathbf{w}_0
 -\mathbf{u}} + (\omega + \beta) \sum_{l=0}^{j-1} \rho^{j-1-l} \eta_l\leq
 \rho^{j}\bigl(1+j(\omega+\beta)\bigr) \theta^k \delta \,.
\end{equation*}
Thus on the one hand, if the inner loop exits with the first condition in line
\ref{alg:cddtwo_looptermination_line}, then
\eqref{eq:cddtwo_innerloop_reduction} holds by definition of $J$. On the other
hand, if the second condition is met, then \eqref{eq:cddtwo_innerloop_reduction}
holds because
\begin{align*}  \norm{\mathbf{w}_j - \mathbf{u}} 
& \leq \rho 
\norm{\mathbf{w}_{j-1} - \mathbf{u}} + (\omega+\beta) \eta_{j-1} \\ 
& \leq \rho c_{\mathbf{A}}^{-1} (
\norm{\mathbf{r}_{j-1}} + \eta_{j-1}) + (\omega+\beta) \eta_{j-1} \leq
\kappa_1\theta^{k+1}\delta\,.
\qedhere
\end{align*}
\end{proof}

\subsection{Complexity}\label{ssect:complexity}

Quite in the spirit of adaptive wavelet methods we analyze the performance 
of the above scheme by comparing it to an ``optimality benchmark'' addressing the following question:
suppose the unknown solution exhibits a certain (unknown) rate of tensor approximability
where the involved tensors have a certain (unknown) best $N$-term approximability with respect to their wavelet representations.
Does the scheme automatically recover these rates? 
Thus, unlike the situation in wavelet analysis we are dealing here with {\em two} types of approximation,
and the choice of corresponding rates as a benchmark model should, of course, be representative for relevant application scenarios. For the present complexity analysis, we focus 
on growth sequences of {\em subexponential or exponential} type for the 
involved low-rank approximations, combined with an {\em algebraic} approximation
rate for the corresponding tensor mode frames.  The rationale for this choice is the following.
Approximation rates in classical methods are governed by the {\em regularity} of the approximand which,
unless the approximand is analytic, results in algebraic rates suffering from the curse of dimensionality.
However, functions of many variables may very well exhibit a high degree of tensor sparsity without being very regular
in the Sobolev or Besov sense. Therefore, fast tensor-rates combined with polynomial rates for the
compressibility of the mode frames mark an ideal target scenario for tensor methods, since, as
will be shown, the curse 
of dimensionality can be significantly ameliorated without requiring excessive regularity.
 
The precise formulation of our benchmark model reads as follows.

\begin{assumptions}\label{ass:approximability}
Concerning the tensor approximability of $\bu$, $\mathbf{A}$, and $\mathbf{f}$,
we make the following assumptions:
\begin{enumerate}[{\rm(i)}]
\item $\bu \in \AH{\ga_\bu}$ with $\ga_\bu(n) = e^{d_\bu n^{1/b_\bu}}$
for some $d_\bu>0$, $b_\bu \geq 1$.
\item $\mathbf{A}$ satisfies \eqref{eq:operator_lowrank_approx}
for an $M_\mathbf{A} > 0$, with $\ga_\mathbf{A}(n) = e^{d_\mathbf{A}
n^{1/b_\mathbf{A}}}$ where $d_\mathbf{A} >0$, $b_\mathbf{A} \geq 1$.
\item Furthermore, let $\mathbf{f} \in \AH{\ga_\mathbf{f}}$ with
$\ga_\mathbf{f}(n) = e^{d_\mathbf{f} \, n^{1/b_\mathbf{f}}}$,
where $d_\mathbf{f} = \min\{ d_\bu, d_\mathbf{A} \}$ and
$b_\mathbf{f} = b_\bu + b_\mathbf{A}$.
\end{enumerate}
Concerning the approximability of lower-dimensional components, we
assume that for some $s^* > 0$, we have the following:
\begin{enumerate}[{\rm(i)}]
 \setcounter{enumi}{3}
 \item $\pi^{(i)}(\bu) \in \As$ for $i=1,\ldots,m$, for any $s$ with $0 < s
 <s^*$.
 \item The low-rank approximations to $\mathbf{A}$ are uniformly
 $s^*$-compressible in the sense of 
 \S\ref{sec:operator_lowrank_approx}, with $C_{\mathbf{A}} := \sup_{\eta>0}
 C_{\mathbf{A},\mathbf{\tilde A}} < \infty$, where 
 $C_{\mathbf{A},\mathbf{\tilde A}}$ is defined as in
 \eqref{eq:opapprox_constant} for each value of $\eta$.
 \item $\pi^{(i)}(\mathbf{f}) \in \As$ for $i=1,\ldots,m$, for any $s$ with $0 < s
 <s^*$.
\end{enumerate}
Furthermore, we assume that the number of operations required for evaluating each required entry in the tensor approximations of $\bA$ or $\bbf$ is uniformly bounded.
\end{assumptions}

Note that the requirement on ${\bf f}$ in (iii) is actually very mild because the data are
typically more tensor sparse than the solution.

The following complexity estimates are formulated only
for the more interesting case of the hierarchical Tucker format.
Similar statements hold for the Tucker format, involving however additional terms
that depend exponentially on $m$, which makes this format suitable only for
moderate values of $m$.

\begin{remark}
Let $\bv$ have finite support with finite ranks, i.e.,
$\rank_\alpha(\bv) < \infty$ for $\alpha\in\hdimtree{m}$.
Then under Assumptions \ref{ass:approximability}, $\apply$ can be
realized numerically such that for $\bw_\eta := \apply(\bv; \eta)$ we have
(see Theorem \ref{lmm:tensor_scompr_est} and Remark \ref{rem:apply})
\begin{gather}
 \#  \supp_i (\bw_\eta)  \lesssim
  C_{\mathbf{A}}^{\frac{1}{s}} \bigl(d_\mathbf{A}^{-1}
    \ln(M_\mathbf{A}/\eta)\bigr)^{(1 + s^{-1})b_\mathbf{A}} 
     \Bigl( \sum_{j=1}^m 
  \norm{\pi^{(j)}(\mathbf{v})}_{\Acal^s}  \Bigr)^{\frac{1}{s}}
    \eta^{-\frac{1}{s}}  \,, \\
    \label{eq:rmrk_apply_spnorm}
  \norm{\pi^{(i)}(\mathbf{w}_\eta)}_{\As} \lesssim
      C_{\mathbf{A}} \bigl(d_\mathbf{A}^{-1}
    \ln(M_\mathbf{A}/\eta)\bigr)^{(s+1)b_\mathbf{A}}
      \norm{\pi^{(i)}(\mathbf{v})}_{\Acal^s} \,, \\
    \abs{\rank(\mathbf{w}_\eta)}_\infty \leq 
      \bigl(d_\mathbf{A}^{-1}
    \ln(M_\mathbf{A}/\eta)\bigr)^{b_\mathbf{A}}
      \abs{\rank(\mathbf{v})}_\infty \,,
\end{gather}
and, by \eqref{eq:hier_opapprox_work_2},
\begin{multline}\label{eq:apply_complexity_summary}
 \ops(\bw_\eta) \lesssim 
    (m-1) \bigl(d_\mathbf{A}^{-1}
    \ln(M_\mathbf{A}/\eta)\bigr)^{3 b_\mathbf{A}}
      \abs{\rank(\bv)}_\infty^3    \\
       +\; m\, C_{\mathbf{A}}^{\frac{1}{s}} \, \,
         \bigl(d_\mathbf{A}^{-1}
    \ln(M_\mathbf{A}/\eta)\bigr)^{(1+s^{-1}) b_\mathbf{A}}
      \abs{\rank(\mathbf{v})}_\infty 
       \Bigl( \sum_{i=1}^m 
  \norm{\pi^{(i)}(\mathbf{v})}_{\Acal^s}  \Bigr)^{\frac{1}{s}}
    \eta^{-\frac{1}{s}}  \,.
  \end{multline}
Thus, up to polylogarithmic terms, the curse of dimensionality is avoided.
If in addition the approximations of $\bA$ are  equi-$s^*$-compressible,
the polylogarithmic terms in the above estimates improve according to
Remark \ref{rmrk:equicompr_improvement}.
\end{remark}

\begin{remark}
Under Assumptions \ref{ass:approximability}, the routine $\rhs$ can be
realized numerically such that for $\mathbf{f}_\eta :=\rhs(\eta)$ we have
\begin{gather}
 \#  \supp_i (\mathbf{f}_\eta)  \lesssim
       \eta^{-\frac{1}{s}} \norm{\pi^{(i)}(\mathbf{f})}_{\As}^{\frac{1}{s}}\,, \\
  \norm{\pi^{(i)}(\mathbf{f}_\eta)}_{\As} \lesssim
      \norm{\pi^{(i)}(\mathbf{f})}_{\As}\,, \\
    \abs{\rank(\mathbf{f}_\eta)}_\infty \lesssim 
      \bigl(d_\mathbf{f}^{-1}
    \ln(\norm{\mathbf{f}}_{\AH{\ga_\mathbf{f}}}/\eta)\bigr)^{b_\mathbf{f}} \,,
\end{gather}
as well as
\begin{equation}
 \ops(\mathbf{f}_\eta) \lesssim 
    (m-1) \abs{\rank(\mathbf{f}_\eta)}_\infty^3      
       +\;  \abs{\rank(\mathbf{f}_\eta)}_\infty 
      \sum_{i=1}^m \#\supp_i(\mathbf{f}_\eta)    \,.
  \end{equation}
\end{remark}

\begin{remark}
We take $\recompress$ as a numerical realization of $\hatPsvd{\eta}$
as defined in \eqref{eq:tucker_recomp_def}.
This amounts to the computation of an HOSVD or $\Hcal$SVD, respectively,
for which we have the complexity bounds given in 
Remarks \ref{rmrk:hosvd_complexity} and \ref{rmrk:hiersvd_complexity}.

Likewise, $\coarsen$ is a numerical realization
of $\hatCctr{\eta}$ as defined in
\eqref{eq:tensorcoarsen_def}, with the modification of
replacing the exact sorting of the values $\pi^{(i)}_{\nu_i}(\cdot)$, 
$i=1,\ldots,m$, $\nu\in\nabla^{d_i}$, as required by
$\hatCctr{\eta}$, by an approximate sorting 
as proposed in \cite{Metselaar:02,Barinka:05}, see Remark \ref{rmrk:opapprox_tucker_ops_est}.
This leads to an increase of $\constcrs$ by only a fixed factor;
for finitely supported $\bv$, the procedure can be realized in practice such
that $\constcrs = 2\sqrt{m}$, and using a number of operations bounded by
$$  
  C \abs{\rank(\bv)}_\infty \sum_{i=1}^m \#\supp_i(\bv) 
$$
with a fixed $C>0$.
Note that here we make the implicit assumption that the orthogonality
properties required by $\coarsen$ have been 
enforced if necessary  before the application of $\coarsen$. This can be
done by an application of $\recompress(\cdot,0)$.
\end{remark}

Note that under the assumptions of Proposition
\ref{prp:tensor_iteration_opeq_convergence},
the iteration converges for any fixed $\beta \geq 0$.
A call to $\recompress$ (possibly with $\beta =0$, i.e., without performing
an approximation) is in fact necessary in each inner iteration to ensure the
orthogonality properties required by $\apply$.

The main result of this paper is the following theorem. It says that whenever the solution
has the approximation properties specified in Assumptions \ref{ass:approximability}, then the adaptive
scheme recovers these rates and the required computational work has optimal complexity up to logarithmic factors.
We have made an attempt to identify the dependencies of the involved constants on the problem parameters as explicitly as possible.

\begin{theorem}\label{thm:complexity}
Let $\alpha > 0$ and let $\constsvd, \constcrs$ be as in Theorem
\ref{lmm:combined_coarsening}.
Let the constants $\kappa_1,\kappa_2,\kappa_3$ in
Algorithm \ref{alg:tensor_opeq_solve} be chosen as
\begin{gather*}
  \kappa_1 = \bigl(1 + (1+\alpha)(\constsvd + \constcrs +
  \constsvd\constcrs)\bigr)^{-1}\,, \\
  \kappa_2 = (1+\alpha)\constsvd \kappa_1\,,\qquad 
  \kappa_3 = \constcrs(\constsvd + 1)(1+\alpha)\kappa_1 \,.
\end{gather*}
Let $\mathbf{A}\bu = \mathbf{f}$, where $\mathbf{A}$, $\bu$, $\mathbf{f}$
satisfy Assumptions \ref{ass:approximability}.
Then $\bu_\varepsilon$ produced by Algorithm \ref{alg:tensor_opeq_solve}
satisfies 
\begin{gather}
 \label{eq:complexity_rank} 
 \abs{\rank(\bu_\varepsilon)}_\infty 
   \leq \, \bigl( d_\mathbf{\bu}^{-1}
   \ln\bigl[(\theta\alpha)^{-1} \rho_{\ga_\bu}
   \,\norm{\bu}_{\AH{\ga_\mathbf{\bu}}}\,\varepsilon^{-1}\bigr] \bigr)^{ b_\mathbf{\bu}} \,,
   \\
 \label{eq:complexity_supp} \sum_{i=1}^m \#\supp_i(\bu_\varepsilon) \lesssim
     \Bigl(\sum_{i=1}^m \norm{ \pi^{(i)}(\bu)}_{\As} \Bigr)^{\frac{1}{s}}
           \varepsilon^{-\frac{1}{s}} \,,
\end{gather}
as well as
\begin{gather}
  \label{eq:complexity_ranknorm} \norm{\bu_\varepsilon}_{\AH{\ga_\mathbf{\bu}}}
  \lesssim
      \norm{\bu}_{\AH{\ga_\mathbf{\bu}}}    \,,   \\
  \label{eq:complexity_sparsitynorm} \sum_{i=1}^m \norm{
  \pi^{(i)}(\bu_\varepsilon)}_{\As} \lesssim
      \sum_{i=1}^m \norm{ \pi^{(i)}(\bu)}_{\As}  \,.
\end{gather}
The multiplicative constant in \eqref{eq:complexity_ranknorm} depends only on $\alpha$ and $m$,
those in \eqref{eq:complexity_supp} and \eqref{eq:complexity_sparsitynorm} depend only on
$\alpha$, $m$ and $s$.
For the number of required operations, we have the estimate
\begin{equation} 
\label{eq:complexity_totalops}
\ops(\bu_\varepsilon) \lesssim 
   \abs{\ln \varepsilon}^{J (3 + s^{-1}) b_\mathbf{A} + 2 b_\mathbf{f}}\,
   \Bigl( \sum_{i=1}^m \max\{\norm{\pi^{(i)}(\bu)}_\As,
        \norm{\pi^{(i)}(\mathbf{f})}_\As \} \Bigr)^\frac{1}{s} \,
     \varepsilon^{-\frac{1}{s}}
    \,, \end{equation}
with a multiplicative constant independent of $\varepsilon$ and
$\norm{ \pi^{(i)}(\bu)}_{\As}$, $\norm{ \pi^{(i)}(\mathbf{f})}_{\As}$, and
with an algebraic explicit dependence on $m$ and $C_\mathbf{A}$.
\end{theorem}

\begin{remark}
{Recalling the form of the growth sequence $\gamma_\bu(n)= e^{d_\bu n^{1/b_\bu}}$, the rank bound \eqref{eq:complexity_rank} 
can be reformulated in terms of $\gamma_\bu^{-1}\big(C\norm{\bu}_{\AH{\ga_\mathbf{\bu}}}/\varepsilon\big)$ which, in view
of Remark \ref{rem:howtoread}, means that up to a multiplicative constant, the ranks remain minimal. On account of Remark \ref{rem:howtoread2}, the same holds for the bound \eqref{eq:complexity_supp} on the sparsity of the factors.}
\end{remark}

\begin{remark}
The maximum number of inner iterations $J$ that arises in the complexity
estimate is defined in line \ref{alg:jchoice} of Algorithm
\ref{alg:tensor_opeq_solve}. This value depends on the freely chosen algorithm
parameters $\beta$ and $\theta$, on the constants $\omega$ and $\rho$ that
depend only on $\bA$, and on $\kappa_1$. Thus, $J$ depends on $m$: The choice of
$\kappa_1$ in Theorem \ref{thm:complexity} leads to $\kappa_1 \sim m^{-1}$,
and hence $J\sim \log m$. Note that since $\abs{\ln \varepsilon}^{c \ln m} =
m^{c \ln\abs{\ln \varepsilon}}$, this leads to an algebraic dependence of
the complexity estimate on $m$.
Furthermore, the precise dependence of the constant in \eqref{eq:complexity_totalops} on $m$ is also influenced by
the problem parameters from Assumption \ref{ass:approximability}, which may contain
additional implicit dependencies on $m$. In particular, as can be seen from the proof,
the constant has a linear dependence on $C_\bA^{J/s}$ if $C_\bA > 1$ (cf.\ Remark \ref{rmrk:operator_constants}).
\end{remark}

\begin{proof}[Theorem \ref{thm:complexity}]
By the choice of $\kappa_1$, $\kappa_2$, $\kappa_3$, we can apply Lemma
\ref{lmm:combined_coarsening} to each $\bu_i$ produced in
line \ref{alg:cddtwo_coarsen_line} of Algorithm \ref{alg:tensor_opeq_solve},
which yields the bounds \eqref{eq:complexity_rank}, \eqref{eq:complexity_supp},
\eqref{eq:complexity_ranknorm}, \eqref{eq:complexity_sparsitynorm} for
the values $\varepsilon = \theta^k \delta$, $k\in\N$.

It therefore remains to estimate the computational complexity of each inner
loop.
Note that $\recompress$ in line \ref{alg:tensor_solve_innerrecomp}
does not deteriorate the approximability of the intermediates $\bw_j$
as a consequence of Lemma \ref{lmm:contraction_monotonicity}.

Let $\varepsilon_k := \theta^k \delta$. We already know from Theorem \ref{lmm:combined_coarsening} that
\begin{align} 
 \label{eq:rankuiest}
\abs{\rank(\bu_k)}_\infty 
   &\leq \, \bigl( d_\mathbf{\bu}^{-1}
   \ln [\alpha^{-1} \rho_{\ga_\bu}
   \,\norm{\bu}_{\AH{\ga_\mathbf{\bu}}}\,\varepsilon_k^{-1}] \bigr)^{
   b_\mathbf{\bu}} 
    \lesssim \abs{\ln\varepsilon_k}^{b_\bu} \,,\\
       \label{eq:suppuiest}   
    \sum_{i=1}^m \#\supp_i(\bu_k) &\lesssim
    \Bigl(\sum_{i=1}^m \norm{ \pi^{(i)}(\bu)}_{\As}
 \Bigr)^{\frac{1}{s}} \varepsilon_k^{-\frac{1}{s}}  \,, \\
   \label{eq:spnormuiest}
  \sum_{i=1}^m \norm{\pi^{(i)}(\bu_k)}_{\As} &\lesssim
      \sum_{i=1}^m \norm{ \pi^{(i)}(\bu)}_{\As} \,,
\end{align}
where the multiplicative constants in the last two equations depend on $\alpha$, $m$, and $s$.
Similarly, we obtain \eqref{eq:complexity_ranknorm} from \eqref{eq:combinedcoarsen_rankest}.
Furthermore, by definition of the iteration,
\begin{equation*}
  \abs{\rank(\bw_{j+1})}_\infty \leq \bigl(
  d_\mathbf{A}^{-1} \ln (2 M_\mathbf{A} / \eta_j )\bigr)^{b_\mathbf{A}} 
    \abs{\rank(\bw_j)}_\infty  + \bigl(
  d_\mathbf{f}^{-1} \ln (2 \abs{\mathbf{f}}_{\AH{\ga_\mathbf{f}}} / \eta_j
  )\bigr)^{b_\mathbf{f}} \,.
\end{equation*}
Combining this with \eqref{eq:rankuiest} and using $b_\mathbf{f} > b_\bu$, we
obtain
\begin{equation*}
  \abs{\rank(\bw_{j})}_\infty \lesssim \abs{ \ln \varepsilon_{k}}^{j
  b_\mathbf{A} + b_\mathbf{f}}  \,.
\end{equation*}
The definition of the iterates also yields
\begin{multline*}
 \quad \#\supp_i(\bw_{j+1})
    \lesssim \#\supp_i(\bw_{j}) 
  \\  +   C_\mathbf{A}^{\frac{1}{s}} \bigl(
  d_\mathbf{A}^{-1} \ln (2 M_\mathbf{A} / \eta_j
  )\bigr)^{(1+s^{-1})b_\mathbf{A}} 
  \Bigl( \sum_{l=1}^m \norm{\pi^{(l)}(\bw_j)}_{\As} \Bigr)^{\frac{1}{s}}
   \eta_j^{-\frac{1}{s}}\\
    + \norm{\pi^{(i)}(\mathbf{f})}_{\As}^{\frac{1}{s}}
   \eta_j^{-\frac{1}{s}} \,,\quad
\end{multline*}
and by \eqref{eq:rmrk_apply_spnorm},
\begin{multline*}  \norm{\pi^{(i)}(\bw_j)}_\As 
   \lesssim 
   \norm{\pi^{(i)}(\bw_{j-1})}_\As\\
    +  \omega C_\mathbf{A} \bigl( 
    d_\mathbf{A}^{-1} \ln(2M_\mathbf{A}/\eta_{j-1}) \bigr)^{(1+s)b_\mathbf{A}}
     \norm{\pi^{(i)}(\bw_{j-1})}_\As
     + \omega \norm{\pi^{(i)}(\mathbf{f})}_\As  \,.
\end{multline*}
Using these estimates recursively together with \eqref{eq:spnormuiest},
\eqref{eq:suppuiest}, we obtain
\begin{equation*}
  \norm{\pi^{(i)}(\bw_j)}_\As 
   \lesssim \abs{\ln \varepsilon_k}^{j\,(1+s)\, b_\mathbf{A}}     
    \max\bigl\{ \norm{\pi^{(i)}(\bu)}_\As,
        \norm{\pi^{(i)}(\mathbf{f})}_\As \bigr\}           
\end{equation*}
and
\begin{equation*}
  \sum_{i=1}^m \#\supp_i(\bw_{j})
    \lesssim \abs{\ln \varepsilon_k}^{j\,(1+s^{-1})\, b_\mathbf{A}}
        \Bigl( \sum_{i=1}^m \max\{\norm{\pi^{(i)}(\bu)}_\As,
        \norm{\pi^{(i)}(\mathbf{f})}_\As \} \Bigr)^\frac{1}{s} 
        \,\varepsilon_k^{-\frac{1}{s}} \,.
\end{equation*}
The total number of operations for the calls of $\apply$ in an
inner loop according to \eqref{eq:apply_complexity_summary} is dominated
by that for the calls of $\recompress$, which can be bounded up to a constant
by
$$  m \abs{\rank(\bw_J)}_\infty^4 
+ \abs{\rank(\bw_J)}_\infty^2 \sum_{i=1}^m \#\supp_i(\bw_{J})\,.  $$
We thus arrive at \eqref{eq:complexity_totalops}.
\end{proof}

\begin{remark}\label{rmrk:sobolev}
The above results apply directly to problems posed on separable tensor product
Hilbert spaces, for which tensor product Riesz bases are available.
{Note that this is \emph{not} the case for standard Sobolev spaces
$\spH{s}(\Omega^d)$, since in this case the norm induced by the scalar product is not a cross norm.
However, for tensor product domains $\Omega^d$, these spaces can be
represented as \emph{intersections} of $d$ tensor product spaces with induced norms.}

As mentioned in the introduction, from a sufficiently regular tensor product
wavelet basis $ \{ \Psi_\nu := \psi_{\nu_1}\otimes
\cdots\otimes\psi_{\nu_d}\}_{\nu\in\nabla^d}$ of $\spL{2}(\Omega^d)$, 
we can obtain a Riesz basis of $\spH{s}(\Omega^d)$ by a
level-dependent rescaling of basis functions, e.g.,
$$  \bigl\{  2^{-s \max_i\abs{\nu_i}} \Psi_\nu \bigr\}_{\nu\in\nabla^d}  \,. $$
To again arrive at a problem on $\spl{2}$, we now rewrite the original
operator equation $Au=f$, with $A\colon \spH{s}(\Omega^d)\to
(\spH{s}(\Omega^d))'$, in the form 
\begin{multline*}  \sum_{\mu\in\nabla^d} \bigl( 2^{-s(\max_i\abs{\nu_i} +
\max_i\abs{\mu_i})} \langle A \Psi_\mu , \Psi_\nu \rangle \bigr)
\bigl(2^{s\max_i\abs{\mu_i}}\langle u, \Psi_\mu\rangle\bigr)\\
 = 2^{-s(\max_i\abs{\nu_i}} \langle f, \Psi_\nu\rangle \,,
 \quad  \nu\in\nabla^d\,.
\end{multline*}
We thus obtain a well-posed problem on $\spl{2}(\nabla^d)$ for the
rescaled coefficient sequence $\bu = 2^{s\max_i\abs{\mu_i}}\langle u,
\Psi_\mu\rangle$ and the infinite matrix $\bA = 2^{-s(\max_i\abs{\nu_i}
 + \max_i\abs{\mu_i})} \langle A \Psi_\mu , \Psi_\nu \rangle$.
 
This diagonal rescaling, which in the case of finite-dimensional Galerkin
approximations corresponds to a preconditioning of $A$, leads to
additional problems in our context: the sequence $(2^{-s \max_i\abs{\nu_i}})_{\nu\in\nabla^d}$
(as well as possible equivalent alternatives) has infinite rank on the full
index set $\nabla^d$. {Hence, the application of $\mathbf{A}$ has to involve
an approximation by low rank operators as discussed in \S\ref{sec:operator_lowrank_approx}.
Strategies for handling this issue are discussed in more detail in \cite{Bachmayr:12}.
The complexity analysis of iterative schemes when $\bA$ involves
such a rescaling will be treated in a separate paper.}
\end{remark}

\section{Numerical Experiments}\label{sect:numexp}

We choose our example to illustrate the results of the previous section numerically
according to several criteria.
In order to arrive at a valid comparison between different dimensions, we
choose a problem on $\spL{2}([0,1]^d)$ that has similar properties for different 
values of $d$. The problem has a discontinuous right hand side and solution,
which means that reasonable convergence rates can be achieved only by adaptive approximation.
It is also still sufficiently simple such that all constants
used in Algorithm \ref{alg:tensor_opeq_solve} can be chosen rigorously according to 
the requirements of the convergence analysis.

We set $\Omega:=[0,1]^d$ and use tensor order $m=d$.
As an orthonormal wavelet basis $\{ \psi_\nu \}_{\nu\in\nabla}$ of
$\spL{2}([0,1])$, we use Alpert multiwavelets \cite{Alpert:91} of polynomial order $p\in\N$.
Let 
$$(T v )(t):= \int_0^t v \dif s \,, $$
then $T$ is a compact operator on
$\spL{2}([0,1])$ with $\norm{T} = 2 / \pi$. The infinite matrix representation
$\bigl( \langle T \psi_\mu,\psi_\nu
\rangle \bigr)_{\nu,\mu\in\nabla}$ is $s^*$-compressible for any $s^* >0$.

For $f\in\spL{2}(\Omega)$, we consider the
integral equation
\begin{equation}\label{eq:volterraeq_highdim}
  \Bigl( \id  - \omega_d \bigotimes_{i=1}^d T \Bigr) u = f 
\end{equation}
with $\omega_d = \frac{1}{2}(\frac{\pi}{2})^d$.
Note that for $B := \omega_d \bigotimes_{i=1}^d T$ and $A := \id -B$ we have
$\norm{B} = \frac{1}{2}$, and therefore
\begin{equation*}
  A^{-1} = (\id - B)^{-1} = \sum_{k=0}^\infty B^k = \sum_{k=0}^\infty \omega_d^k
  \bigotimes_{i=1}^d T^k \,.
\end{equation*}
Furthermore, $A := \id - B$ is a nonsymmetric, $\spL{2}$-elliptic operator with
$\langle A v, v\rangle \geq \frac{1}{2} \norm{v}^2_{\spL{2}(\Omega)}$ as well
as $\norm{A} \leq \frac{3}{2}$. Since $\mathbf{A}$ is the representation with
respect to an orthonormal basis, we obtain $\lambda_\bA = \frac{1}{2}$ and
$\Lambda_\bA = \frac{3}{2}$.
Due to the special structure of the operator, choosing the iteration parameter
$\omega$ as $\omega := 1$, we have $\norm{\id - \omega \bA} \leq \frac{1}{2}
=: \rho$.
We choose the right hand side as
\begin{equation}\label{eq:testrhs_full}   f = (1-\tau) {\sum_{k=0}^\infty \tau^k \bigotimes_{i=1}^d  f_k } \,,\quad 
  f_k(x) := \sqrt{2 \pi}\,\chi_{[0,1/\pi]} \cos(2 \pi^2 (k+1) x)\,,  
\end{equation}
where $\tau \in(0,1)$.
This gives $\norm{f_k}_{\spL{2}([0,1])} =1$ and $\norm{f}_{\spL{2}(\Omega)} =\norm{\mathbf{f}} = 1$, and
$\pi^{(i)}(\mathbf{f}) \in \As$ for any $s>0$. 
The functions $f_k$ have jump discontinuities at $\pi^{-1}$, which need to be resolved adaptively
in order to maintain the optimal approximation rate for the given wavelet basis.

From the expansion for $(\id - B)^{-1}$, we already know that $\pi^{(i)}(\bu)\in\As$ for
any $s < p$, for $i=1,\ldots,m$. We also have the explicit representation
\begin{equation*}
  u = (1-\tau) \sum_{k, n=0}^\infty \tau^k \omega_d^n \bigotimes_{i=1}^d T^n f_k \,.
\end{equation*}
For the choice of $f_k$ under consideration, evaluating $\omega_d^n \bigotimes_{i=1}^d T^n f_k$, we obtain $u\to f$ as $d\to \infty$; 
that is, the mode singular values of the solution approach exponential decay with rate $\tau$ 
for growing $d$. Since $\norm{u-f}_{\spL{2}}$ is small for any $d>3$,
$\bu$ has similar low-rank approximability for all relevant $d$.

Hence for our particular choice of $f$, the action of $A^{-1}$ is close to the identity. It should be
emphasized, however, that this only simplifies the interpretation of the results, but does not
simplify the problem from a computational point of view, since our algorithm does not make use of this particularity.
We have also chosen a problem that is completely symmetric with respect to all variables to simplify the tests and the comparison between values of $d$, but do not make computational use of this symmetry.

For the further constants arising in the iteration, we choose $\theta :=
\frac{1}{2}$ and $\beta :=1$. For the hierarchical Tucker format, we have
$\constsvd = \sqrt{2m-3}$ and $\constcrs = \sqrt{m}$, and fix the derived
constants $\kappa_1,\kappa_2,\kappa_3$ as in Theorem \ref{thm:complexity} 
by taking $\alpha :=1$. Furthermore, we have $\delta =
\lambda_\bA^{-1}\norm{\mathbf{f}} = 2$.

\begin{remark}
Since many steps of the algorithm -- including the comparably expensive approximate application of
lower-dimensional operators to tensor factors and QR factorizations
of mode frames -- can be done independently for each mode,
an effective parallelization of our adaptive scheme is quite easy to achieve.
\end{remark}

In all following examples, we use piecewise cubic wavelets.
The implementation was done in C++ using standard LAPACK routines for linear algebra operations.
Iterations are stopped as soon as a required wavelet index cannot be represented as a signed 64-bit integer.

We make some simplifications in counting the number of required operations:
For each matrix-matrix product, $QR$ factorization, and SVD, we use the standard estimates
for the required number of multiplications (see, e.g., \cite{Hackbusch:12}); for the approximation
of $\bA$ and $\bbf$, we count one operation per multiplication with a matrix entry and 
per generated right hand side entry, respectively (note that we thus make the simplifying assumption
that all required wavelet coefficients can be evaluated using $\Ocal(1)$ operations, which could
in principle be realized in the present example, but is not strictly satisfied in our current implementation).
We thus neglect some minor contributions that do not play any asymptotic role, such as the number of operations required for adding two tensor representations, and the sorting of tensor contraction values for
$\coarsen$, which here is done by a standard library call for simplicity.

\subsection{Results with Right Hand Side of Rank 1}

For comparison, we first consider a simplified version of the right hand side reduced to the first
summand, that is,
$$ f =   \bigotimes_{i=1}^d \sqrt{2 \pi}\,\chi_{[0,1/\pi]} \cos(2 \pi^2 \,\cdot )\,. $$
In high dimensions, the solution $u$ coincides with $f$ up to very small correction terms.

\begin{figure}[ht!]\hspace{-9pt}
\begin{tabular}{ccc}
\includegraphics[width=4.6cm]{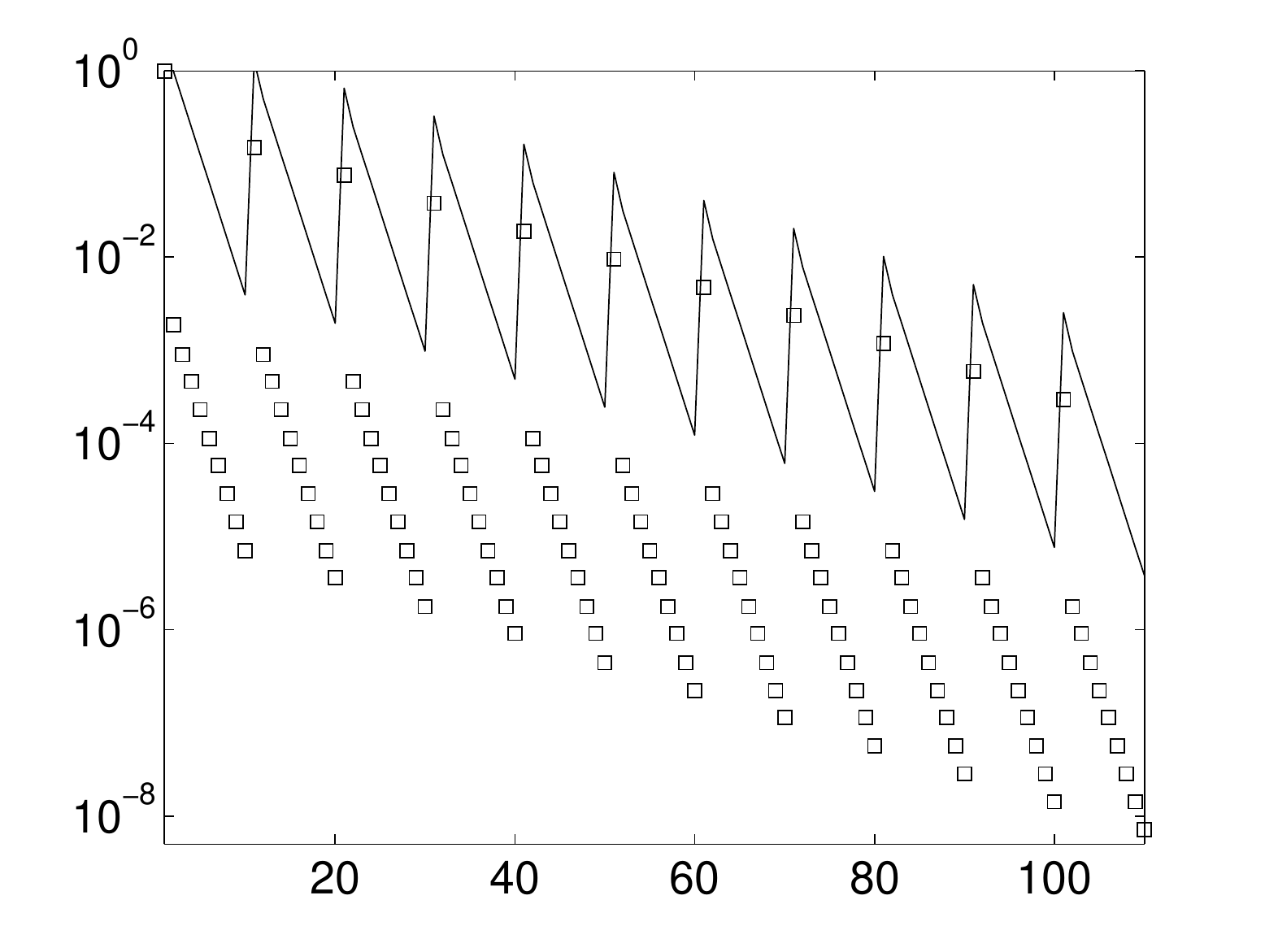} &
\includegraphics[width=4.6cm]{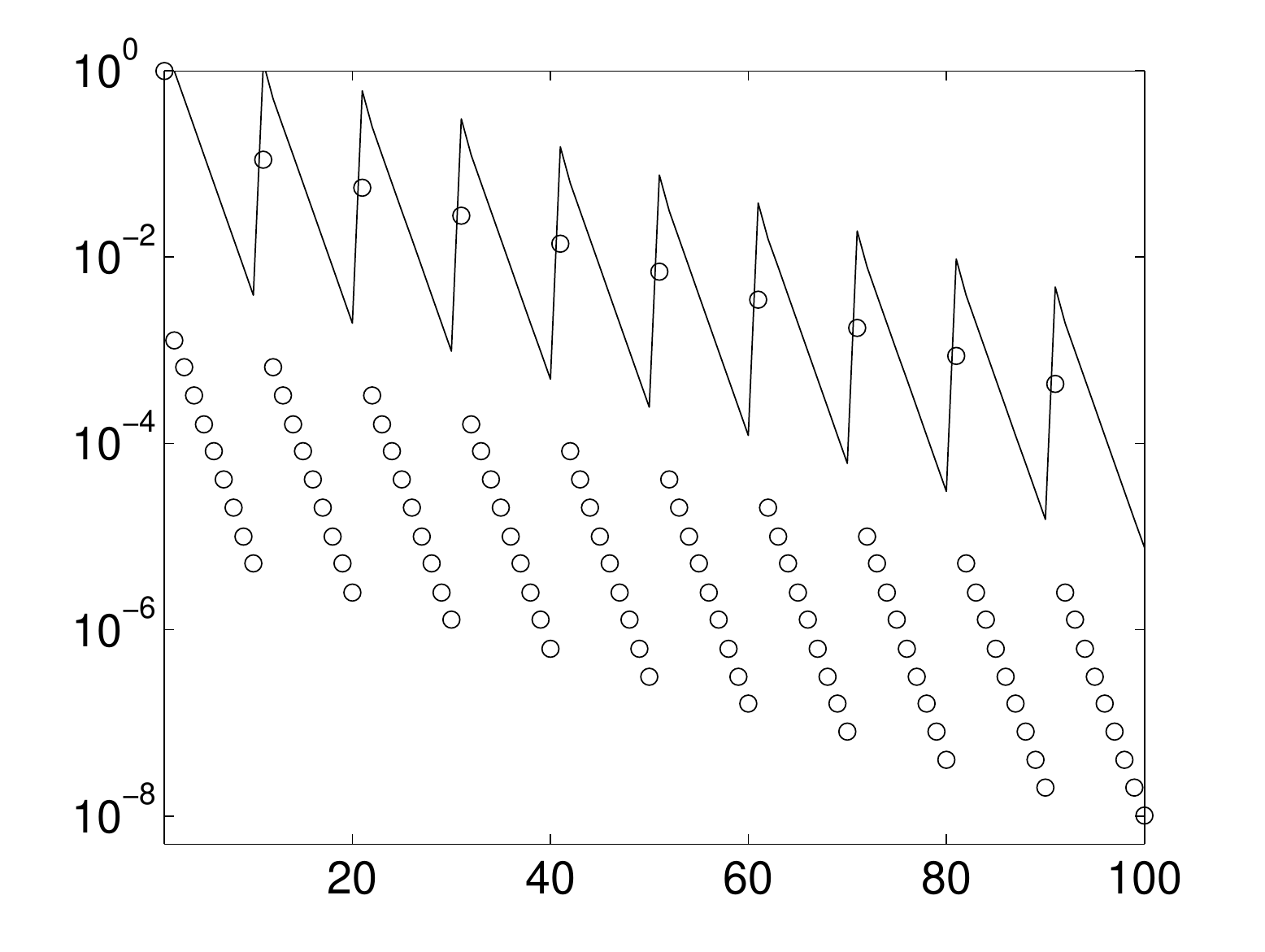} &
\includegraphics[width=4.6cm]{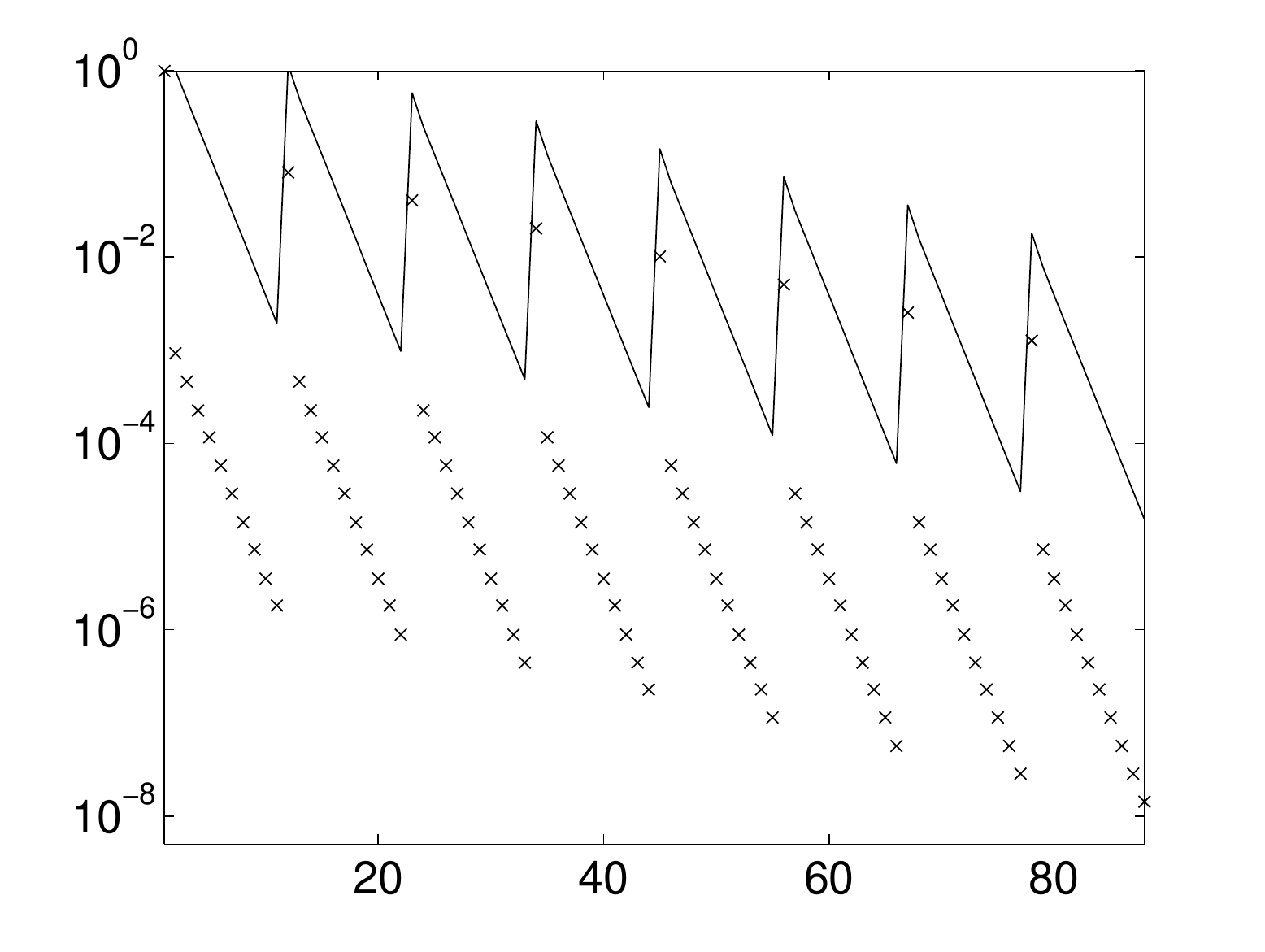} \\[-6pt]
\scriptsize $d=32$ & \scriptsize $d=64$ & \scriptsize $d=128$
\end{tabular}\vspace{-2pt}
\caption{Computed approximate residual norms (markers) and corresponding solution error estimates (solid lines), for $f$ of rank one, {in dependence on the total number of inner iterations (horizontal axis).}}
\label{fig:resrank1}
\end{figure}
The evolution of the computed approximate residual norms and the corresponding estimates
for the $\spL{2}$-deviation from the solution of the infinite-dimensional problem is 
shown in Figure \ref{fig:resrank1}. Here one can clearly observe the effect of the coarsening steps
after a certain number of inner iterations.
Apart from the expected increase in the number $J$ of such inner iterations with dimension,
the iteration shows quite similar behaviour for different $d$.
In particular, in each case the resulting iterates $\bw_j$ in Algorithm \ref{alg:tensor_opeq_solve} 
have rank 1, the residuals $\mathbf{r}_j$ have ranks at most 3,
thus the maximum rank arising in the iteration is 4.

Note that the iteration is stopped a few steps earlier with increasing dimension because slightly
stricter error tolerances are applied in the approximation of operator and right hand side.
This means that the technical limit for the maximum possible wavelet level is reached earlier.

\begin{figure}[ht!]\hspace{-9pt}
\centering
\includegraphics[width=8cm]{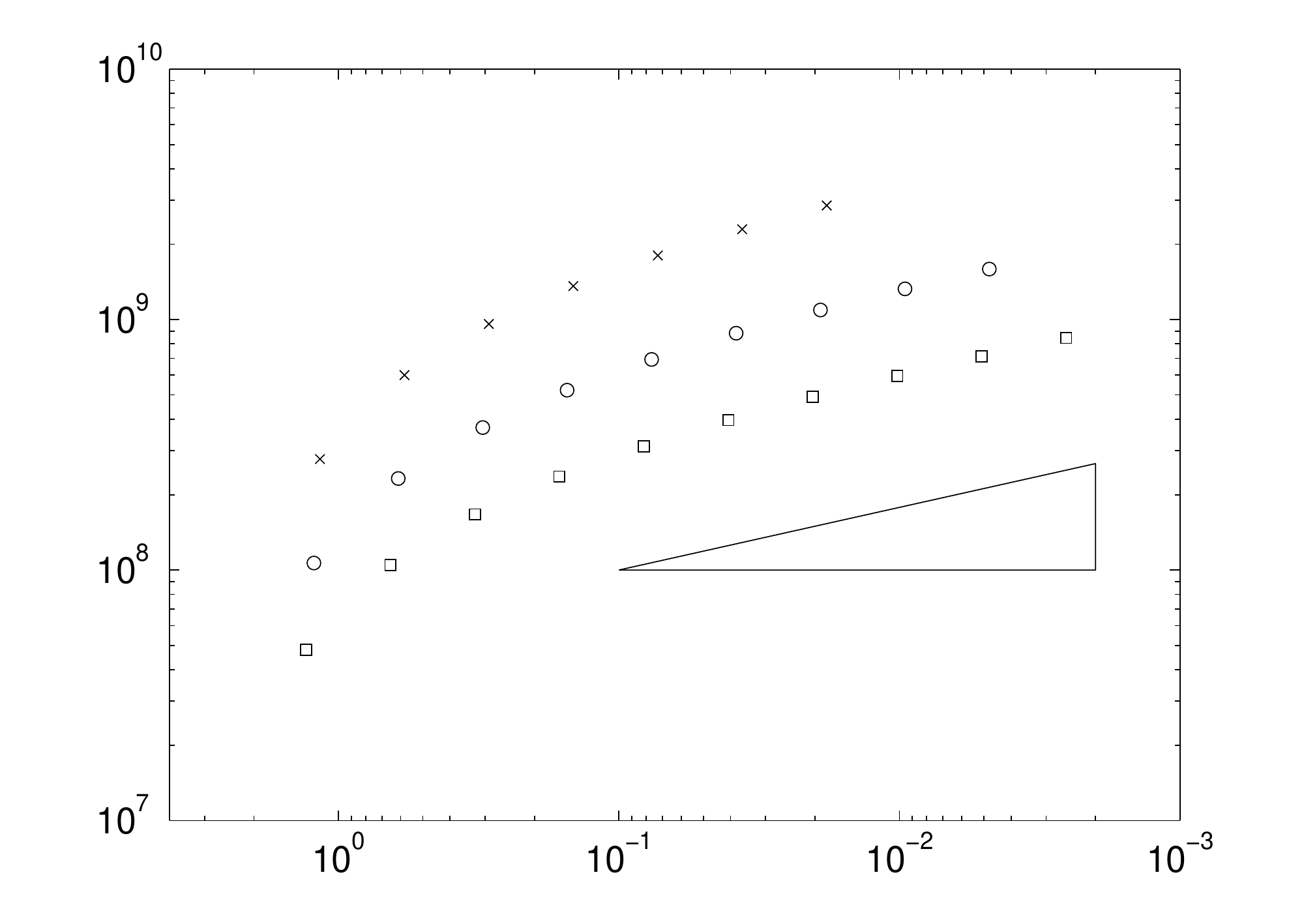}
\caption{Total operation count {($\square$ $d=32$, $\circ$ $d=64$, $\times$ $d=128$)} at the end of each inner iteration in dependence on the estimated error {(horizontal axis)}, for $f$ of rank one. The triangle shows a slope of $\frac{1}{4}$.}
\label{fig:opsrank1}
\end{figure}
We see that the number of operations, shown in Figure \ref{fig:opsrank1}, increases at a rate
close to the approximation order $4$ of our wavelet basis.
What is most remarkable here, however, is the very mild -- almost linear -- dependence of
the total complexity on the dimension: a doubling of dimension leads to only slightly more than
twice the number of operations.

\subsection{Results with Right Hand Side of Unbounded Rank}

We now use the full right hand side $f$ as in \eqref{eq:testrhs_full}, which leads to a solution 
with unbounded rank, and approximately the same exponential decay of singular values as $f$.

\begin{figure}[ht!]\hspace{-9pt}
\begin{tabular}{ccc}
\includegraphics[width=4.6cm]{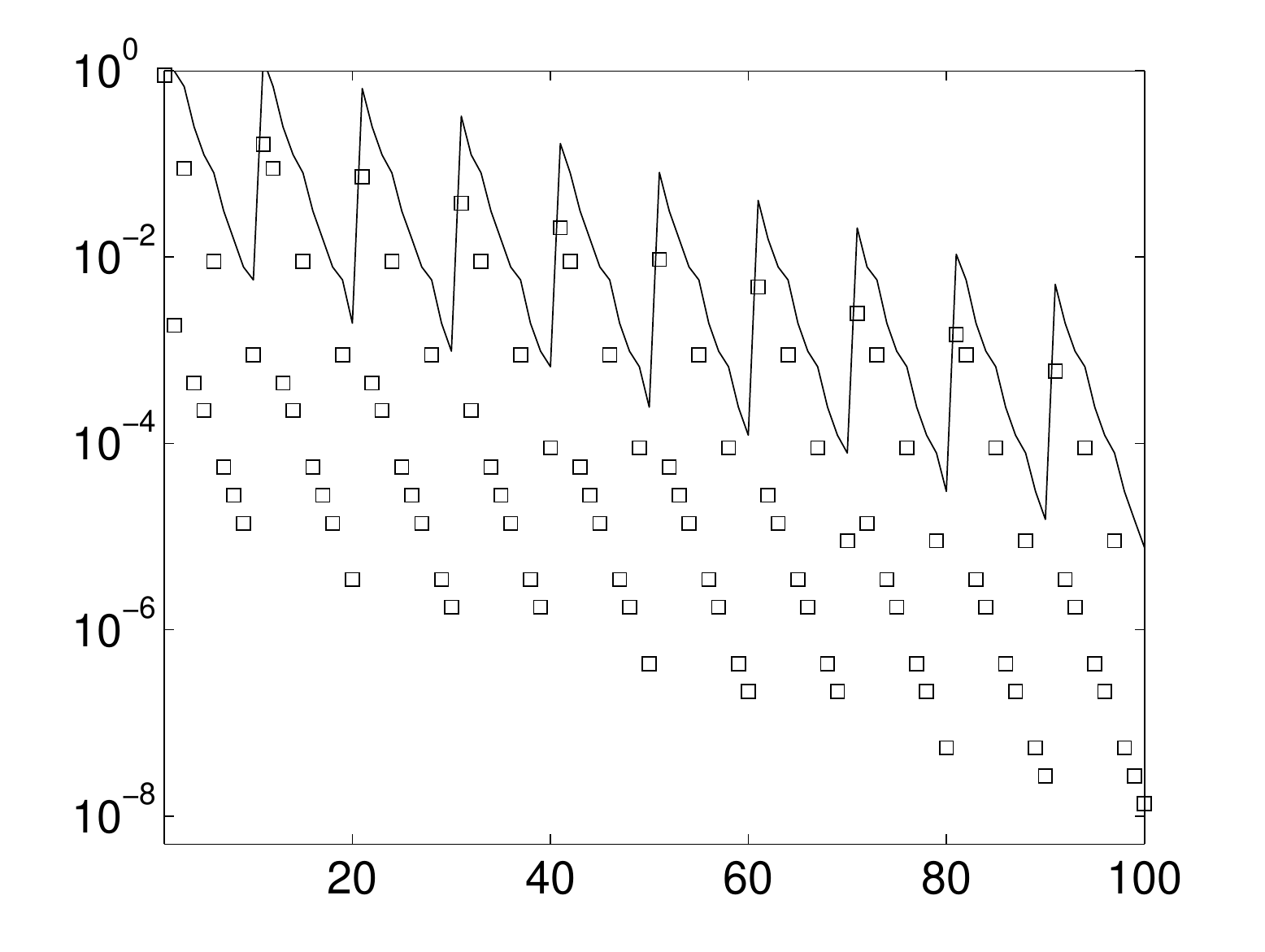} &
\includegraphics[width=4.6cm]{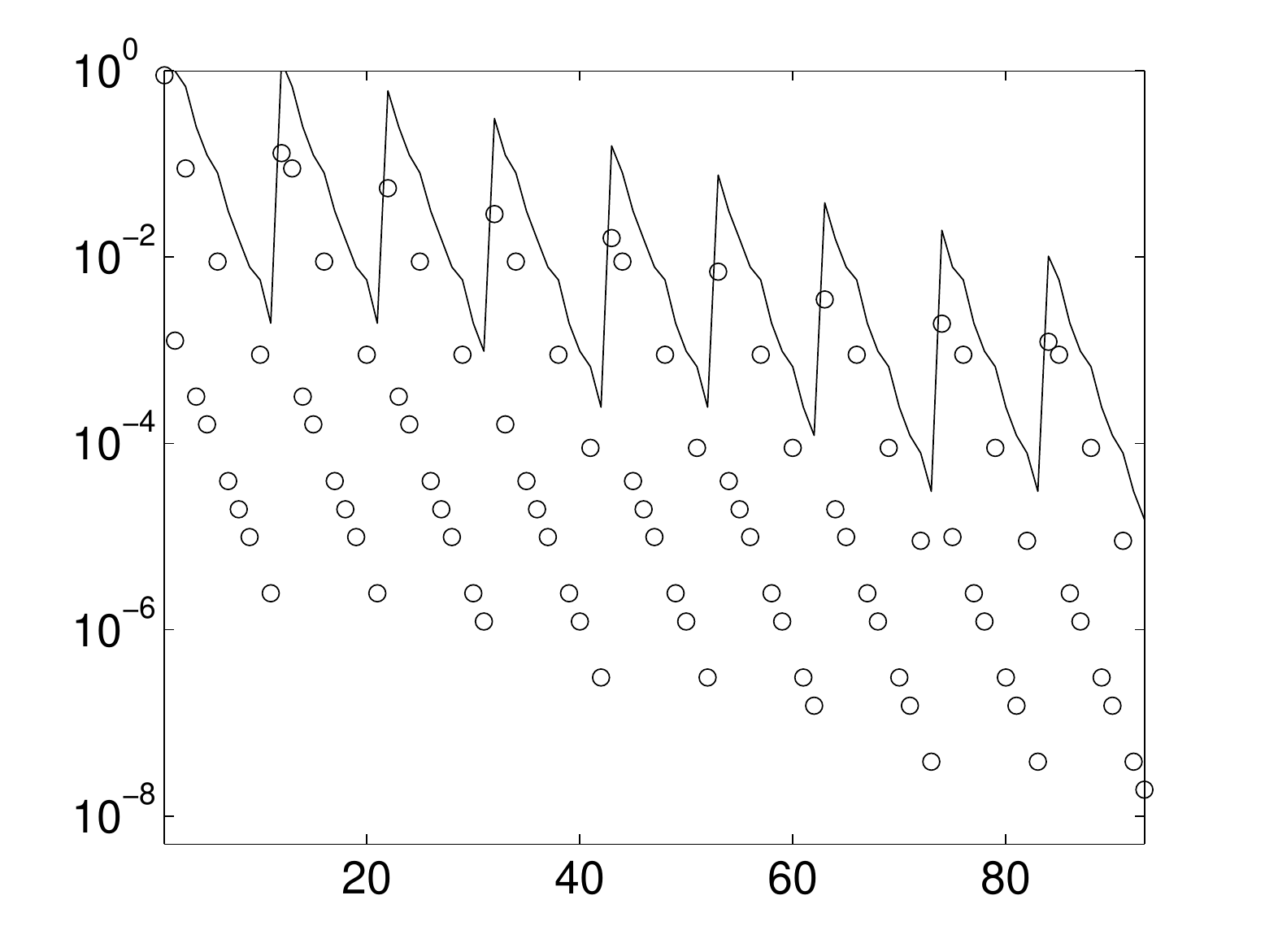} &
\includegraphics[width=4.6cm]{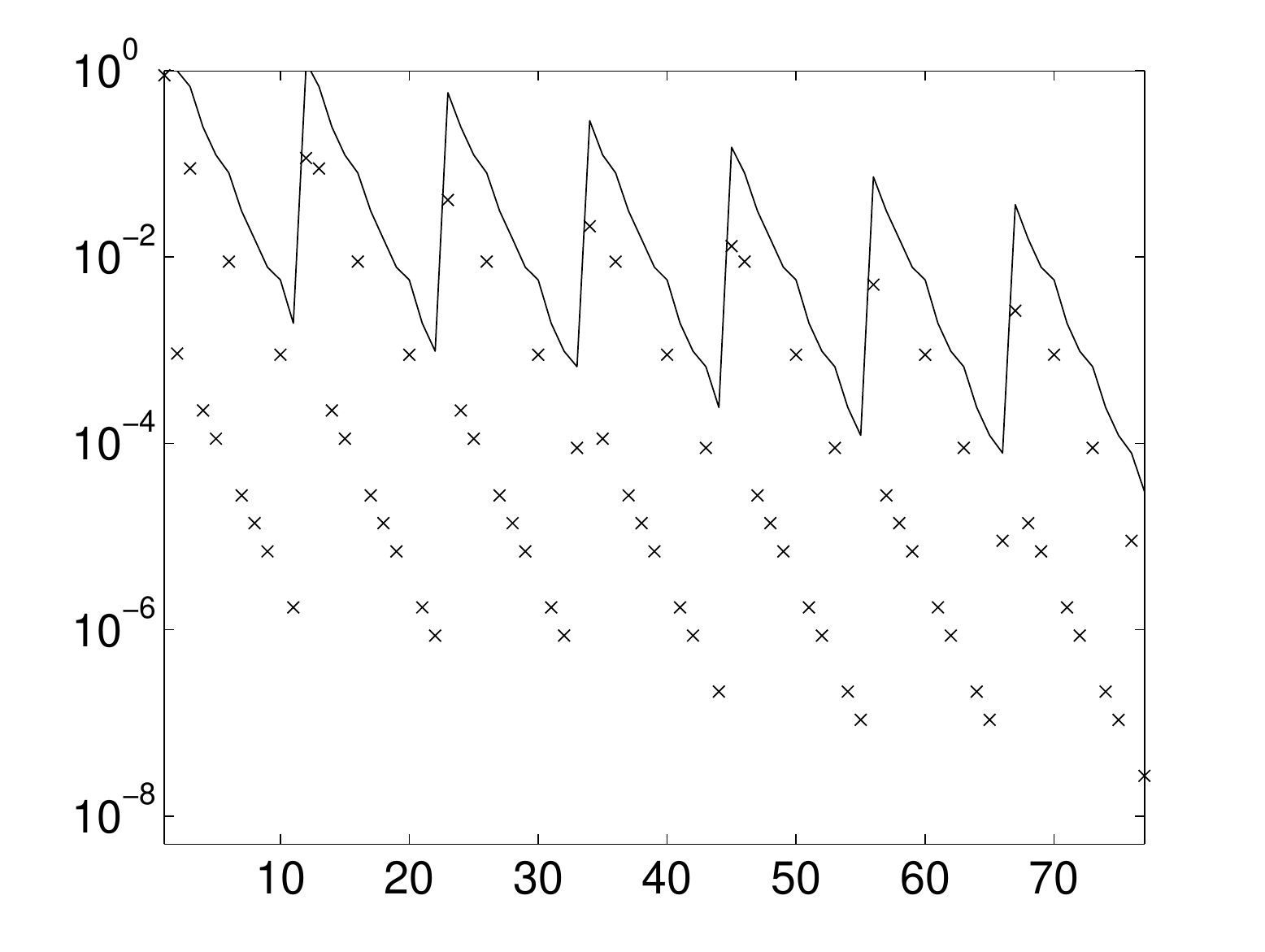} \\[-6pt]
\scriptsize $d=32$ & \scriptsize $d=64$ & \scriptsize $d=128$
\end{tabular}\vspace{-2pt}
\caption{Computed approximate residual norms (markers) and corresponding solution error estimates (solid lines),
for $f$ of unbounded rank, {in dependence on the total number of inner iterations (horizontal axis).}}
\label{fig:resrankx}
\end{figure}
As shown in Figure \ref{fig:resrankx}, the computed residual estimates and the corresponding estimates for the solution error behave quite similarly to the previous example. 
In the present case, the computed residual norms show a less regular pattern, which is mostly due to the adjustment of approximation ranks for the right hand side.

\begin{figure}[ht!]\hspace{-9pt}
\begin{tabular}{ccc}
\includegraphics[width=4.6cm]{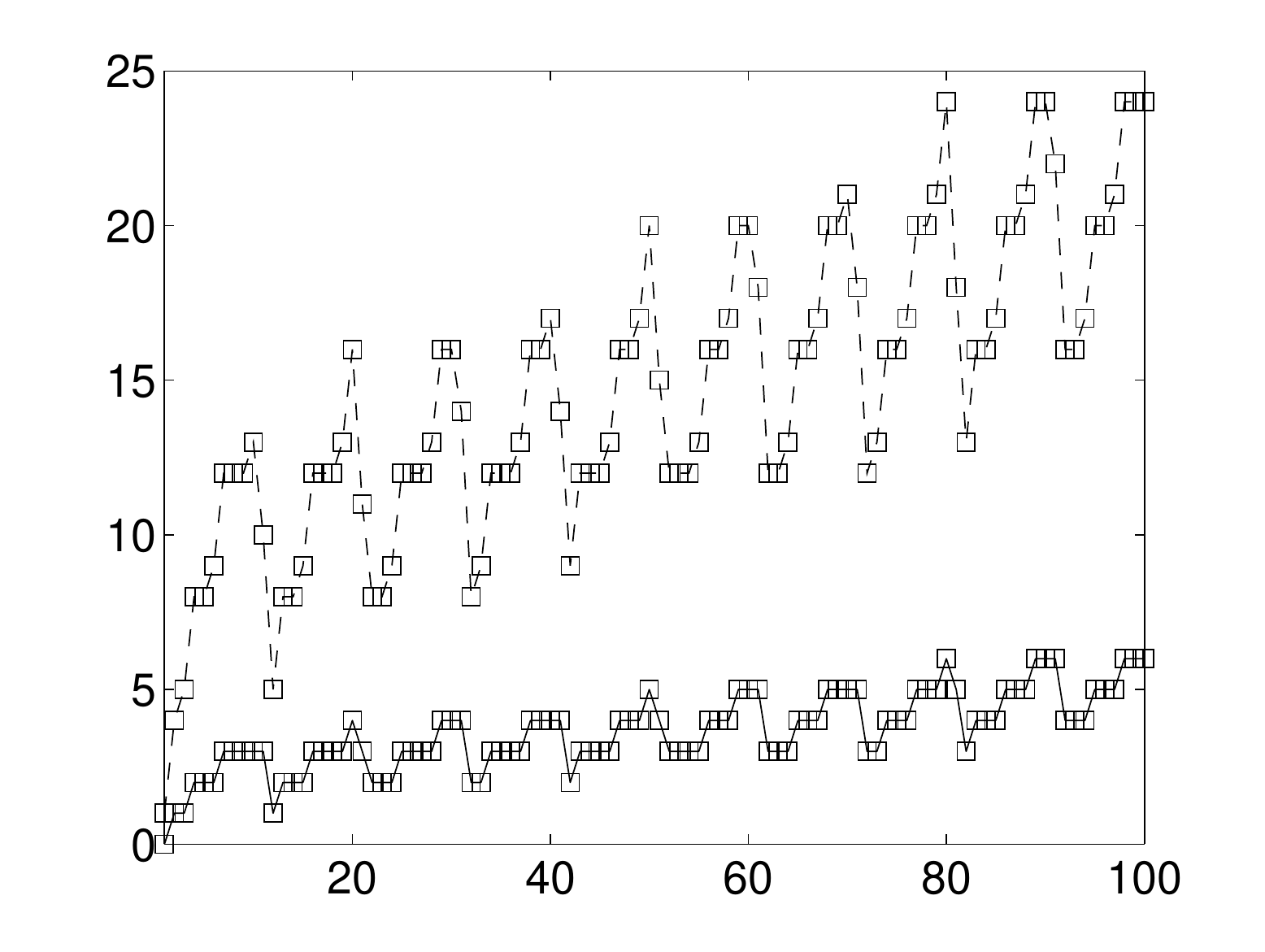} &
\includegraphics[width=4.6cm]{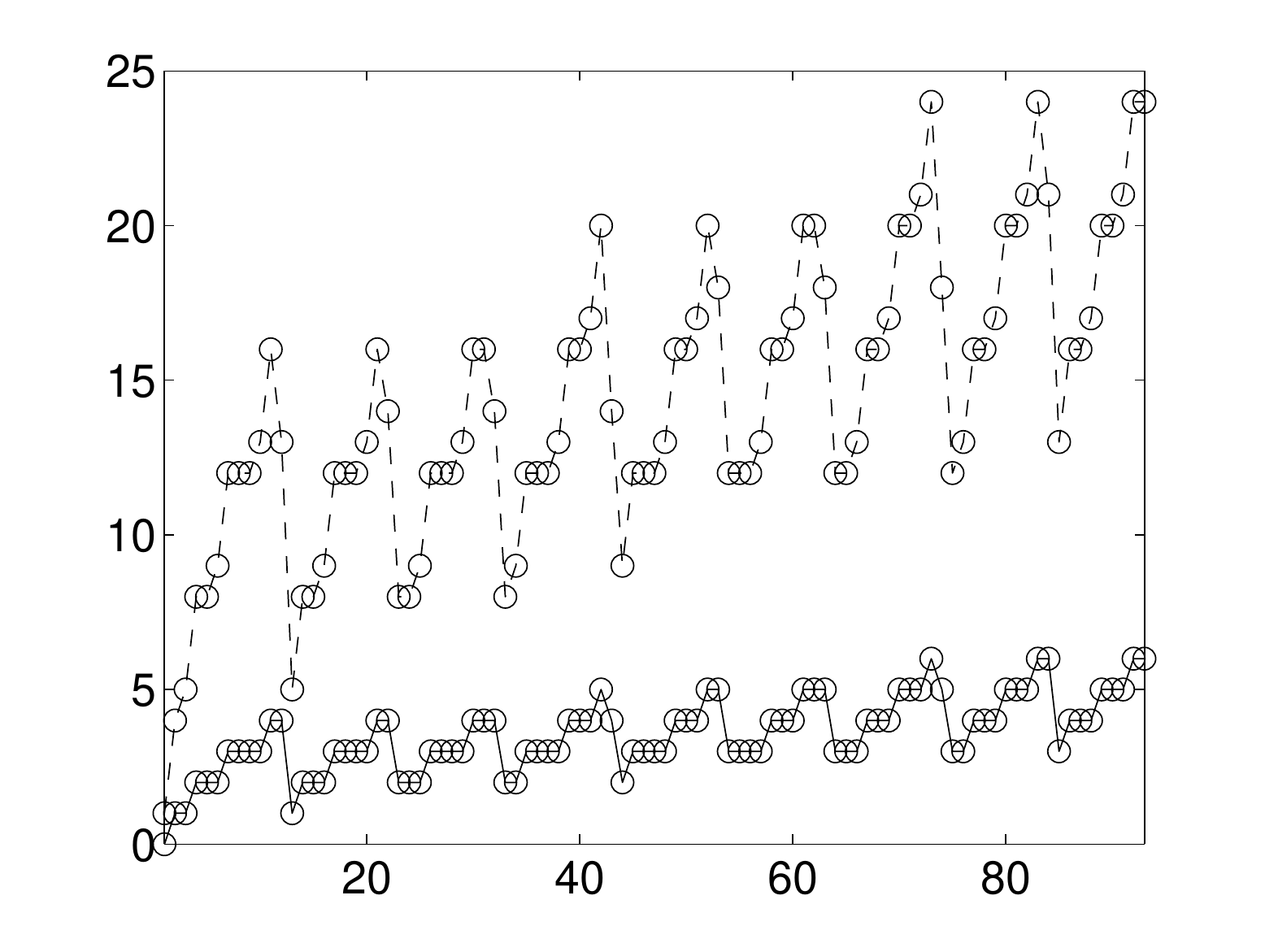} &
\includegraphics[width=4.6cm]{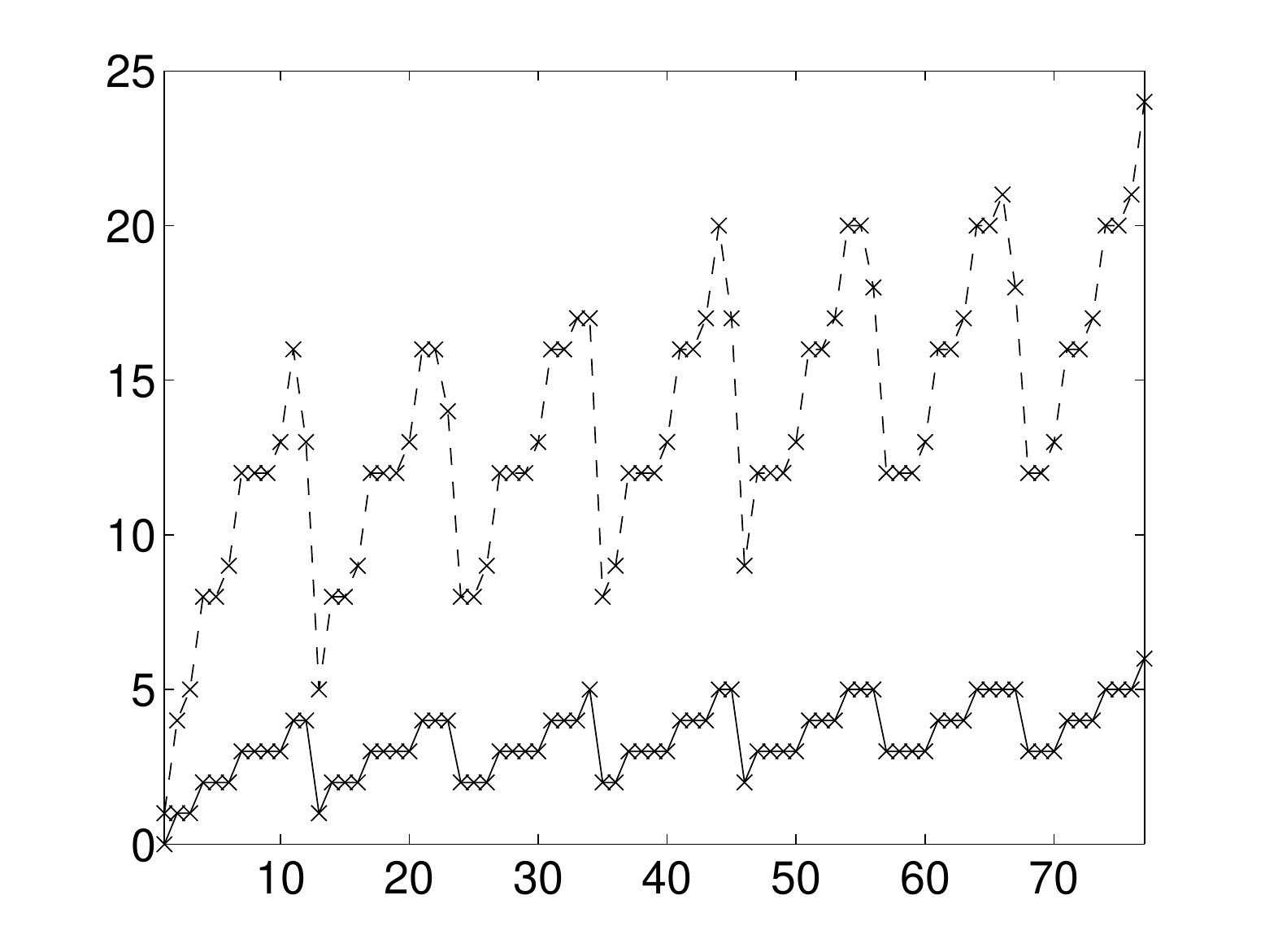} \\[-6pt]
\scriptsize $d=32$ & \scriptsize $d=64$ & \scriptsize $d=128$
\end{tabular}\vspace{-2pt}
\caption{Maximum ranks of iterates $\bw_j$ (solid lines) {and maximum ranks} of all intermediates arising in the inner
iteration steps (dashed lines),
for $f$ of unbounded rank, {in dependence on the total number of inner iterations  (horizontal axis).}}
\label{fig:ranksx}
\end{figure}
The ranks of the produced iterates $\bw_j$, as well as those of the intermediate quantities arising
in the iteration (see line \ref{alg:tensor_solve_innerrecomp}
of Algorithm \ref{alg:tensor_opeq_solve} prior to the recompression operation), shows a steady but controlled increase during the iteration, as shown in Figure \ref{fig:ranksx}.

\begin{figure}[ht!]
\centering
\includegraphics[width=8cm]{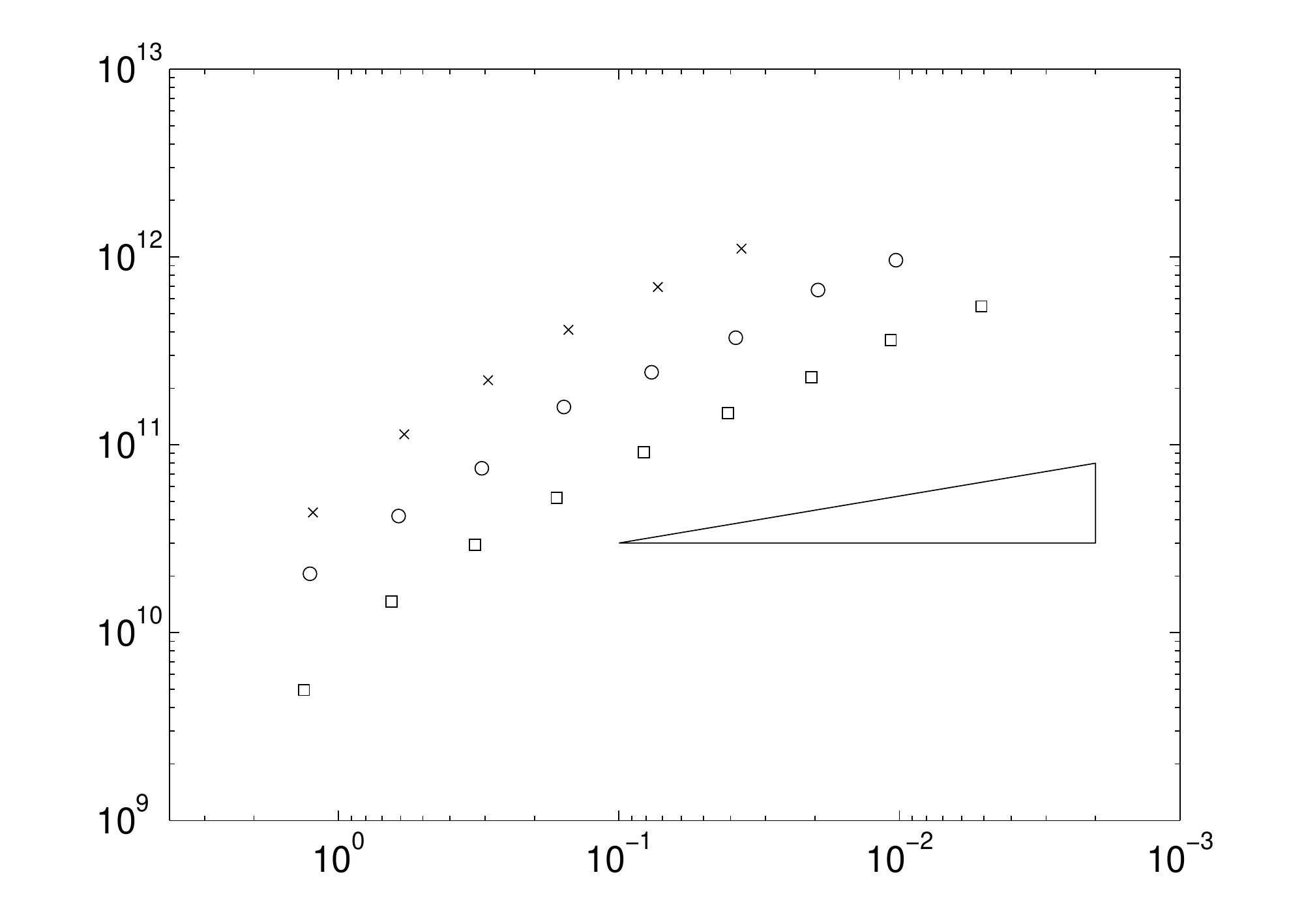}
\caption{Operation count {($\square$ $d=32$, $\circ$ $d=64$, $\times$ $d=128$)} at the end of each inner iteration in dependence on the estimated error {(horizontal axis)}, for $f$ of unbounded rank. The triangle shows a slope of $\frac{1}{4}$.}
\label{fig:opsrankx}
\end{figure}
Note that in this case, the number of operations, shown in Figure \ref{fig:opsrankx}, increases visibly faster than the 
limiting rate corresponding to the approximation order of the lower-dimensional multiresolution spaces. Due to the higher tensor ranks involved, this is to be expected in view of our complexity estimates.
The increase of complexity with the problem dimension, however, still remains very moderate.

\section{Conclusion and Outlook}

The presented theory and examples indicate that the schemes developed in this work can be applied to very high-dimensional problems, with a rigorous foundation for the type of elliptic operator equations considered here.
The results can be extended to more general operator equations, as long as the variational formulation,
in combination with a suitable basis, induces a well-conditioned isomorphism on $\spl{2}$.
However, when the operator represents an isomorphism between spaces that
are not simple tensor products, such as Sobolev spaces and their duals, additional concepts are required,
which will be developed in a subsequent publication.
\vspace{12pt}

\textbf{Acknowledgements.}
This work was funded in part by  the Excellence Initiative of the German Federal and State Governments,
  DFG Grant GSC 111 (Graduate School AICES), the DFG Special Priority Program 1324, and NSF Grant \#1222390.

\bibliography{adaptlowrank}

\end{document}